\RequirePackage{fix-cm}
\documentclass[dvipsnames,smallextended]{svjour3}
\smartqed

\usepackage{mathtools}
\usepackage{autonum}
\usepackage{amssymb}
\usepackage[utf8]{inputenc}
\PassOptionsToPackage{final}{graphicx}
\usepackage{pgfplots}
\usepackage{algorithmic}
\usepackage{algorithm2e}
\usepackage{ifdraft}
\usepackage[colorinlistoftodos]{todonotes}
\usepackage[caption=true]{subfig} 
\usepackage{url}
\usepackage{pdfpages}
	
\title{A multilevel Monte Carlo method for asymptotic-preserving particle schemes in the diffusive limit\thanks{A preliminary version of this
paper, cited as~\cite{Loevbak2020}, was published in \textit{Monte Carlo and Quasi-Monte Carlo Methods 2018}.}}
\titlerunning{An MLMC method for AP-particle schemes in the diffusive limit}
\author{Emil~L{\o}vbak \and Giovanni~Samaey \and Stefan~Vandewalle}

\institute{E.~L{\o}vbak \at
              NUMA Section, Department of Computer Science, KU Leuven, Belgium \\
              \email{emil.loevbak@cs.kuleuven.be} 
              ORCID: 0000-0003-1520-4922          
           \and
           G.~Samaey \at
              NUMA Section, Department of Computer Science, KU Leuven, Belgium \\
              \email{giovanni.samaey@cs.kuleuven.be} 
              ORCID: 0000-0001-8433-4523
           \and
           S.~Vandewalle \at
              NUMA Section, Department of Computer Science, KU Leuven, Belgium \\
              \email{stefan.vandewalle@cs.kuleuven.be} 
              ORCID: 0000-0002-8988-2374
}

\date{Received: date / Accepted: date}

\DeclareMathOperator{\E}{\mathbb{E}}
\DeclareMathOperator{\V}{\mathbb{V}}
\DeclareMathOperator{\Cov}{\mathrm{Cov}}

\renewcommand{\epsilon}{\varepsilon}
\newcommand{\R}{\mathbb{R}}
\newcommand{\CV}{\tilde{v}}
\newcommand{\VD}{\bar{V}}
\newcommand{\VDR}{\bar{V}}
\newcommand{\PT}{u}
\newcommand{\B}{\mathcal{B}}
\newcommand{\D}{\mathcal{D}}
\newcommand{\U}{T}
\newcommand{\dm}{\mathrel{\Delta m}}
\newcommand{\dns}{\mathrel{\Delta n}}
\newcommand{\dn}{{n - n^\prime}}
\newcommand{\pcf}{p_{c,\Delta t_\ell}}
\newcommand{\pcr}{p_{c,\Delta t_{\ell-1}}}
\newcommand{\pncf}{p_{nc,\Delta t_\ell}}
\newcommand{\pncr}{p_{nc,\Delta t_{\ell-1}}}

\newcommand{\mse}{E}
\newcommand{\covone}{\sigma_1}
\newcommand{\covtwo}{\sigma_2}
\newcommand{\dx}{\text{d}x}
\newcommand{\dv}{\text{d}v}

\newcommand{\retainlabel}[1]{\label{#1}\sbox0{\ref{#1}}}

\AtEndEnvironment{proof}{\phantom{}\qed}

\usetikzlibrary{external}
\tikzexternaldisable

\pgfplotsset{compat=newest}
\newlength\figureheight
\newlength\figurewidth
\newlength\oldfw
\newlength\oldfh
\newcommand{\logLogSlopeTriangleIncrease}[5]
{

    \pgfplotsextra
    {
        \pgfkeysgetvalue{/pgfplots/xmin}{\xmin}
        \pgfkeysgetvalue{/pgfplots/xmax}{\xmax}
        \pgfkeysgetvalue{/pgfplots/ymin}{\ymin}
        \pgfkeysgetvalue{/pgfplots/ymax}{\ymax}
        \pgfmathsetmacro{\xArel}{#1}
        \pgfmathsetmacro{\yArel}{#3}
        \pgfmathsetmacro{\xBrel}{#1-#2}
        \pgfmathsetmacro{\yBrel}{\yArel}
        \pgfmathsetmacro{\xCrel}{\xArel}
        \pgfmathsetmacro{\lnxB}{\xmin*(1-(#1-#2))+\xmax*(#1-#2)}
        \pgfmathsetmacro{\lnxA}{\xmin*(1-#1)+\xmax*#1}
        \pgfmathsetmacro{\lnyA}{\ymin*(1-#3)+\ymax*#3}
        \pgfmathsetmacro{\lnyC}{\lnyA+#4*(\lnxA-\lnxB)}
        \pgfmathsetmacro{\yCrel}{\lnyC-\ymin)/(\ymax-\ymin)}
        \coordinate (A) at (rel axis cs:\xArel,\yArel);
        \coordinate (B) at (rel axis cs:\xBrel,\yBrel);
        \coordinate (C) at (rel axis cs:\xCrel,\yCrel);

        \draw[#5]   (A)-- node[pos=0.5,anchor=north, below=0pt] {\scriptsize 1}
                    (B)-- 
                    (C)-- node[pos=0.5,anchor=west] {\scriptsize #4}
                    cycle;
    }
}

\newcommand{\logLogSlopeTriangle}[5]
{
\pgfplotsextra
    {
        \pgfkeysgetvalue{/pgfplots/xmin}{\xmin}
        \pgfkeysgetvalue{/pgfplots/xmax}{\xmax}
        \pgfkeysgetvalue{/pgfplots/ymin}{\ymin}
        \pgfkeysgetvalue{/pgfplots/ymax}{\ymax}
        \pgfmathsetmacro{\xArel}{#1}
        \pgfmathsetmacro{\yArel}{#3}
        \pgfmathsetmacro{\xBrel}{#1-#2}
        \pgfmathsetmacro{\yBrel}{\yArel}
        \pgfmathsetmacro{\xCrel}{\xArel}
        \pgfmathsetmacro{\lnxB}{\xmin*(1-(#1-#2))+\xmax*(#1-#2)}
        \pgfmathsetmacro{\lnxA}{\xmin*(1-#1)+\xmax*#1}
        \pgfmathsetmacro{\lnyA}{\ymin*(1-#3)+\ymax*#3}
        \pgfmathsetmacro{\lnyC}{\lnyA-#4*(\lnxA-\lnxB)}
        \pgfmathsetmacro{\yCrel}{\lnyC-\ymin)/(\ymax-\ymin)} 
        \coordinate (A) at (rel axis cs:\xArel,\yArel);
        \coordinate (B) at (rel axis cs:\xBrel,\yBrel);
        \coordinate (C) at (rel axis cs:\xCrel,\yCrel);

        \draw[#5]   (A)-- node[pos=0.5,anchor=north, below=0pt] {\scriptsize 1}
                    (B)-- 
                    (C)-- node[pos=0.5,anchor=east] {\scriptsize #4}
                    cycle;
    }
}

\allowdisplaybreaks

\usepackage{soul}

\begin{document}

\maketitle

\begin{abstract}
Kinetic equations model distributions of particles in position-velocity phase space. Often, one is interested in studying the long-time behavior of particles in high-collisional regimes in which an approximate (advection)-diffusion model holds. In this paper we consider the diffusive scaling. Classical particle-based techniques suffer from a strict time-step restriction in this limit, to maintain stability. Asymptotic-preserving schemes avoid this problem, but introduce an additional time discretization error, possibly resulting in an unacceptably large bias for larger time steps. Here, we present and analyze a multilevel Monte Carlo scheme that reduces this bias by combining estimates using a hierarchy of different time step sizes. We demonstrate how to correlate trajectories from this scheme, using different time steps. We also present a strategy for selecting the levels in the multilevel scheme. Our approach significantly reduces the computation required to perform accurate simulations of the considered kinetic equations, compared to classical Monte Carlo approaches.
\end{abstract}

\section{Introduction}

In many application domains, one encounters kinetic equations, modelling particle behavior in a position-velocity phase space. Examples are plasma physics~\cite{Birdsall2004}, bacterial chemotaxis~\cite{Rousset2011c} and computational fluid dynamics~\cite{Pope1981}. Many of these domains exhibit a strong time-scale separation, leading to an unacceptably high simulation cost \cite{Cercignani1988}. Typically, one is interested in some macroscopic quantities of interest, e.g., some moments of the particle distribution, which are computed as averages over velocity space. The time-scale at which these quantities of interest change is often much slower than that governing the particle dynamics, making these models very stiff problems: a naive simulation requires both small time steps to capture the fast dynamics, and long time horizons to capture the evolution of the macroscopic quantities of interest. The exact nature of the macroscopic behavior depends on the problem scaling, which can be hyperbolic or diffusive~\cite{Dimarco2014}.

We are interested in a $d$-dimensional kinetic equation of the form
\begin{equation}
\label{eq:kinetic}
\partial_t f(x,v,t) + v\cdot\nabla_x f(x,v,t) = Q\left(f(x,v,t)\right), 
\end{equation}
where $f(x,v,t)$ represents the distribution of particles as a function of space $x\in \D_x \subset \R^d$, velocity $v \in \D_v \subset \R^d$ and time $t \in \R^+$ and $Q(f(x,v,t))$ is a collision operator, resulting in discontinuous velocity changes. In this paper, we consider the BGK operator~\cite{Bhatnagar1954}, which linearly drives the velocity distribution to a steady state distribution $\mathcal{M}(v)$. In this case, the operator $Q\left(f(x,v,t)\right)$ is written as \begin{equation}
\label{eq:BGK}
Q\left(f(x,v,t)\right) = \mathcal{M}(v)\rho(x,t)-f(x,v,t),
\end{equation}
in which we have introduced the position density
\begin{equation}
\rho(x,t) = \int_{\D_v} f(x,v,t)\dv.	
\end{equation}
With the collision operator \eqref{eq:BGK}, individual particles follow a velocity-jump process.

To make the time-scale separation explicit, we consider a dimensionless, diffusively scaled version of \eqref{eq:kinetic}. We introduce a parameter $\epsilon$, representing the mean free path. When $\epsilon$ decreases, the average time between collisions decreases.  In the diffusive scaling, we factor out the fast collision time scale by writing the right hand side as $(1/\epsilon) \; Q$, while simultaneously re-scaling time by $1/\epsilon$:
\begin{equation}
\label{eq:kineticdiffusive}
\epsilon\partial_t f(x,v,t) + v\cdot\nabla_x f(x,v,t) = \frac{1}{\epsilon}\left(\mathcal{M}(v)\rho(x,t)-f(x,v,t)\right).
\end{equation}
Taking the diffusion limit $\epsilon \to 0$, we drive the rate of collisions to infinity, while simultaneously increasing the slow time-scale.
It can be shown that, in the limit $\epsilon \to 0$, the particle density resulting from \eqref{eq:kineticdiffusive} converges to the diffusion equation~\cite{Lapeyre2003}
\begin{equation}
\label{eq:heat}
\partial_t \rho(x,t) = \Delta_{xx} \rho(x,t).
\end{equation}

Equation \eqref{eq:kineticdiffusive} can be simulated using a broad selection of methods. Deterministic methods solve the kinetic equation~\eqref{eq:kineticdiffusive} for $f(x,v,t)$ on a grid (using, for instance, finite differences or finite volumes), giving the particle distribution in the position-velocity phase space. This approach quickly becomes computationally infeasible as the dimension grows, as a grid must be formed over the $2d$-dimensional domain, $\D_x \times \D_v$. Stochastic methods, on the other hand, perform simulations of individual particle trajectories, with each trajectory representing a sample of the probability distribution $f(x,v,t)$. These methods do not suffer from the curse of dimensionality, but introduce a statistical error in the computed solution. When using explicit time steps, both approaches become prohibitively expensive for small values of $\epsilon$ due to the time-scale separation.

To avoid the issues caused by time-scale separation, one can use asymptotic-preserving schemes. Such methods preserve the macroscopic limit equation, in our case \eqref{eq:heat}, as $\epsilon$ tends to zero, but do not suffer from the time step constraints caused by the time-scale separation. A large number of such methods have already been developed in literature for deterministic methods in the diffusive limit. For a non-exhaustive list, we refer to~\cite{Bennoune2008,Boscarino2013,Buet2007,Crouseilles2011,Dimarco2012,Gosse2002,Jin1999,Jin1998,Jin2000,Klar1998,Klar1999,Larsen1974,Lemou2008,Naldi2000} and to a recent review paper~\cite{Dimarco2014}, which gives an overview of the current state of the art concerning these methods. In the particle setting, only a few asymptotic-preserving methods exist, mostly in the hyperbolic scaling \cite{Degond2011,Dimarco2008,Dimarco2010,Pareschi1999,Pareschi2001,Pareschi2005}.  In the diffusive scaling, we are only aware of three works~\cite{Crestetto2018,Dimarco2018,Mortier2019}. In this paper, we make use of the scheme proposed in~\cite{Dimarco2018}, where operator splitting was successfully applied to a modified kinetic equation, resulting in an unconditionally stable fixed time step particle method. This stability comes at the cost of an extra bias in the model, proportional to the size of the time step (see Section~\ref{sec:ap_scheme} for details).

In this paper, we present and analyse an approach to eliminate the bias introduced by using the asymptotic-preserving schemes with large time steps, through the use of the multilevel Monte Carlo method~\cite{Giles2008}. Multilevel Monte Carlo methods first compute an initial estimate, using a large number of samples with a large time step (and hence a low cost per sample). This estimate has a small variance, but is expected to have a large bias. Afterwards, the bias in this initial estimate is reduced by performing corrections using a hierarchy of simulations with increasingly smaller time steps. Under correct conditions, far fewer samples with small time steps are needed, compared to a direct Monte Carlo simulation with the smallest time step, resulting in a reduced computational cost. The multilevel Monte Carlo method was first introduced in finance~\cite{Giles2008}, and has since been applied to other fields such as biochemistry~\cite{Anderson2011}, data science~\cite{Hoel2016} and structural engineering~\cite{Blondeel2020}. Recently, the method has been applied to the simulation of large PDEs with random coefficients~\cite{Cliffe2011}, as well as to optimisation on these models~\cite{VanBarel2019}. The method has also recently been extended to higher dimensional parameter spaces as the multi-index Monte Carlo method~\cite{Haji-Ali2016}. A preliminary description of the algorithm presented here was published in~\cite{Loevbak2020}, together with partial numerical results. 

The rest of this paper is structured as follows: In Section~\ref{sec:kinetic_equations}, we present the main ideas behind kinetic equations and the asymptotic-preserving particle scheme. We also describe the model equation that will be used in the numerical experiments. In Section~\ref{sec:mlmc}, we give an overview of the multilevel Monte Carlo method. Section~\ref{sec:correlation} contains the main algorithmic contribution of this paper: we present an algorithm for generating coupled particle trajectories using the asymptotic-preserving Monte Carlo scheme at different levels of the multilevel Monte Carlo hierarchy. Section~\ref{sec:analysis} contains the corresponding numerical analysis. Here, we present some lemmas on the properties of the correlation of the coupled trajectories, as a function of the time step and the time-scale separation $\epsilon$. We then prove convergence of the scheme, making use of these lemmas. In Section~\ref{sec:experiments} we present a strategy for selecting which levels to include in the scheme and motivate this strategy with numerical results. Finally, in Section~\ref{sec:conclusion}, we summarize the main results and discuss possible extensions.

\section{Kinetic equations and asymptotic-preserving particle schemes}
\label{sec:kinetic_equations}
\subsection{Model equation and particle scheme}
\label{sec:diffusive_limit}
For the sake of exposition, we limit this work to one spatial dimension. The proposed method is, however, general. We will explicitly mention where care is needed when extending the method or its analysis to higher-dimensional models. We rewrite \eqref{eq:kineticdiffusive} as 
\begin{equation}
\label{eq:kineticdimless}
\partial_t f(x,v,t) + \dfrac{v}{\epsilon}\partial_x f(x,v,t) = \frac{1}{\epsilon^2}\left(\mathcal{M}(v)\rho(x,t)-f(x,v,t)\right),
\end{equation}
with $x\in\R$, $v\in\R$ and $t\in\R^+$.

Equation \eqref{eq:kineticdimless} can be simulated using particle schemes with  a finite time step of size $\Delta t$. Each particle has a state in the position-velocity phase space $(X,V)$ at each time step $n$, i.e., $X_{p,\Delta t}^n \approx X_p(n\Delta t)$ and $V_{p,\Delta t}^n \approx V_p(n\Delta t)$. We then represent the distribution of particles by an ensemble of $P$ particles, with indices $p \in \{1,\dots,P\}$,\begin{equation}
\label{eq:ensemble}
\left\{\left(X_{p,\Delta t}^n,V_{p,\Delta t}^n\right)\right\}_{p=1}^P.
\end{equation}
Classically, the ensemble \eqref{eq:ensemble} is simulated by operator splitting, which is first order in the time step $\Delta t$~\cite{Pareschi2005a}. For \eqref{eq:kineticdimless}, operator splitting results in two actions for each time step:
\begin{enumerate}
\item \textbf{Transport step.} Each particle's position is updated based on its velocity 
\begin{equation}
\label{eq:transport_step}
X^{n+1}_{p,\Delta t} = X^n_{p,\Delta t} + \Delta t V_{p,\Delta t}^n.
\end{equation}
\item \textbf{Collision step.} Between transport steps, each particle's velocity is either left unchanged (no collision) or re-sampled from $\mathcal{M}(v)$ (collision), i.e.,
\begin{equation}
\label{eq:collision_step}
V_{p,\Delta t}^{n+1} = 
\begin{cases}
V_{p,\Delta t}^{n,*} \sim \frac{1}{\epsilon}\mathcal{M}(v),&\text{with probability } p_{c,\Delta t} = \Delta t/\epsilon^2,\\
V_{p,\Delta t}^n,&\text{otherwise.} 
\end{cases}
\end{equation}
\end{enumerate}

To simplify the analysis in Section \ref{sec:analysis}, we consider a transformation of the velocity distribution $\mathcal{M}(v)$, where it is re-written as a distribution with unit variance by re-scaling both the x-axis and y-axis by its characteristic velocity $\CV$. Mathematically, we write this transformation as
\begin{equation}
\label{eq:velocity_decomposition}
\mathcal{M}(v) = \frac{1}{\CV}\B\left(\frac{v}{\CV}\right)=\frac{1}{\CV}\B(\bar{v}),
\end{equation}
with $\B(\bar{v})$ a probability distribution for which, given $\VDR \sim \B(\bar{v})$,
\begin{equation}
\label{eq:beta_properties}
\E\left[\VDR\right] = 0 \quad \text{and} \quad \V\left[\VDR\right] = 1.
\end{equation}
A visualization of the transformation shown in \eqref{eq:velocity_decomposition} is shown in Figure~\ref{fig:velocity_decomposition}.
\begin{figure}
\centering
\definecolor{mycolor1}{rgb}{0.12156862745098,0.466666666666667,0.705882352941177}
\definecolor{mycolor2}{rgb}{1,0.498039215686275,0.0549019607843137}

\pgfmathdeclarefunction{gauss}{2}{\pgfmathparse{1/(#2*sqrt(2*pi))*exp(-((x-#1)^2)/(2*#2^2))}}

\begin{tikzpicture}

\begin{axis}[%
width=0.6\figurewidth,
height=\figureheight,
at={(0\figurewidth,0\figureheight)},
xlabel={$v$},
ylabel={$\mathcal{M}(v)$},
xmin=-6.4,
xmax=6.4,
ymin=0,
ymax=0.42,
scale only axis,
axis x line=bottom,
axis y line=center,
xtick distance=2,
xticklabels={,-3$\CV$,-2$\CV$,-$\CV$,0,$\CV$,2$\CV$,3$\CV$},
ytick distance=0.199,
yticklabels={,,$\left( \CV\sqrt{2\pi} \right)^{-1}$},
axis on top
]
\addplot [fill=mycolor2!50, draw=none, domain=-6.4:6.4] {gauss(0,2)} \closedcycle;
\path (current bounding box.north east) -- (current bounding box.south east) coordinate[near start](leftcenter);
\end{axis}

\begin{axis}[%
width=0.3\figurewidth,
height=\figureheight,
at={(0.7\figurewidth,0\figureheight)},
xlabel={$\bar{v}$},
ylabel={$\mathcal{B}(\bar{v})$},
xmin=-3.2,
xmax=3.2,
ymin=0,
ymax=0.42,
scale only axis,
axis x line=bottom,
axis y line=center,
xtick={-3,-2,-1,0,1,2,3},
ytick distance=0.399,
yticklabels={,,$\sqrt{2\pi}^{-1}$},
axis on top
]
\addplot [fill=mycolor2!50, draw=none, domain=-3.2:3.2] {gauss(0,1)} \closedcycle;
\path (current bounding box.north west) -- (current bounding box.south west) coordinate[near start](rightcenter);
\end{axis}
\draw[-latex,very thick] (leftcenter) -- (leftcenter-|rightcenter) node[midway,above=2mm,text width=2cm]{\centering Re-scaling for unit variance}; 
\end{tikzpicture}%
\caption{Illustration of \eqref{eq:velocity_decomposition} for the case of a 1D Gaussian velocity distribution with $\CV>1$.
\label{fig:velocity_decomposition}}
\end{figure}
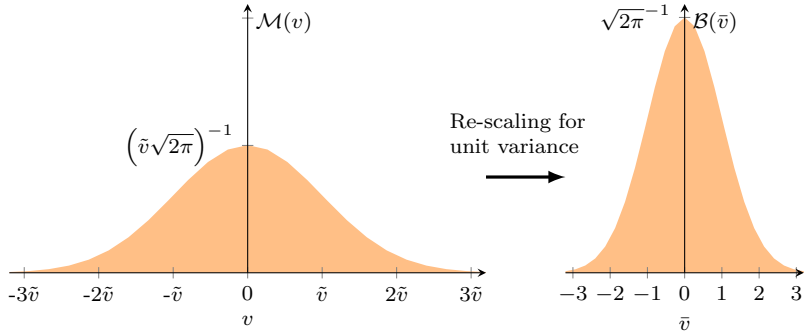

We give two examples:
\begin{itemize}
\item \textbf{Two discrete velocities.} To limit the cost of our simulations, we will consider $\mathcal{M}(v) = \frac{1}{2} \left( \delta_{v,-1} + \delta_{v,1}  \right)$, with $\delta$ the Kronecker delta function, i.e., $v$ can take the values $\pm 1$, with equal probability. In this case, \eqref{eq:kineticdimless} becomes
\begin{equation}
\label{eq:GT}
\begin{dcases}
\partial_t f_+(x,t) + \frac{1}{\epsilon} \partial_x f_+(x,t) = \frac{1}{\epsilon^2} \left( \frac{\rho(x,t)}{2} - f_+(x,t) \right) \\
\partial_t f_-(x,t) - \frac{1}{\epsilon} \partial_x f_-(x,t) = \frac{1}{\epsilon^2} \left( \frac{\rho(x,t)}{2} - f_-(x,t) \right)
\end{dcases},
\end{equation}
with $f_+(x,t)$ the distribution of particles with positive velocity and $f_-(x,t)$ that of particles with negative velocity. The total particle density is given by $\rho(x,t) = f_+(x,t) + f_-(x,t)$. Equation \eqref{eq:GT} is known as the Goldstein-Taylor model~\cite{Goldstein1951}.
\item \textbf{Normal distribution.} Another common choice is $\mathcal{M}(v) = \mathcal{N}(v;0,\CV^2)$, i.e., the normally distributed with expected value 0 and variance $\CV^2$. 
\end{itemize}

Scheme \eqref{eq:transport_step}--\eqref{eq:collision_step} has a severe time step restriction $\Delta t = \mathcal{O}(\epsilon^2)$ when approaching the limit $\epsilon \to 0$. This time step restriction will often result in unacceptably high simulation costs, despite the well-defined limit~\cite{Dimarco2018}.

\subsection{Asymptotic-preserving Monte Carlo scheme\label{sec:ap_scheme}}
In~\cite{Dimarco2018}, an asymptotic-preserving scheme was proposed as a solution to the high simulation cost of \eqref{eq:transport_step}--\eqref{eq:collision_step} in the limit $\epsilon \to 0$. This asymptotic-preserving scheme works by rewriting \eqref{eq:kineticdimless} as
\begin{equation}
\label{eq:GTmod}
\partial_t f + \frac{\epsilon v}{\epsilon^2+\Delta t} \partial_x f = \frac{\tilde{v}^2 \Delta t}{\epsilon^2+\Delta t} \partial_{xx} f + \frac{1}{\epsilon^2+\Delta t} \left(\mathcal{M}(v)\rho - f \right),
\end{equation}
using an approach based on the IMEX discretization. In \eqref{eq:GTmod}, we omit the space, velocity and time dependency of $f(x,v,t)$ and $\rho(x,t)$, for conciseness. In the limit $\epsilon \to 0$, it can be shown that the modified equation \eqref{eq:GTmod} converges to the diffusion limit \eqref{eq:heat}. It can also be shown that, in the limit $\Delta t \to 0$, \eqref{eq:GTmod} converges to the original kinetic equation \eqref{eq:kineticdimless} with a rate $\mathcal{O}(\Delta t)$, see~\cite{Dimarco2018}.

Particle trajectories are now simulated as follows:
\begin{enumerate}
\item \textbf{Transport-diffusion step.} The position of the particle is updated based on its velocity and a Brownian increment
\begin{equation}
\label{eq:transport}
X^{n+1}_{p,\Delta t} = X^n_{p,\Delta t} + V^n_{p,\Delta t} \Delta t + \sqrt{2 \Delta t}\sqrt{D_{\Delta t}}\xi^n_p, \quad V^n_{p} \sim \mathcal{M}_{\Delta t}(v),
\end{equation}
in which we have taken $\xi_p^n \sim \mathcal{N}(0,1)$ and introduced a $\Delta t$-dependent velocity distribution $\mathcal{M}_{\Delta t}(v)$ and diffusion coefficient $D_{\Delta t}$:
\begin{equation}
\mathcal{M}_{\Delta t}(v) = \frac{\epsilon}{\epsilon^2+\Delta t} \mathcal{M}(v), \qquad D_{\Delta t}=\frac{\tilde{v}^2\Delta t}{\epsilon^2 + \Delta t}.
\end{equation}
We define the characteristic velocity of $\mathcal{M}_{\Delta t}(v)$ as $\displaystyle \CV_{\Delta t} = \frac{\epsilon}{\epsilon^2+\Delta t}\CV.$
\item \textbf{Collision step.} During collisions, each particle's velocity is updated as:
\begin{equation}
\label{eq:collision}
V_{p,\Delta t}^{n+1} = 
\begin{cases}
V_{p,\Delta t}^{n,*} \sim \mathcal{M}_{\Delta t}(v),&\text{with probability } p_{c,\Delta t} = \dfrac{\Delta t}{\epsilon^2+\Delta t},\\
V_{p,\Delta t}^n,&\text{otherwise.} 
\end{cases}
\end{equation}
\end{enumerate}
For the Goldstein-Taylor model, sampling the time-step dependent velocity distribution $\mathcal{M}_{\Delta t}(v)$ means multiplying the characteristic velocity $\CV_{\Delta t}$ with $\pm 1$ with equal probability, which satisfies \eqref{eq:velocity_decomposition}--\eqref{eq:beta_properties}. For more details on the scheme \eqref{eq:transport}--\eqref{eq:collision} see~\cite{Dimarco2018}.

\section{Multilevel Monte Carlo method}
\label{sec:mlmc}
We want to calculate the value of a quantity of interest (QoI) $Y(t^*)$, which is the integral of a function $F(x,v)$ of the particle position $X(t)$ and velocity $V(t)$ at time $t=t^*$ with respect to $f(x,v,t)$. In practical applications, integration is commonly performed only over the velocity space, giving a quantity of interest that is time and space dependent, e.g., a density or flux. Here, we opt to also integrate over position space, i.e.,
\begin{equation}
\label{eq:QoI}
Y(t^*) = \E \left[F\left(X(t^*),V(t^*)\right)\right] = \int_{\D_v}\int_{\D_x} F(x,v) f(x,v,t^*)\dx\dv.
\end{equation}
This choice simplifies the notation in the remainder of this paper, while also decreasing the simulation cost of our test problems, facilitating an in-depth numerical analysis.

A classical Monte Carlo estimator $\hat{Y}(t^*)$ for $Y(t^*)$ is given by
\begin{equation}
\label{eq:MCestimator}
\hat{Y}(t^*) = \frac{1}{P}\sum_{p=1}^P F\left(X^N_{p,\Delta t}, V^N_{p,\Delta t}\right), \quad t^* = N \Delta t,
\end{equation}
with particles $\left(X^N_{p,\Delta t}, V^N_{p,\Delta t}\right)$, simulated using the time discretization \eqref{eq:transport}--\eqref{eq:collision}.

Given a fixed cost budget for the estimator $\hat{Y}(t^*)$, i.e., a maximal value for the product of the number of time steps and particle simulations $N \times P$, we have to make a trade-off. On the one hand, if we choose to perform more accurate simulations by taking a small $\Delta t$, the individual trajectories will have a small bias, but the required number of time steps $N$ will be very large. As a consequence, we can only simulate a limited number of trajectories $P$. Given that the variance of \eqref{eq:MCestimator} is given by
\begin{equation}
\V\left[\hat{Y}(t^*)\right] = \frac{1}{P}\V\left[F\left(X^N_{p,\Delta t}, V^N_{p,\Delta t}\right)\right],
\end{equation}
the estimated quantity of interest $\hat{Y}(t^*)$ will have a large variance, due to insufficient sampling of the trajectory space. On the other hand, if we choose to simulate a large number of particles $P$, reducing the number of time steps $N$ by choosing a larger $\Delta t$, the estimator variance will be smaller, but the simulation bias will be larger due to the time discretization error.

The core idea behind the multilevel Monte Carlo method (MLMC)~\cite{Giles2008,Heinrich2001} is to avoid this trade-off by combining estimators based on trajectories with different time step sizes. This is done by starting with a coarse time step size $\Delta t_0$. At this coarse level, we can cheaply simulate a large number of trajectories $P_0$, as the number of required time steps $N_0$ to reach the end time $t^*$ is small. This estimator has a large bias, but a low variance, and is given by
\begin{equation}
\label{eq:MC0estimator}
\hat{Y}_0(t^*) = \frac{1}{P_0}\sum_{p=1}^{P_0} F\left(X^{N_0}_{p,\Delta t_0},V^{N_0}_{p,\Delta t_0}\right), \quad t^* = N_0 \Delta t_0.
\end{equation}
The estimator $\hat{Y}_0(t^*)$ is then refined upon by a sequence of $L$ difference estimators at levels $\ell=1,\ldots,L$. Each difference estimator uses an ensemble of $P_\ell$ particle pairs
\begin{equation}
\label{eq:MClestimator}
\hat{Y}_\ell(t^*) = \frac{1}{P_\ell}\sum_{p=1}^{P_\ell}\left(F\left(X^{N_\ell}_{p,\Delta t_\ell},V^{N_\ell}_{p,\Delta t_\ell}\right)-F\left(X^{N_{\ell-1}}_{p,\Delta t_{\ell-1}},V^{N_{\ell-1}}_{p,\Delta t_{\ell-1}}\right)\right),\quad t^* = N_\ell \Delta t_\ell.
\end{equation}
Each particle pair consists of two coupled particles: a particle with a fine time step $\Delta t_\ell$ and a particle with a coarse time step $\Delta t_{\ell-1} = M \Delta t_\ell$, with $M$ a positive integer. The variance of one such particle pair
\begin{equation}
\V \left[ F\left(X^{N_\ell}_{p,\Delta t_\ell},V^{N_\ell}_{p,\Delta t_\ell}\right)-F\left(X^{N_{\ell-1}}_{p,\Delta t_{\ell-1}},V^{N_{\ell-1}}_{p,\Delta t_{\ell-1}}\right) \right] 
\end{equation}
can be decomposed as 
\begin{align}
&\V \left[ F\left(X^{N_\ell}_{p,\Delta t_\ell},V^{N_\ell}_{p,\Delta t_\ell}\right) \right] + \V \left[ F\left(X^{N_{\ell-1}}_{p,\Delta t_{\ell-1}},V^{N_{\ell-1}}_{p,\Delta t_{\ell-1}}\right) \right] \\
&\qquad - 2 \Cov \left( F\left(X^{N_\ell}_{p,\Delta t_\ell},V^{N_\ell}_{p,\Delta t_\ell}\right) , F\left(X^{N_{\ell-1}}_{p,\Delta t_{\ell-1}},V^{N_{\ell-1}}_{p,\Delta t_{\ell-1}}\right)  \right).
\end{align}
If we ensure that $F\left(X^{N_\ell}_{p,\Delta t_\ell},V^{N_\ell}_{p,\Delta t_\ell}\right)$ and $F\left(X^{N_{\ell-1}}_{p,\Delta t_{\ell-1}},V^{N_{\ell-1}}_{p,\Delta t_{\ell-1}}\right)$ are sufficiently correlated, then we expect that
\begin{equation}
\V \left[ F\left(X^{N_\ell}_{p,\Delta t_\ell},V^{N_\ell}_{p,\Delta t_\ell}\right)-F\left(X^{N_{\ell-1}}_{p,\Delta t_{\ell-1}},V^{N_{\ell-1}}_{p,\Delta t_{\ell-1}}\right) \right] \leq
\V \left[ F\left(X^{N_\ell}_{p,\Delta t_\ell},V^{N_\ell}_{p,\Delta t_\ell}\right) \right].
\end{equation}
This means that fewer particle simulations are needed in order to achieve a given variance when estimating \eqref{eq:MClestimator}, as compared to a single level estimator using a time step size $\Delta t_\ell$. We achieve correlation in the sampled quantity of interest by making each particle pair undergo correlated simulations, which intuitively can be understood as an attempt to let two particles follow the same trajectory for two different simulation accuracies. We give detailed explanation on how this is achieved in Section~\ref{sec:correlation}. 
 
One can interpret the difference estimator as using the fine simulation to estimate the bias in the coarse simulation. Given a sequence of levels $\ell \in \{0,\dots,L\}$, with decreasing step sizes, and the corresponding estimators given by \eqref{eq:MC0estimator}--\eqref{eq:MClestimator}, the multilevel Monte Carlo estimator for the quantity of interest $Y(t^*)$ is computed by the telescopic sum
\begin{equation}
\label{eq:telescopic}
\hat{Y}(t^*) = \sum_{\ell=0}^{L} \hat{Y}_\ell(t^*).
\end{equation}
It is clear that the expected value of the estimator \eqref{eq:telescopic} is the same as that of \eqref{eq:MCestimator}, with the finest time step $\Delta t = \Delta t_L$. Given a sufficiently quick reduction in the number of simulated (pairs of) trajectories $P_\ell$ as $\ell$ increases, it is possible to show that the multilevel Monte Carlo estimator is able to achieve the same mean square error as the classical Monte Carlo estimator at a lower computational cost.

For a detailed overview of the multilevel Monte Carlo method and its properties, we refer the reader to~\cite{Giles2015}. Here, we limit ourselves to mentioning the theorem originally presented by Giles in~\cite{Giles2008} and further generalized in~\cite{Cliffe2011} and~\cite{Giles2015}, which gives an upper bound for the computational complexity of multilevel Monte Carlo. 

\begin{remark}[Notation]
For conciseness, we use $\hat{F}_\ell$ instead of $F\left(X^{N_\ell}_{p,\Delta t_\ell},V^{N_\ell}_{p,\Delta t_\ell}\right)$ further in this work when considering approximations of the function $F\left(X(t^*),V(t^*)\right)$ at level $\ell$. For $F\left(X(t^*),V(t^*)\right)$ we use the short-hand notation $F$.  Additionally, we introduce the requested bound on the mean square error $E$, and a series of positive constants $\alpha$, $\beta$, $\gamma$, $c_1$, $c_2$, $c_3$ and $c_4$, which are problem-dependent. In the theorem, $e$ is used to denote Euler's constant.
\end{remark}

\begin{theorem}
\label{thm:giles}
Let $F$ denote a random variable, and let $\hat{F}_\ell$ denote the
corresponding level $\ell$ numerical approximation. If there exist independent estimators $\hat{Y}_\ell$ based on $P_\ell$ Monte Carlo samples, each with expected cost $C_\ell$ and variance $V_\ell$, and positive constants $\alpha$, $\beta$, $\gamma$, $c_1$, $c_2$, $c_3$ such that $\alpha \geq \frac{1}{2} \min(\beta, \gamma)$ and
\begin{enumerate}
\item $\E\left[\hat{Y}_\ell\right]= \begin{cases} \E\left[\hat{F}_0\right]&\ell = 0,\\
                                  \E\left[\hat{F}_\ell - \hat{F}_{\ell-1}\right]&\ell > 0,\end{cases}$ \hfill{(Estimator notation)}
\item $\left|\E\left[\hat{F}_\ell - F\right]\right| \leq c_1 2^{-\alpha\ell}$, \hfill{(Bias decreases with increasing $\ell$)}
\item $V_\ell \leq c_2 2^{-\beta \ell}$, \hfill{(Variance decreases with increasing $\ell$)}
\item $C_\ell \leq c_3 2^{\gamma \ell}$, \hfill{(Cost increases with increasing $\ell$)}
\end{enumerate}
then there exists a positive constant $c_4$ such that for any $\mse < e^{-1}$ there are values $L$ and $N_\ell$ for which the multilevel estimator \eqref{eq:telescopic} has a mean square error with bound $\text{MSE} \equiv \E (\hat{Y} - \E[F])^2 < \mse^2$ with a computational complexity $C$ with bound $$\E\left[C\right]\leq \begin{cases} c_4 \mse^{-2}, &\beta > \gamma,\\
c_4 \mse^{-2}\log^2 \mse, &\beta = \gamma,\\
c_4 \mse^{-2-\frac{(\gamma - \beta)}{\alpha}}, &\beta < \gamma. \end{cases}$$
\end{theorem}
The essential idea behind the theorem is the following. If the variance decreases faster than the cost increases, $\beta > \gamma$, then most of the work will be done at coarser levels. In this case the same complexity is achieved as a single level Monte Carlo simulation at the coarsest level. This is case in which the method is the most effective. If the variance decreases slower than the cost increases, $\beta < \gamma$, then most of the work is done at the finer levels. In this case the multilevel Monte Carlo has a much better asymptotic complexity as the finest level requires $\mathcal{O}(1)$ samples, each with cost $\mathcal{O}\left( \mse^{-\gamma / \alpha} \right)$. However, the constant $c_4$ will be large, so the multilevel speedup will be relatively small. In the intermediate case, $\beta = \gamma$, the computational cost is spread over the levels. In this case the $\log^2 \mse$ factor corresponds with the number of required levels. We will refer back to this theorem in Section~\ref{sec:proof}, where we demonstrate the convergence of our scheme.

\section{Correlating particle trajectories}
\label{sec:correlation}

Recall that the time steps at levels $\ell$ and $\ell-1$ are related through $\Delta t_{\ell-1}=M\Delta t_{\ell}$. At the fine level $\ell$, we therefore define a sub-step index $m \in \{1,\dots, M\}$, i.e., we write $X^{n,m}_{p,\Delta t_\ell}\approx X_p(n\Delta t_{\ell-1}+m\Delta t_{\ell})\equiv X_p((nM+m)\Delta t_{\ell})$. Subsequently, we introduce a coupled pair of simulations spanning a time step with size $\Delta t_{\ell-1}$: (i) a simulation at level $\ell-1$, using a single time step of size $\Delta t_{\ell-1}$ and (ii) a simulation at level $\ell$, using $M$ time steps of size $\Delta t_\ell$:
\begin{equation}
\label{eq:coupled_simulation}
\begin{dcases}
X^{n+1}_{p,\Delta t_{\ell-1}} = X^{n}_{p,\Delta t_{\ell-1}} + \Delta t_{\ell-1}V_{p,\Delta t_{\ell-1}}^n + \sqrt{2 \Delta t_{\ell-1}} \sqrt{D_{\Delta t_{\ell-1}}} \xi^{n}_{p,\ell-1} \\
X^{n+1,0}_{p,\Delta t_{\ell}} = X^{n,0}_{p,\Delta t_{\ell}} + \sum_{m=0}^{M-1} \left( \Delta t_\ell V_{p,\Delta t_\ell}^{n,m} + \sqrt{2 \Delta t_{\ell}} \sqrt{D_{\Delta t_\ell}} \xi^{n,m}_{p,\ell}\right)
\end{dcases},
\end{equation}
with $\xi^{n}_{p,\ell-1}, \xi^{n,m}_{p,\ell} \sim \mathcal{N}(0,1)$, $V_{p,\Delta t_{\ell-1}}^n \sim \mathcal{M}_{\Delta t_{\ell-1}}(v)$ and $V_{p,\Delta t_{\ell}}^{n,m} \sim \mathcal{M}_{\Delta t_\ell}(v)$. 

The differences in \eqref{eq:MClestimator} will only have low variance if the simulated paths up to $X^{N_\ell}_{\Delta t_\ell,p} = X^{N_{\ell-1},0}_{\Delta t_\ell,p}$ and $X^{N_{\ell-1}}_{\Delta t_{\ell-1},p}$ are correlated. To achieve correlation, we need to perform coupled simulations of the scheme \eqref{eq:transport}--\eqref{eq:collision} with different time steps. In each time step using the asymptotic-preserving particle scheme, there are two  sources of stochastic behavior. On the one hand, a new Brownian increment $\xi_p^n$ is generated for each particle in each transport-diffusion step~\eqref{eq:transport}. On the other hand, in each collision step~\eqref{eq:collision}, a fraction of particles randomly get a new velocity $V_p^{n,*}$ for use in the following time step. Particle trajectories can be coupled by correlating the random numbers used for the individual particles in the transport-diffusion and collision phase of each time step. 

To this end, we will not draw independent samples at level $\ell-1$, but instead compute the values $\xi^n_{p,\ell-1}$ and $V_{p,\Delta t_{\ell-1}}^{n+1} = V_{p,\Delta t_{\ell-1}}^{n,*}$ at level $\ell-1$ based on the respective values of $\xi^{n,m}_{p,\ell}$ and $V_{p,\Delta t_\ell}^{n,m,*}$ in the fine simulation, while ensuring their correct statistical distribution. To achieve this, we thus first perform the $M$ fine simulation steps. Using the random numbers in these $M$ sub-steps, we then compute values with which to define the value of the random numbers in the single step of size $\Delta t_{\ell-1}$. The way in which we compute these dependent coarse random numbers at level $\ell-1$ is described in Algorithm~\ref{alg:correlation}. If the numbers at level $\ell-1$ have the correct statistical distribution, the coarse simulation statistics are not affected by the introduced correlation. Generalization of Algorithm~\ref{alg:correlation} to higher dimensional domains for the position and velocity, can be done by simply replacing the random values $\xi$ and $\VD$ with random vector quantities of equal dimension to the respective domains.

\begin{algorithm}[]
\begin{algorithmic}[1]
\FOR{Each time step $n$}
\FOR{$m = 0 \dots M-1$}
\STATE Simulate \eqref{eq:transport}--\eqref{eq:collision} with $\Delta t_\ell$, saving the $\xi^{n,m}_{p,\ell}$, $\PT^{n,m}_{p,\ell}$ and $\VD^{n,m}_{p,\ell}$
\ENDFOR
\STATE Set $\xi^{n}_{p,\ell-1} = \frac{1}{\sqrt{M}}\sum_{m=0}^{M-1} \xi^{n,m}_{p,\ell}$
\STATE Set $\PT_{p,\ell-1}^n = \left( \max_m \PT_{p,\ell}^{n,m} \right)^M$
\IF{$\PT_{p,\ell-1}^n \geq p_{nc,\Delta t_{\ell-1}}$}
\STATE Set $\VD_{p,{\ell-1}}^{n+1} = \VD_{p,{\ell-1}}^{n,M,*}$
\ELSE
\STATE Set $\VD_{p,{\ell-1}}^{n+1} = \VD_{p,{\ell-1}}^{n}$
\ENDIF
\ENDFOR
\end{algorithmic}
¨ \caption{Performing correlated simulation steps. The values $u^{n,m}_{p,\ell}, u^{n}_{p,\ell-1} \sim \mathcal{U}([0,1])$ are compared with the respective probabilities $\pncf$ and $\pncr$. If they are larger than the probability that no collision takes place, then a collision takes place in the given time step.\label{alg:correlation}}
\end{algorithm}

In the remainder of this section we will motivate the different steps in Algorithm~\ref{alg:correlation}. To achieve correlated simulations, we need to couple two sources of random behavior in \eqref{eq:coupled_simulation}. On the one hand, we generate a new normally distributed $\xi_{p,\ell}^n$ in each transport-diffusion step. On the other hand, there is a possibility that a collision occurs at each simulation step, causing the selection of a new velocity $V_{p,\ell}^n$. To discuss how correlation is introduced in both of these cases, we decompose each step in \eqref{eq:coupled_simulation} into two parts: the transport part and the diffusion part.  Defining transport increments
\begin{equation}
\label{eq:transport_increments}
\Delta \U^n_{p,\ell-1} = \Delta t_{\ell-1}V_{p,\Delta t_{\ell-1}}^n , \quad
\Delta \U^{n,m}_{p,\ell} = \Delta t_\ell V_{p,\Delta t_\ell}^{n,m},
\end{equation}
and  Brownian increments
\begin{equation}
\label{eq:brownian_increments}
\Delta W^n_{p,\ell-1} = \sqrt{2\Delta t_{\ell-1}} \sqrt{D_{\ell-1}} \xi^{n}_{p,\ell-1} , \quad
\Delta W^{n,m}_{p,\ell} = \sqrt{2\Delta t_{\ell}} \sqrt{D_{\ell}} \xi^{n,m}_{p,\ell},
\end{equation}
we can rewrite \eqref{eq:coupled_simulation} as
\begin{equation}
\label{eq:deltaX}
\begin{dcases}
X^{n+1}_{p,\Delta t_{\ell-1}} = X^{n}_{p,\Delta t_{\ell-1}} + \Delta \U^n_{p,\ell-1} + \Delta W^n_{p,\ell-1} \\
X^{n+1,0}_{p,\Delta t_{\ell}} = X^{n,0}_{p,\Delta t_{\ell}} + \sum_{m=0}^{M-1} \left( \Delta \U^{n,m}_{p,\ell} + \Delta W^{n,m}_{p,\ell}\right)
\end{dcases},
\end{equation}

We first motivate the correlation of the Brownian increments in line 5 of Algorithm~\ref{alg:correlation} (Section~\ref{sec:correlation_brownian}). Next, we motivate the correlation of the transport increments in lines 6--11 of Algorithm~\ref{alg:correlation} (Section~\ref{sec:correlation_transport}). After giving the motivation for the algorithm, we briefly show that this algorithm indeed achieves correlated simulations in Section~\ref{sec:correlation_variance}.

\subsection{Brownian increments}
\label{sec:correlation_brownian}

We now consider just the Brownian increments \eqref{eq:brownian_increments} for two processes spanning a time interval $\Delta t_{\ell-1}$, one spanning the interval in a single time step and the other taking $M$ time steps of size $\Delta t_\ell$. After $M$ fine Brownian increments $\Delta W^{n,m}_{p,\ell}$ at level $\ell$, according to \eqref{eq:brownian_increments}, we get a Brownian increment spanning $\Delta t_{\ell-1}$,
\begin{equation}
\sum_{m=0}^{M-1} \Delta W^{n,m}_{p,\ell} = \sqrt{2\Delta t_{\ell}} \sqrt{D_{\ell}} \sum_{m=0}^{M-1} \xi^{n,m}_{p,\ell} .
\end{equation}
Given that
\begin{equation}
\V\left[\sum_{m=0}^{M-1} \xi^{n,m}_{p,\ell} \right] = M,
\end{equation}
we can compute a standard normally distributed $\xi^{n}_{p,\ell-1}$ from the $\xi^{n,m}_{p,\ell}$ as
\begin{equation}
\label{eq:xicorr}
\xi^{n}_{p,\ell-1} = \frac{1}{\sqrt{M}}\sum_{m=0}^{M-1} \xi^{n,m}_{p,\ell} \sim \mathcal{N}(0,1),
\end{equation}
giving us the expression on line 5 of Algorithm~\ref{alg:correlation}.

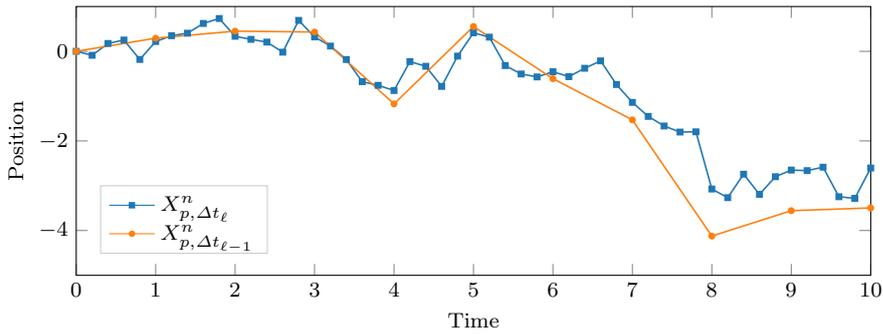
\begin{figure}
\centering
%
%
\definecolor{mycolor1}{rgb}{0.12156862745098,0.466666666666667,0.705882352941177}
\definecolor{mycolor2}{rgb}{1,0.498039215686275,0.0549019607843137}
\begin{tikzpicture}

\begin{axis}[%
width=\figurewidth,
height=\figureheight,
at={(0\figurewidth,0\figureheight)},
scale only axis,
xlabel={Time},
ylabel={Position},
xmin=0,
xmax=10,
ymin=-5,
ymax=1,
axis background/.style={fill=white},
legend entries={{$X^{n}_{p,\Delta t_{\ell}}$},{$X^{n}_{p,\Delta t_{\ell-1}}$}},
legend cell align={left},
legend style={at={(0.03,0.08)}, anchor=south west, draw=white!80.0!black}
]
\addlegendimage{mycolor1, mark=square*, mark size=1}
\addlegendimage{mycolor2, mark=*, mark size=1}
\addplot [semithick, color=mycolor1, mark=square*, mark options={}, mark size=1pt, forget plot]
  table[row sep=crcr]{%
0	0\\
0.2	-0.0884229073368045\\
0.4	0.175180492605981\\
0.6	0.25243589000908\\
0.8	-0.181752229749213\\
1	0.218474835734754\\
1.2	0.347943508695561\\
1.4	0.404938268644316\\
1.6	0.622185198819496\\
1.8	0.732403782657772\\
2	0.335438526543269\\
2.2	0.266990953728434\\
2.4	0.205408912753681\\
2.6	-0.01890677949138\\
2.8	0.69033036977426\\
3	0.321090456745502\\
3.2	0.117096120129971\\
3.4	-0.18311135711372\\
3.6	-0.678202746988402\\
3.8	-0.759258044550535\\
4	-0.874816199941612\\
4.2	-0.229680982501482\\
4.4	-0.334679033178855\\
4.6	-0.783390806676384\\
4.8	-0.107313847625855\\
5	0.413272590199273\\
5.2	0.316453577390658\\
5.4	-0.318599113881461\\
5.6	-0.506070662020094\\
5.8	-0.571821175830221\\
6	-0.455420579551636\\
6.2	-0.565536841243248\\
6.4	-0.378573746682483\\
6.6	-0.213336639530805\\
6.8	-0.7406691684776\\
7	-1.14036458813111\\
7.2	-1.45284235408876\\
7.4	-1.66695703338024\\
7.6	-1.80212353510279\\
7.8	-1.79686612564391\\
8	-3.07407943703199\\
8.2	-3.26677372892507\\
8.4	-2.74291148366623\\
8.6	-3.19267228421021\\
8.8	-2.7989779229455\\
9	-2.65126961941534\\
9.2	-2.66349952322763\\
9.4	-2.5865709347805\\
9.6	-3.24645649084197\\
9.8	-3.28210146529928\\
10	-2.60581830361661\\
};
\addplot [semithick, color=mycolor2, mark=*, mark options={}, mark size=1pt, forget plot]
  table[row sep=crcr]{%
0	0\\
1	0.293114750445606\\
2	0.450038008573871\\
3	0.430788052925638\\
4	-1.1736890945325\\
5	0.554463362953792\\
6	-0.611010824537885\\
7	-1.52995964291682\\
8	-4.12431035366269\\
9	-3.55705145741561\\
10	-3.49607211833995\\
};
\end{axis}

\end{tikzpicture}%
\caption{Correlated diffusion steps with $\epsilon=0.5$, $\Delta t_\ell=0.2$ and $\Delta t_{\ell-1}=1$.
\label{fig:corrpathdiff}}
\end{figure}

Correlating the simulations in this way means that both simulations follow the same Brownian path, and differences in the diffusion behavior only result from differences in the diffusion coefficients $D_\ell$ and $D_{\ell-1}$ due to the different time steps. In Figure \ref{fig:corrpathdiff} we show two particle trajectories, containing a series of increments $\Delta W^{n,m}_{p,\ell}$ and $\Delta W^{n}_{p,\ell-1}$, using coupled normally distributed numbers as described in \eqref{eq:xicorr}, with $\epsilon=0.5$, $\Delta t_\ell=0.2$ and $M=5$. We observe that the paths have similar behavior, i.e., if the fine simulation tends towards negative values, so does the coarse simulation and vice versa. Still, there is an observable difference between them. This is due to difference in simulation variance caused by the differing diffusion coefficients.

\subsection{Transport increments}
\label{sec:correlation_transport}

While correlating Brownian increments \eqref{eq:brownian_increments} is straightforward, correlating the transport increments \eqref{eq:transport_increments} is more involved. Since we simulate level $\ell$ first, we have the increments $\Delta \U^{n,m}_{p,\ell}$, spanning a total time interval $\Delta t_\ell$, at our disposal. Our goal is to use the random numbers in these increments to calculate a single increment $\Delta \U^{n}_{p,\ell-1}$ that spans the same time interval. Note that, in the collision phase of the asymptotic-preserving particle scheme (Section~\ref{sec:ap_scheme}), both the value of the velocity and the probability of collision depend on the value of the time step $\Delta t$, and therefore depend on the level $\ell$.  The coupling is done in two steps. First, the occurrence of a collision in each simulation step of the coarse simulation is coupled to the occurrence of a collision in at least one of the $M$ sub-steps of the coupled fine simulation. Then, if a collision occurs at both level $\ell$ and $\ell-1$, we will correlate the new velocities generated at both levels.

\subsubsection{Deciding upon collision in the coarse simulation}

Let us first consider the simulation at level $\ell$. In each of the $M$ fine simulation time steps of size $\Delta t_\ell$, we need to decide whether or not the particle collided. To this end, we draw a random number $\PT^{n,m}_{p,\ell} \sim \mathcal{U}([0,1])$ in the evaluation of \eqref{eq:collision}. If this number is larger than the probability that no collision has occurred in the the time step $p_{nc,\Delta t_\ell}=1-p_{c,\Delta t_\ell}$
\begin{equation}
\label{eq:collisioncondition}
\PT_{p,\ell}^{n,m} \geq p_{nc,\Delta t_\ell} = \frac{\epsilon^2}{\epsilon^2+\Delta t_\ell},
\end{equation}
then a collision is performed, i.e., a new velocity is randomly drawn from $\mathcal{M}_{\Delta t_\ell}$ at the end of that time step. At least one collision takes place in the interval spanning $\Delta t_{\ell-1}$ if at least one of the generated $\PT_{p,\ell}^{n,m}$, $m \in \{0,\dots,M-1\}$, satisfies \eqref{eq:collisioncondition}, i.e.,

 \begin{equation}
\label{eq:maxalpha}
\PT^{n,\text{max}}_{p,\ell} = \max_m \PT_{p,\ell}^{n,m} \geq p_{nc,\Delta t_\ell}.
\end{equation}
When~\eqref{eq:maxalpha} is satisfied, we want the probability of a collision taking place in the correlated coarse simulation to be large. 

We thus wish to generate a $\PT_{p,\ell-1}^n \sim \mathcal{U}\left([0,1]\right)$ based on the value of $\PT^{n,\text{max}}_{p,\ell}$, which we can use to test for a collision at level $\ell-1$. 
Since the cumulative density function of $\PT^{n,\text{max}}_{p,\ell}$ is given by
\begin{equation}
\label{eq:CDF_maxalpha}
\text{CDF}\left(\PT^{n,\text{max}}_{p,\ell}\right) = {\left(\PT^{n,\text{max}}_{p,\ell}\right)}^M,
\end{equation}
we get, by the inverse transform method, that ${\left(\PT^{n,\text{max}}_{p,\ell}\right)}^M \sim \mathcal{U}([0,1])$. This means that we can achieve our goal by setting
\begin{equation}
\label{eq:alphacorr}
\PT_{p,\ell-1}^n = {\left(\PT^{n,\text{max}}_{p,\ell}\right)}^M,
\end{equation}
giving line 6 of Algorithm~\ref{alg:correlation}.

A collision at level $\ell-1$ then occurs when 
\begin{equation}
\label{eq:collisioncondition_level_coarse}
\PT_{p,\ell-1}^n \geq p_{nc,\Delta t_{\ell-1}} = \frac{\epsilon^2}{\epsilon^2+\Delta t_{\ell-1}},
\end{equation}
giving line 7 of Algorithm~\ref{alg:correlation}.

\subsubsection{Choosing a new velocity}
We first show that a collision at level $\ell-1$ can only occur when at least one collision has taken place in the simulation at level $\ell$.
\begin{lemma}
\label{lem:collision_likelihood}
Given a simulation interval $\Delta t_{\ell-1}$ and a pair of paths $X_{p,\ell-1}$ and $X_{p,\ell}$ spanning this interval with time step sizes, respectively, $\Delta t_{\ell-1}$ and $\Delta t_\ell$, with collision behavior correlated as in \eqref{eq:alphacorr}. Then, the absence of collision in the $M$ sub-steps of the fine simulation $X_{p,\ell}$ guarantees the absence of collisions in the coarse simulation $X_{p,\ell-1}$.
\end{lemma}
\begin{proof}
By the definition of $\pncr$ and the fact that $\Delta t_{\ell-1} = M\Delta t_\ell$, checking this statement corresponds to verifying that 
\begin{equation}
\PT_{p,\ell-1}^n > \pncr = \frac{\epsilon^2}{\epsilon^2 +M\Delta t_\ell} 
\quad \text{implies} \quad
 \PT^{n,\text{max}}_{p,\ell} > \pncf =  \frac{\epsilon^2}{\epsilon^2 +\Delta t_\ell}.
\end{equation}
Using \eqref{eq:collisioncondition_level_coarse}, this is equivalent to showing that
\begin{equation}
\label{eq:to_check_collisions}
 \left(\frac{\epsilon^2}{\epsilon^2 +\Delta t_\ell}\right)^M \leq \frac{\epsilon^2}{\epsilon^2 +M\Delta t_\ell}.
 \end{equation}
We do this by rewriting \eqref{eq:to_check_collisions} as 
 \begin{align}
   \epsilon^{2M-2}(\epsilon^2 + M\Delta t_\ell) &\leq (\epsilon^2+\Delta t_\ell)^M
 \end{align}
and using the binomial theorem to obtain
\begin{align} 
\label{eq:collision_proof_inequality}
\epsilon^{2M} + M\epsilon^{2(M-1)}\Delta t_\ell &\leq \epsilon^{2M} + M\epsilon^{2(M-1)}\Delta t_\ell + \dotsb + \Delta t_\ell^M.
\end{align}
Since $\epsilon$, $M$ and $\Delta t_\ell$ are all positive, the statement is proved.
\end{proof}
\begin{remark}
\label{rem:collision_probabilities}
Note that the inequality in \eqref{eq:collision_proof_inequality} can be made strict, so it is possible that a collision is performed in the simulation at level $\ell$ without a collision occurring at level $\ell-1$. This is expected behavior, given the $\Delta t$ dependence of the collision rate in \eqref{eq:GTmod}.
\end{remark}

Lemma~\ref{lem:collision_likelihood} implies that we never select a new velocity for the simulation at level $\ell-1$ without already selecting a new velocity for in least one of the sub-steps at level $\ell$. If a collision is performed in the simulation at level $\ell-1$, we want to correlate it with the fine simulation velocity at the end of the  time interval. We consider the transformation of $\mathcal{M}(v)$ given in \eqref{eq:velocity_decomposition}. If we use the same $\VDR$ to sample from $\mathcal{M}_{\Delta t_\ell}(v)$ and $\mathcal{M}_{\Delta t_{\ell-1}}(v)$, we can expect the resulting velocities to be correlated. In each of the $M$ fine sub-steps containing a collision we draw a $\VDR^{n,m,*}_{p,\ell}$ for use in the subsequent time step. Given that the $\VDR^{n,m,*}_{p,\ell}$ are i.i.d., we can select one freely to use as $\VDR^{n,*}_{p,{\ell-1}}$. To maximise the correlation of the velocities at the end of the time interval, we choose to take the last generated $\VDR^{n,m,*}_{p,\ell}$, i.e., we choose
\begin{equation}
\label{eq:betacorr}
V^{n+1}_{p,{\ell-1}} = \CV_{\Delta t_{\ell-1}} \VDR^{n+1}_{p,{\ell-1}}, \quad \VDR^{n+1}_{p,{\ell-1}} = \VDR^{n,*}_{p,{\ell-1}} = \VDR^{n,M-1,*}_{p,\ell},
\end{equation}
giving line 8 of Algorithm~\ref{alg:correlation}.

\subsubsection{Numerical illustration}

In Figure \ref{fig:corrpathtransp}, we show two particle trajectories, containing a series of increments $\Delta \U^{n,m}_{p,\ell}$ and $\Delta \U^{n}_{p,\ell-1}$ for the two speed model \eqref{eq:GTmod}, using coupled uniformly distributed numbers as described in \eqref{eq:alphacorr} and \eqref{eq:betacorr}, with $\epsilon=0.5$, $\Delta t_\ell=0.2$ and $M=5$. In this model, we sample from $\mathcal{M}_{\Delta t_\ell}(v)$ by drawing $\VDR^{n,m}_{p,\ell}$ from $\{-1,1\}$ and multiplying it with $\CV_{\Delta t_\ell}$. 
\begin{figure}
\centering
\begin{minipage}{0.05\textwidth}\end{minipage}
\begin{minipage}{0.95\textwidth}
%
%
\definecolor{mycolor1}{rgb}{0.12156862745098,0.466666666666667,0.705882352941177}
\definecolor{mycolor2}{rgb}{1,0.498039215686275,0.0549019607843137}
\begin{tikzpicture}

\begin{axis}[%
width=\figurewidth,
height=\figureheight,
at={(0\figurewidth,0\figureheight)},
xlabel={Time},
ylabel={Position},
scale only axis,
xmin=0,
xmax=10,
ymin=0,
ymax=3.5,
axis background/.style={fill=white},
legend entries={{$X^{n}_{p,\Delta t_{\ell}}$},{$X^{n}_{p,\Delta t_{\ell-1}}$}},
legend cell align={left},
legend style={at={(0.03,0.92)}, anchor=north west, draw=white!80.0!black}
]
\addlegendimage{mycolor1, mark=square*, mark size=1}
\addlegendimage{mycolor2, mark=*, mark size=1}
\addplot [semithick, color=mycolor1, forget plot]
  table[row sep=crcr]{%
0	0\\
0.2	0.222222222222222\\
0.4	0.444444444444444\\
0.6	0.666666666666667\\
0.8	0.444444444444445\\
1	0.666666666666667\\
1.2	0.888888888888889\\
1.4	1.11111111111111\\
1.6	0.888888888888889\\
1.8	0.666666666666667\\
2	0.444444444444445\\
2.2	0.666666666666667\\
2.4	0.888888888888889\\
2.6	1.11111111111111\\
2.8	1.33333333333333\\
3	1.55555555555556\\
3.2	1.77777777777778\\
3.4	2\\
3.6	1.77777777777778\\
3.8	1.55555555555556\\
4	1.33333333333333\\
4.2	1.55555555555556\\
4.4	1.77777777777778\\
4.6	2\\
4.8	1.77777777777778\\
5	1.55555555555556\\
5.2	1.33333333333333\\
5.4	1.11111111111111\\
5.6	0.888888888888889\\
5.8	0.666666666666667\\
6	0.888888888888889\\
6.2	1.11111111111111\\
6.4	1.33333333333333\\
6.6	1.11111111111111\\
6.8	1.33333333333333\\
7	1.55555555555556\\
7.2	1.77777777777778\\
7.4	2\\
7.6	2.22222222222222\\
7.8	2.44444444444444\\
8	2.66666666666667\\
8.2	2.88888888888889\\
8.4	3.11111111111111\\
8.6	2.88888888888889\\
8.8	2.66666666666667\\
9	2.44444444444444\\
9.2	2.66666666666667\\
9.4	2.88888888888889\\
9.6	2.66666666666667\\
9.8	2.44444444444444\\
10	2.22222222222222\\
};
\addplot [semithick, color=mycolor2, forget plot]
  table[row sep=crcr]{%
0	0\\
1	0.4\\
2	0.8\\
3	1.2\\
4	1.6\\
5	2\\
6	1.6\\
7	2\\
8	2.4\\
9	2.8\\
10	3.2\\
};
\addplot[only marks, mark=asterisk, mark options={}, mark size=2.5pt, draw=mycolor1] table[row sep=crcr]{%
x	y\\
0.2	0.222222222222222\\
0.4	0.444444444444444\\
0.8	0.444444444444445\\
1.2	0.888888888888889\\
1.4	1.11111111111111\\
1.8	0.666666666666667\\
2	0.444444444444445\\
2.8	1.33333333333333\\
3	1.55555555555556\\
3.2	1.77777777777778\\
3.4	2\\
3.8	1.55555555555556\\
4	1.33333333333333\\
4.6	2\\
5.4	1.11111111111111\\
5.6	0.888888888888889\\
5.8	0.666666666666667\\
6.2	1.11111111111111\\
6.4	1.33333333333333\\
6.6	1.11111111111111\\
7	1.55555555555556\\
7.4	2\\
8	2.66666666666667\\
8.4	3.11111111111111\\
8.8	2.66666666666667\\
9	2.44444444444444\\
9.4	2.88888888888889\\
9.8	2.44444444444444\\
};
\addplot[only marks, mark=asterisk, mark options={}, mark size=2.5pt, draw=mycolor2] table[row sep=crcr]{%
x	y\\
1	0.4\\
2	0.8\\
3	1.2\\
4	1.6\\
5	2\\
6	1.6\\
7	2\\
9	2.8\\
10	3.2\\
};
\addplot[only marks, mark=square*, mark options={}, mark size=1pt, draw=mycolor1, fill=mycolor1] table[row sep=crcr]{%
x	y\\
0	0\\
0.6	0.666666666666667\\
1	0.666666666666667\\
1.6	0.888888888888889\\
2.2	0.666666666666667\\
2.4	0.888888888888889\\
2.6	1.11111111111111\\
3.6	1.77777777777778\\
4.2	1.55555555555556\\
4.4	1.77777777777778\\
4.8	1.77777777777778\\
5	1.55555555555556\\
5.2	1.33333333333333\\
6	0.888888888888889\\
6.8	1.33333333333333\\
7.2	1.77777777777778\\
7.6	2.22222222222222\\
7.8	2.44444444444444\\
8.2	2.88888888888889\\
8.6	2.88888888888889\\
9.2	2.66666666666667\\
9.6	2.66666666666667\\
10	2.22222222222222\\
};
\addplot[only marks, mark=*, mark options={}, mark size=1pt, draw=mycolor2, fill=mycolor2] table[row sep=crcr]{%
0	0\\
8	2.4\\
};
\end{axis}
\end{tikzpicture}%
\end{minipage}
\caption{Correlated transport steps for the two speed model with $\epsilon=0.5$, $\Delta t_\ell=0.2$ and $\Delta t_{\ell-1}=1$. Stars mark collisions.\label{fig:corrpathtransp}}
\end{figure}
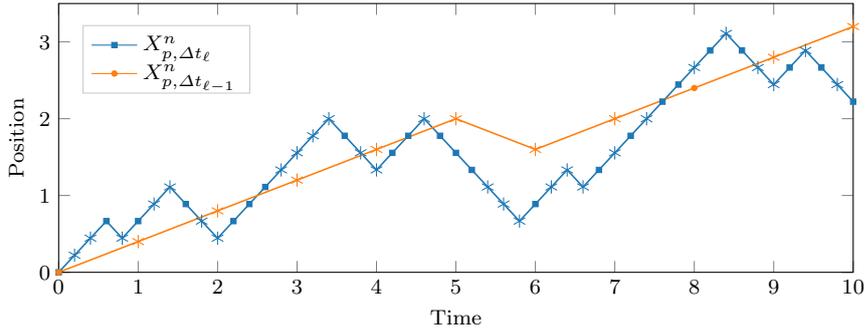
\begin{figure}
\centering
%
%
\definecolor{mycolor1}{rgb}{0.12156862745098,0.466666666666667,0.705882352941177}
\definecolor{mycolor2}{rgb}{1,0.498039215686275,0.0549019607843137}
\begin{tikzpicture}

\begin{axis}[%
width=\figurewidth,
height=\figureheight,
at={(0\figurewidth,0\figureheight)},
xlabel={Time},
ylabel={Position},
scale only axis,
xmin=0,
xmax=10,
ymin=-2,
ymax=3,
axis background/.style={fill=white},
legend entries={{$X^{n}_{p,\Delta t_{\ell}}$},{$X^{n}_{p,\Delta t_{\ell-1}}$}},
legend cell align={left},
legend style={at={(0.03,0.08)}, anchor=south west, draw=white!80.0!black}
]
\addlegendimage{mycolor1, mark=square*, mark size=1}
\addlegendimage{mycolor2, mark=*, mark size=1}
\addplot [semithick, color=mycolor1, forget plot]
  table[row sep=crcr]{%
0	0\\
0.2	0.133799314885418\\
0.4	0.619624937050425\\
0.6	0.919102556675746\\
0.8	0.262692214695232\\
1	0.885141502401421\\
1.2	1.23683239758445\\
1.4	1.51604937975543\\
1.6	1.51107408770838\\
1.8	1.39907044932444\\
2	0.779882970987714\\
2.2	0.9336576203951\\
2.4	1.09429780164257\\
2.6	1.09220433161973\\
2.8	2.02366370310759\\
3	1.87664601230106\\
3.2	1.89487389790775\\
3.4	1.81688864288628\\
3.6	1.09957503078938\\
3.8	0.796297511005021\\
4	0.458517133391721\\
4.2	1.32587457305407\\
4.4	1.44309874459892\\
4.6	1.21660919332362\\
4.8	1.67046393015192\\
5	1.96882814575483\\
5.2	1.64978691072399\\
5.4	0.792511997229651\\
5.6	0.382818226868795\\
5.8	0.0948454908364456\\
6	0.433468309337253\\
6.2	0.545574269867864\\
6.4	0.954759586650851\\
6.6	0.897774471580306\\
6.8	0.592664164855734\\
7	0.415190967424449\\
7.2	0.324935423689017\\
7.4	0.333042966619759\\
7.6	0.420098687119435\\
7.8	0.647578318800531\\
8	-0.407412770365323\\
8.2	-0.377884840036178\\
8.4	0.368199627444877\\
8.6	-0.303783395321323\\
8.8	-0.132311256278836\\
9	-0.206825174970892\\
9.2	0.00316714343903346\\
9.4	0.302317954108386\\
9.6	-0.579789824175304\\
9.8	-0.83765702085484\\
10	-0.383596081394384\\
};
\addplot [semithick, color=mycolor2, forget plot]
  table[row sep=crcr]{%
0	0\\
1	0.693114750445607\\
2	1.25003800857387\\
3	1.63078805292564\\
4	0.426310905467504\\
5	2.55446336295379\\
6	0.988989175462115\\
7	0.470040357083178\\
8	-1.72431035366269\\
9	-0.757051457415612\\
10	-0.296072118339949\\
};
\addplot[only marks, mark=asterisk, mark options={}, mark size=2.5pt, draw=mycolor1] table[row sep=crcr]{%
x	y\\
0.2	0.133799314885418\\
0.4	0.619624937050425\\
0.8	0.262692214695232\\
1.2	1.23683239758445\\
1.4	1.51604937975543\\
1.8	1.39907044932444\\
2	0.779882970987714\\
2.8	2.02366370310759\\
3	1.87664601230106\\
3.2	1.89487389790775\\
3.4	1.81688864288628\\
3.8	0.796297511005021\\
4	0.458517133391721\\
4.6	1.21660919332362\\
5.4	0.792511997229651\\
5.6	0.382818226868795\\
5.8	0.0948454908364456\\
6.2	0.545574269867864\\
6.4	0.954759586650851\\
6.6	0.897774471580306\\
7	0.415190967424449\\
7.4	0.333042966619759\\
8	-0.407412770365323\\
8.4	0.368199627444877\\
8.8	-0.132311256278836\\
9	-0.206825174970892\\
9.4	0.302317954108386\\
9.8	-0.83765702085484\\
};
\addplot[only marks, mark=asterisk, mark options={}, mark size=2.5pt, draw=mycolor2] table[row sep=crcr]{%
x	y\\
1	0.693114750445607\\
2	1.25003800857387\\
3	1.63078805292564\\
4	0.426310905467504\\
5	2.55446336295379\\
6	0.988989175462115\\
7	0.470040357083178\\
9	-0.757051457415612\\
10	-0.296072118339949\\
};
\addplot[only marks, mark=square*, mark options={}, mark size=1pt, draw=mycolor1, fill=mycolor1] table[row sep=crcr]{%
x	y\\
0.6	0.919102556675746\\
1	0.885141502401421\\
1.6	1.51107408770838\\
2.2	0.9336576203951\\
2.4	1.09429780164257\\
2.6	1.09220433161973\\
3.6	1.09957503078938\\
4.2	1.32587457305407\\
4.4	1.44309874459892\\
4.8	1.67046393015192\\
5	1.96882814575483\\
5.2	1.64978691072399\\
6	0.433468309337253\\
6.8	0.592664164855734\\
7.2	0.324935423689017\\
7.6	0.420098687119435\\
7.8	0.647578318800531\\
8.2	-0.377884840036178\\
8.6	-0.303783395321323\\
9.2	0.00316714343903346\\
9.6	-0.579789824175304\\
10	-0.383596081394384\\
};
\addplot[only marks, mark=*, mark options={}, mark size=1pt, draw=mycolor2, fill=mycolor2] table[row sep=crcr]{%
x	y\\
0	0\\
8	-1.72431035366269\\
};
\end{axis}
\end{tikzpicture}%
\caption{Correlated paths steps for the two speed model with $\epsilon=0.5$, $\Delta t_\ell=0.2$ and $\Delta t_{\ell-1}=1$. Stars mark collisions.\label{fig:corrpath}}
\end{figure}
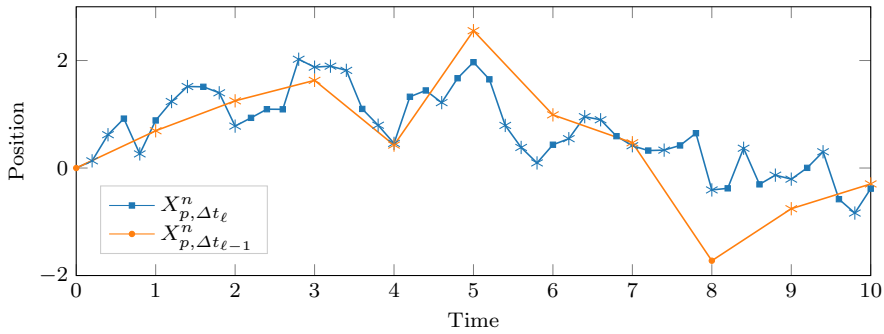
We observe that the paths have the same signs in the velocities at the end of a coarse time step whenever a collision has taken place in both simulations. Still, there is an observable difference between them. This is due to the paths having different characteristic velocities, which are a function of the time step size. There is also a probability of no collision taking place in the coarse simulation, while a collision takes place in the fine simulation, as per Remark~\ref{rem:collision_probabilities}. For instance, no collision occurs at $t=8$ in the coarse simulation, while a collision takes place at time $t=7.4$ and $t=8$ in the fine simulation. However, by coincidence, the new velocity generated at $t=8$ in the fine simulation has the same sign  as the velocity in the coarse simulation. This mismatch is also part of the bias we want to estimate. If we add Figure~\ref{fig:corrpathdiff} to Figure~\ref{fig:corrpathtransp}, we get the paths shown in Figure~\ref{fig:corrpath}.

\subsection{Variance of correlated paths\label{sec:correlation_variance}}

To conclude this section, we present a brief numerical result in which we compute the variance of 10 000 pairs trajectories as shown in Figures~\ref{fig:corrpathdiff}--\ref{fig:corrpath}. The results are shown in Figure~\ref{fig:weak_behavior}. We leave a detailed study of the variance structure resulting from our algorithm to the subsequent sections and limit ourselves to two observations here. We see that the variance of differences (green line with triangles) falls consistently below that of the fine simulation (blue line with squares), meaning that fewer samples are needed to estimate the difference than are needed to directly estimate the result of the fine simulation.

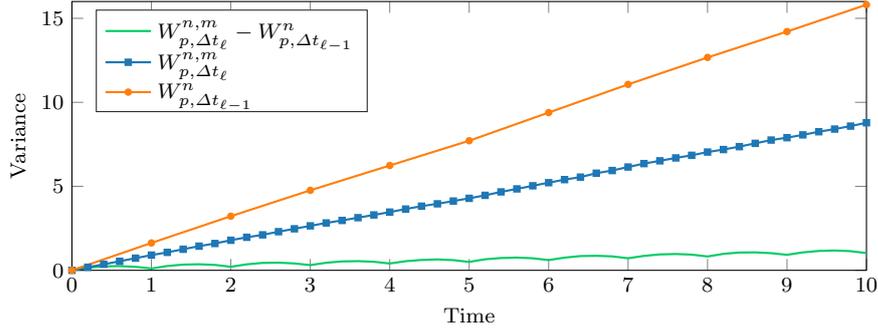
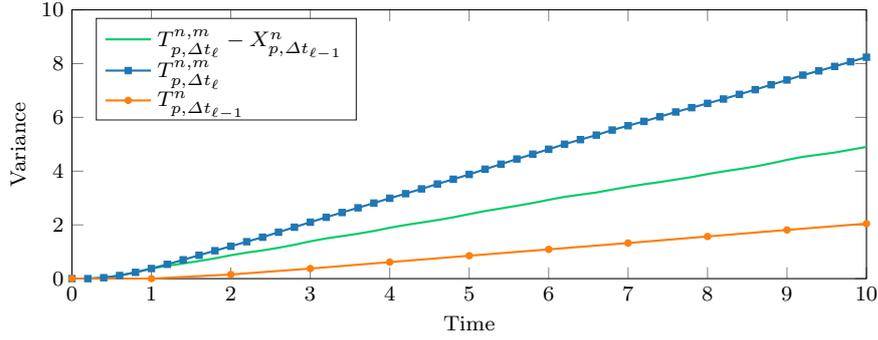
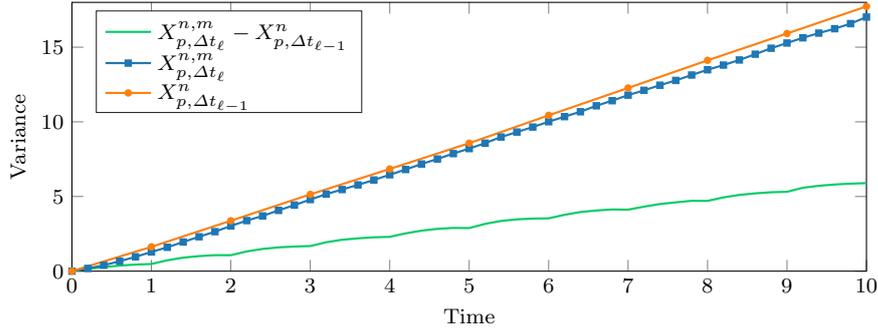
\begin{figure}
\subfloat[Brownian increments]{
%
%
\definecolor{color0}{rgb}{0.12156862745098,0.466666666666667,0.705882352941177}
\definecolor{color1}{rgb}{1,0.498039215686275,0.0549019607843137}
\definecolor{color2}{rgb}{0.01,0.8,0.4}
\begin{tikzpicture}

\begin{axis}[%
width=\figurewidth,
height=\figureheight,
at={(0.758in,0.481in)},
scale only axis,
xmin=0,
xmax=10,
ymin=0,
ymax=16,
xlabel=Time,
ylabel=Variance,
axis background/.style={fill=white},
legend pos=north west,
legend style={legend cell align=left, align=left, draw=white!15!black}
]
\addplot [thick, color=color2]
  table[row sep=crcr]{%
0	0\\
0.2	0.152447228334248\\
0.4	0.233811765328659\\
0.6	0.251930261294831\\
0.8	0.206775052085983\\
1	0.105567258471065\\
1.2	0.256143443206608\\
1.4	0.339632790108909\\
1.6	0.35926709280429\\
1.8	0.322192479159389\\
2	0.208805871406721\\
2.2	0.35568199997303\\
2.4	0.440746036420459\\
2.6	0.455315292794048\\
2.8	0.408173211160524\\
3	0.308784290200583\\
3.2	0.450508505010197\\
3.4	0.528085972658906\\
3.6	0.545177732013622\\
3.8	0.508879567966469\\
4	0.404479235281613\\
4.2	0.541598113944266\\
4.4	0.615824487939951\\
4.6	0.648579927372538\\
4.8	0.604572597496466\\
5	0.500149577822666\\
5.2	0.654824939904366\\
5.4	0.730749793288704\\
5.6	0.759317584369576\\
5.8	0.714024457938759\\
6	0.608863120244818\\
6.2	0.754747223198254\\
6.4	0.847228519544404\\
6.6	0.862598588438924\\
6.8	0.82868604756916\\
7	0.717588188350507\\
7.2	0.857465160275337\\
7.4	0.946030632187863\\
7.6	0.969828985626744\\
7.8	0.934160972587029\\
8	0.821609835948781\\
8.2	0.980923273384621\\
8.4	1.05864966531562\\
8.6	1.06746895443869\\
8.8	1.01980892993257\\
9	0.921496330251598\\
9.2	1.0666378511141\\
9.4	1.15300114037085\\
9.6	1.17906367115872\\
9.8	1.13941492768793\\
10	1.02518478466648\\
};
\addlegendentry{$W_{p,\Delta t_\ell}^{n,m} - W_{p,\Delta t_{\ell-1}}^n$}

\addplot [thick, mark=square*, mark size=1pt, color=color0]
  table[row sep=crcr]{%
0	0\\
0.2	0.183261373820504\\
0.4	0.356819381730182\\
0.6	0.538951197466752\\
0.8	0.722780857057583\\
1	0.904460942322639\\
1.2	1.06893690368744\\
1.4	1.25332017666848\\
1.6	1.43706623339755\\
1.8	1.59588506917166\\
2	1.78897091721661\\
2.2	1.96552226430154\\
2.4	2.11011881982791\\
2.6	2.2984596229575\\
2.8	2.47128199644557\\
3	2.64554876326353\\
3.2	2.8191170755066\\
3.4	2.96914542207603\\
3.6	3.13270747256859\\
3.8	3.29468131013462\\
4	3.46542740231359\\
4.2	3.64529298041566\\
4.4	3.81827499797727\\
4.6	3.95790565047144\\
4.8	4.12208278087362\\
5	4.28509525596661\\
5.2	4.46299184840521\\
5.4	4.66741187310521\\
5.6	4.84639381558736\\
5.8	5.03076694471673\\
6	5.21651238705866\\
6.2	5.39789709758873\\
6.4	5.54697775760128\\
6.6	5.77999744274393\\
6.8	5.93990255415249\\
7	6.14802826591344\\
7.2	6.35201964551421\\
7.4	6.51736032665481\\
7.6	6.68983305542967\\
7.8	6.844890264348\\
8	7.0392469900833\\
8.2	7.18905047018415\\
8.4	7.361512560167\\
8.6	7.55770386058385\\
8.8	7.7529751176222\\
9	7.89503726133674\\
9.2	8.06217697370002\\
9.4	8.24886409142626\\
9.6	8.40706572982076\\
9.8	8.58393267807413\\
10	8.78340131043978\\
};
\addlegendentry{$W_{p,\Delta t_\ell}^{n,m}$}

\addplot [thick, mark=*, mark size=1pt, color=color1]
  table[row sep=crcr]{%
0	0\\
1	1.62802969618075\\
2	3.22014765098989\\
3	4.76198777387436\\
4	6.23776932416446\\
5	7.71317146073993\\
6	9.38972229670558\\
7	11.0664508786442\\
8	12.6706445821499\\
9	14.2110670704061\\
10	15.8101223587916\\
};
\addlegendentry{$W_{p,\Delta t_{\ell-1}}^n$}

\end{axis}
\end{tikzpicture}%
}\\
\subfloat[Transport increments]{
%
%
\definecolor{color0}{rgb}{0.12156862745098,0.466666666666667,0.705882352941177}
\definecolor{color1}{rgb}{1,0.498039215686275,0.0549019607843137}
\definecolor{color2}{rgb}{0.01,0.8,0.4}
\begin{tikzpicture}

\begin{axis}[%
width=\figurewidth,
height=\figureheight,
at={(0.758in,0.481in)},
scale only axis,
xmin=0,
xmax=10,
ymin=0,
ymax=10,
xlabel=Time,
ylabel=Variance,
axis background/.style={fill=white},
legend pos=north west,
legend style={legend cell align=left, align=left, draw=white!15!black}
]
\addplot [thick, color=color2]
  table[row sep=crcr]{%
0	0\\
0.2	3.02885853828187e-28\\
0.4	0.0341750224405229\\
0.6	0.116612895857489\\
0.8	0.234098400457312\\
1	0.37648406618442\\
1.2	0.475955621952355\\
1.4	0.563689737151522\\
1.6	0.655543870317907\\
1.8	0.755221784222828\\
2	0.870466811619455\\
2.2	0.968496376635062\\
2.4	1.05043891263121\\
2.6	1.14908205084526\\
2.8	1.26053334624998\\
3	1.38765530154992\\
3.2	1.49654360983752\\
3.4	1.58247990875141\\
3.6	1.67127169281866\\
3.8	1.77263521939353\\
4	1.89370391095409\\
4.2	1.99627834902326\\
4.4	2.08735347456796\\
4.6	2.17872517919721\\
4.8	2.28175122896902\\
5	2.4000376307063\\
5.2	2.51506461316829\\
5.4	2.61052439872653\\
5.6	2.71024531763125\\
5.8	2.81378990596817\\
6	2.93091919989031\\
6.2	3.04285584575128\\
6.4	3.12393544054768\\
6.6	3.20403976949522\\
6.8	3.312846733401\\
7	3.41479291435304\\
7.2	3.50994153155877\\
7.4	3.59225715283904\\
7.6	3.68453949508407\\
7.8	3.7772297193879\\
8	3.89020983172884\\
8.2	3.98790018466858\\
8.4	4.07661924568434\\
8.6	4.17229372594622\\
8.8	4.28879035336146\\
9	4.41771263546107\\
9.2	4.5317773959685\\
9.4	4.61089565840254\\
9.6	4.69106429723594\\
9.8	4.79531705896402\\
10	4.90156324187488\\
};
\addlegendentry{$T_{p,\Delta t_\ell}^{n,m} - X_{p,\Delta t_{\ell-1}}^n$}

\addplot [thick, mark=square*, mark size=1pt, color=color0]  table[row sep=crcr]{%
0	0\\
0.2	1.41665235412454e-27\\
0.4	0.0341750224405147\\
0.6	0.116612895857478\\
0.8	0.234098400457296\\
1	0.376484066184426\\
1.2	0.533022607939814\\
1.4	0.699668189041153\\
1.6	0.870455243055156\\
1.8	1.03756492982627\\
2	1.20562480741909\\
2.2	1.37288779544625\\
2.4	1.54130787696064\\
2.6	1.72710302338871\\
2.8	1.9134942630066\\
3	2.09903847471158\\
3.2	2.28276243723133\\
3.4	2.46055143785987\\
3.6	2.63500228738915\\
3.8	2.81156949077625\\
4	2.99075010908513\\
4.2	3.16042349914006\\
4.4	3.33878471501474\\
4.6	3.51996919000531\\
4.8	3.69772700726831\\
5	3.87693090296679\\
5.2	4.07039933622997\\
5.4	4.2581430982605\\
5.6	4.44962129003008\\
5.8	4.63080663424363\\
6	4.81411789771595\\
6.2	5.00143478841692\\
6.4	5.17175471423713\\
6.6	5.33964144858911\\
6.8	5.52771587084661\\
7	5.69029669584224\\
7.2	5.85043509980624\\
7.4	6.02287270702388\\
7.6	6.20042799341667\\
7.8	6.36074229645149\\
8	6.52143229335234\\
8.2	6.68285613746528\\
8.4	6.85542118458772\\
8.6	7.0309153878347\\
8.8	7.21413359903898\\
9	7.39092389436506\\
9.2	7.57129120516999\\
9.4	7.73048002874346\\
9.6	7.89455648577182\\
9.8	8.07042700121878\\
10	8.2387129596914\\
};
\addlegendentry{$T_{p,\Delta t_\ell}^{n,m}$}

\addplot [thick, mark=*, mark size=1pt, color=color1]  table[row sep=crcr]{%
0	0\\
1	4.03325748846496e-27\\
2	0.153868826882686\\
3	0.37387853345335\\
4	0.613982063806294\\
5	0.852154565856553\\
6	1.08865199479955\\
7	1.32397943794371\\
8	1.56807482988307\\
9	1.81115567556763\\
10	2.04600891449119\\
};
\addlegendentry{$T_{p,\Delta t_{\ell-1}}^n$}

\end{axis}
\end{tikzpicture}%
}\\
\subfloat[Sum of Brownian and transport increments]{
%
%
\definecolor{color0}{rgb}{0.12156862745098,0.466666666666667,0.705882352941177}
\definecolor{color1}{rgb}{1,0.498039215686275,0.0549019607843137}
\definecolor{color2}{rgb}{0.01,0.8,0.4}
\begin{tikzpicture}

\begin{axis}[%
width=\figurewidth,
height=\figureheight,
at={(0.758in,0.481in)},
scale only axis,
xmin=0,
xmax=10,
ymin=0,
ymax=18,
xlabel=Time,
ylabel=Variance,
axis background/.style={fill=white},
legend pos=north west,
legend style={legend cell align=left, align=left, draw=white!15!black}
]
\addplot [thick, color=color2]
  table[row sep=crcr]{%
0	0\\
0.2	0.152447228334248\\
0.4	0.26802627590943\\
0.6	0.373239538977247\\
0.8	0.440106386320373\\
1	0.480742193821827\\
1.2	0.725597902167047\\
1.4	0.904901614057616\\
1.6	1.00625950086439\\
1.8	1.07068333962252\\
2	1.06903699995239\\
2.2	1.31933571049813\\
2.4	1.48845766406291\\
2.6	1.59129826039059\\
2.8	1.64841136023838\\
3	1.68055717707864\\
3.2	1.93660976349563\\
3.4	2.10232631413347\\
3.6	2.19602852364108\\
3.8	2.26802531528896\\
4	2.29652183174515\\
4.2	2.52932486520716\\
4.4	2.71183117449191\\
4.6	2.83279742025656\\
4.8	2.89475560126787\\
5	2.88534139863846\\
5.2	3.1660969726512\\
5.4	3.36104097319934\\
5.6	3.46935895732345\\
5.8	3.52046368413509\\
6	3.53663626511506\\
6.2	3.77453594228998\\
6.4	3.95909556940934\\
6.6	4.06005127842469\\
6.8	4.13279808359038\\
7	4.11871690245999\\
7.2	4.31963091473123\\
7.4	4.49529861190758\\
7.6	4.6050586627492\\
7.8	4.70992643210039\\
8	4.7061118089094\\
8.2	4.94101445886504\\
8.4	5.10331375656229\\
8.6	5.20962344938997\\
8.8	5.28359550004692\\
9	5.31441086928399\\
9.2	5.58531139244089\\
9.4	5.71604368554934\\
9.6	5.80288934145808\\
9.8	5.85961217155657\\
10	5.89581350686457\\
};
\addlegendentry{$X_{p,\Delta t_\ell}^{n,m} - X_{p,\Delta t_{\ell-1}}^n$}

\addplot [thick, mark=square*, mark size=1pt, color=color0]
  table[row sep=crcr]{%
0	0\\
0.2	0.183261373820504\\
0.4	0.39271344127915\\
0.6	0.663815765817435\\
0.8	0.959684866651143\\
1	1.28477690120066\\
1.2	1.59858132540966\\
1.4	1.95351075170576\\
1.6	2.30281944309284\\
1.8	2.64541014957943\\
2	3.02632209050028\\
2.2	3.38632026563206\\
2.4	3.6983543505312\\
2.6	4.06918794249865\\
2.8	4.42341604933522\\
3	4.79439310943461\\
3.2	5.15843819740803\\
3.4	5.46668252005563\\
3.6	5.77634666758591\\
3.8	6.10484756885381\\
4	6.45415068071879\\
4.2	6.80694241976306\\
4.4	7.17416911572521\\
4.6	7.51102658911414\\
4.8	7.87284168042072\\
5	8.20793353109405\\
5.2	8.57010653178433\\
5.4	8.98250293228844\\
5.6	9.31410704642815\\
5.8	9.6502475285662\\
6	10.0063732998532\\
6.2	10.3482721555359\\
6.4	10.6756546374756\\
6.6	11.074095641058\\
6.8	11.4163025746782\\
7	11.7869689697317\\
7.2	12.109016656704\\
7.4	12.4509596624048\\
7.6	12.7735571975838\\
7.8	13.1345113198703\\
8	13.4878727412773\\
8.2	13.7996311586797\\
8.4	14.1423265233987\\
8.6	14.5289434937871\\
8.8	14.9281655796084\\
9	15.2777243122409\\
9.2	15.6272367504533\\
9.4	15.9415771579267\\
9.6	16.2352347175756\\
9.8	16.5747423109727\\
10	17.0165869589853\\
};
\addlegendentry{$X_{p,\Delta t_\ell}^{n,m}$}

\addplot [thick, mark=*, mark size=1pt, color=color1]
  table[row sep=crcr]{%
0	0\\
1	1.62802969618075\\
2	3.37638581504623\\
3	5.14031684373332\\
4	6.84250804256023\\
5	8.56861697716249\\
6	10.4334754127175\\
7	12.2678690120566\\
8	14.1186241331612\\
9	15.9137884718751\\
10	17.7272341713015\\
};
\addlegendentry{$X_{p,\Delta t_{\ell-1}}^n$}

\end{axis}
\end{tikzpicture}%
}
\caption{Variance as a function of time for (differences of) coarse and fine simulations with $\epsilon=0.5$, $\Delta t_\ell=0.2$ and $\Delta t_{\ell-1}=1$. (A linear interpolation of the coarse simulation is used when computing the difference.) \label{fig:weak_behavior}}
\end{figure}

We also see that the variance of differences of transport processes lies above that of the coarse transport process. This is due to performing simulations with time step sizes with order of magnitude as $\epsilon^2$. In Section~\ref{sec:experiments}, we will show that this is the parameter regime in which the asymptotic-preserving model switches between simulating an approximation of the original kinetic equation~\eqref{eq:kineticdimless} and the limiting diffusion equation. As the simulation with time step $\Delta t_{\ell-1}$ contains much less transport behavior than the simulation with time step $\Delta t_\ell$, the variance of their difference lies between their individual variances, even though the correlation is good.

\section{Analysis of the multilevel scheme}
\label{sec:analysis}
As is clear from Theorem~\ref{thm:giles}, the behavior of the mean and variance of the estimators~\eqref{eq:MClestimator} as a function of the level $\ell$ is of key importance in the analysis of a multilevel Monte Carlo method. In this section, we  present and prove lemmas concerning the mean and variance of the difference in particle position and velocity for the coupled simulations~\eqref{eq:coupled_simulation} as a function of the time step $\Delta t_\ell$ and the refinement factor $M$ at some fixed time $t^*$. When considering the particle position, we look separately at the Brownian increments (Section~\ref{sec:brownian_analysis}), the transport increments (Section~\ref{sec:transport_analysis}) and particle velocities (Section~\ref{sec:velocity_exp_var}).
In Section~\ref{sec:proof}, we make use of these lemmas to prove a bound on the computational complexity of our multilevel Monte Carlo scheme.

For convenience, we introduce an additional notation for the difference of an arbitrary pair of coupled increments at arbitrary levels $\ell$ and $\ell-1$ at time $t_n=n\Delta t_{\ell-1}$:
\begin{equation}
\Delta_{W,n} = \sum_{m=0}^{M-1} \Delta W^{n,m}_{p,\ell} - \Delta W^{n}_{p,\ell-1},\qquad
\Delta_{\U,n} = \sum_{m=0}^{M-1} \Delta \U^{n,m}_{p,\ell} - \Delta \U^{n}_{p,\ell-1}.
\end{equation}

\subsection{Brownian increments}
\label{sec:brownian_analysis}

We first show that the Brownian increments have expected value 0.

\begin{lemma}
\label{lem:brownian_expectation}
Given a simulation interval $\Delta t_{\ell-1}$, an independent fine simulation containing Brownian increments $\Delta W^{n,m}_{p,\ell}$ with time step sizes $\Delta t_{\ell}$ and a coarse simulation containing Brownian increments $\Delta W^{n}_{p,\ell-1}$ with time step size $\Delta t_{\ell-1}$, generated by Algorithm~\ref{alg:correlation}. Then, $$\E\left[\Delta_{W,n}\right] = \E\left[\sum_{m=0}^{M-1} \Delta W^{n,m}_{p,\ell} - \Delta W^{n}_{p,\ell-1}\right] = 0.$$
\end{lemma}
\begin{proof}
By the martingale property of the increments and the fact that the coupling preserves the statistics of the stochastic process at level $\ell-1$, we have
\begin{align}
\E\left[\Delta W^{n,m}_{p,\ell}\right] &= \sqrt{2\Delta t_{\ell}}\sqrt{D_{\Delta t_{\ell}}} \E\left[\xi^{n,m}_{p,\ell}\right] = 0, \label{eq:browinian_martingale_fine}\\
\E\left[\Delta W^{n}_{p,\ell-1}\right]&= \sqrt{2\Delta t_{\ell-1}}\sqrt{D_{\Delta t_{\ell-1}}}\frac{1}{\sqrt{M}} \E\left[\sum_{m=0}^{M-1}\xi^{n,m}_{p,\ell}\right] = 0.\label{eq:browinian_martingale_coarse}
\end{align}
The proof then directly follows from
\begin{align}
\E\left[\Delta_{W,n}\right]= \E\left[\sum_{m=0}^{M-1} \Delta W^{n,m}_{p,\ell} - \Delta W^{n}_{p,\ell-1}\right] = \sum_{m=0}^{M-1} \E\left[\Delta W^{n,m}_{p,\ell}\right] - \E\left[\Delta W^{n}_{p,\ell-1}\right].
\end{align}
\end{proof}

To prove convergence of the scheme in Section~\ref{sec:proof}, we require the following convergence result on the variance of the summed Brownian increments:
\begin{lemma}
\label{lem:brownian_variance}
Given a fixed ratio $M$ between two time step sizes $\Delta t_{\ell-1} = M \Delta t_\ell$, a fixed simulation horizon $t^*=NM\Delta t_\ell$, an independent fine simulation containing Brownian increments $\Delta W^{n,m}_{p,\ell}$ with time step sizes $\Delta t_{\ell}$ and a coarse simulation containing Brownian increments $\Delta W^{n}_{p,\ell-1}$, generated by Algorithm~\ref{alg:correlation} with time step size $\Delta t_{\ell-1}$. Then, the variance $$\V \left[ \sum_{n=0}^{N-1} \Delta_{W,n} \right] = \V\left[\sum_{m=0}^{M-1} \Delta W^{n,m}_{p,\ell} - \Delta W^{n}_{p,\ell-1}\right]$$ converges $\mathcal{O}(\Delta t_\ell)$ to 0, asymptotically for $\Delta t_\ell \to 0$.
\end{lemma}
\begin{proof}
To calculate the variance of the difference $\Delta_{W,n}$ for each Brownian increment at level $\ell$, we write
\begin{align}
 \V\left[\Delta_{W,n}\right]&= \V\left[\sum_{m=0}^{M-1} \Delta W^{n,m}_{p,\ell} - \Delta W^{n}_{p,\ell-1}\right] \\
&= \V\left[\sum_{m=0}^{M-1} \Delta W^{n,m}_{p,\ell}\right] + \V\left[\Delta W^{n}_{p,\ell-1}\right] - 2\Cov\left(\sum_{m=0}^{M-1} \Delta W^{n,m}_{p,\ell},\Delta W^{n}_{p,\ell-1}\right) \\
&= \sum_{m=0}^{M-1} \V\left[\Delta W^{n,m}_{p,\ell}\right] + \V\left[\Delta W^{n}_{p,\ell-1}\right] - 2 \sum_{m=0}^{M-1} \Cov\left(\Delta W^{n,m}_{p,\ell},\Delta W^{n}_{p,\ell-1}\right), \label{eq:var_brown_new}
\end{align}
in which we used independence of the random variables at level $\ell$.

We thus need the variances of individual increments at levels $\ell$ and $\ell-1$, which follow trivially from \eqref{eq:brownian_increments}:
\begin{align}
\V\left[\Delta W^{n,m}_{p,\ell}\right] &= 2 \Delta t_\ell D_{\Delta t_\ell} \V\left[\xi^{n,m}_{p,\ell}\right] = 2 \Delta t_\ell D_{\Delta t_\ell},\label{eq:brownian_variance_fine}\\
\V\left[\Delta W^{n}_{p,\ell-1}\right] &= 2 \Delta t_{\ell-1} D_{\Delta t_{\ell-1}}  \frac{1}{M} \sum_{m=0}^{M-1}\V\left[\xi^{n,m}_{p,\ell}\right] = 2 \Delta t_{\ell-1} D_{\Delta t_{\ell-1}}.\label{eq:brownian_variance_coarse}
\end{align}
The covariance between the  increment $\Delta W^{n}_{p,\ell-1}$ at level $\ell-1$ and each sub-increment $\Delta W^{n,m}_{p,\ell}$ at level $\ell$ can be computed using \eqref{eq:browinian_martingale_fine} and \eqref{eq:browinian_martingale_coarse}:
\begin{align}
\Cov(\Delta W^{n,m}_{p,\ell}, \Delta W^{n}_{p,\ell-1}) &= \E\left[\Delta W^{n,m}_{p,\ell}\Delta W^{n}_{p,\ell-1}\right]\\	
&= \E \left[ 2\sqrt{\frac{\Delta t_{\ell-1} \Delta t_{\ell} D_{\Delta t_{\ell-1}} D_{\Delta t_{\ell}}}{M}} \xi^{n,m}_{p,\ell} \sum_{m^\prime=0}^{M-1} \xi^{n,m^\prime}_{p,\ell} \right] \\
&= 2\sqrt{\frac{\Delta t_{\ell-1} \Delta t_{\ell} D_{\Delta t_{\ell-1}} D_{\Delta t_{\ell}}}{M}} \E \left[ {\left(\xi^{n,m}_{p,\ell}\right)}^2 \right] \\
&= 2\Delta t_{\ell} \sqrt{D_{\Delta t_{\ell-1}} D_{\Delta t_{\ell}}}.\label{eq:brownian_covariance_coarse}
\end{align}

Using \eqref{eq:brownian_variance_fine}, \eqref{eq:brownian_variance_coarse} and \eqref{eq:brownian_covariance_coarse}, we can elaborate~\eqref{eq:var_brown_new} to obtain
\begin{align}
 \V\left[\Delta_{W,n}\right]&= 2 M \Delta t_\ell D_{\Delta t_\ell}  + 2 \Delta t_{\ell-1} D_{\Delta t_{\ell-1}} - 4 M \Delta t_{\ell} \sqrt{D_{\Delta t_{\ell-1}} D_{\Delta t_{\ell}}} \\
&= 2 \Delta t_{\ell-1} \left(D_{\Delta t_\ell} + D_{\Delta t_{\ell-1}} - 2 \sqrt{D_{\Delta t_{\ell-1}} D_{\Delta t_{\ell}}}\right).\label{eq:var_brownian_result_to_expand}
\end{align}
Note that \eqref{eq:var_brownian_result_to_expand} gives the variance $\V\left[\Delta_{W,n}\right]$ as a function of the time steps at both levels, $\Delta t_\ell$ and $\Delta t_{\ell-1}$, which also appear in the diffusion coefficients $D_{\Delta t_{\ell}}$ and $D_{\Delta t_{\ell-1}}$. Equivalently, using the relation $\Delta t_{\ell-1}=M\Delta t_\ell$, the variance can be written in terms of $\Delta t_\ell$ and the refinement factor $M$.

As the time steps are independent, the total variance after $N$ steps is
\begin{equation}
\label{eq:variance_Brownian_sum}
\V\left[\sum_{n=0}^{N-1}\Delta_{W,n}\right] =  N \V\left[\Delta_W\right],
\end{equation}
where we omit the subscript in the right hand side, as the increments are i.i.d. 
We now compute the Maclaurin series in $\Delta t_\ell$ of \eqref{eq:variance_Brownian_sum} as
\begin{equation}
\label{eq:maclaurin_brownian}
\lim_{\Delta t_\ell \to 0} \V \left[ \sum_{n=0}^{N-1} \Delta_{W,n} \right] = \frac{2t^*\tilde{v}^2}{\epsilon^2} \left(\sqrt{M} - 1 \right)^2 \Delta t_\ell + \mathcal{O}(\Delta t_\ell^2).
\end{equation}
We refer to the supplementary materials for a detailed computation of \eqref{eq:maclaurin_brownian}.
\end{proof}

\begin{remark}
In the supplementary materials we further show that
\begin{equation}
\label{eq:lim_browniain}
\lim_{\epsilon \to 0} \V \left[ \sum_{n=0}^{N-1} \Delta_{W,n} \right] = 0.
\end{equation}
This limit matches the expected convergence of \eqref{eq:GTmod} to \eqref{eq:heat}, as Brownian motion, i.e., the Monte Carlo discretization of \eqref{eq:heat}, is unbiased in the time step.
\end{remark} 

\subsection{Transport increments}
\label{sec:transport_analysis}

We now present and prove similar lemmas concerning the transport increments, i.e., the at least one collision has already taken place in both simulations, so that the initial condition on the velocity can be neglected.

\begin{lemma}
\label{lem:transport_expectation}
Given a simulation interval $\Delta t_{\ell-1}$, an independent fine simulation containing transport increments $\Delta T^{n,m}_{p,\ell}$ with time step sizes $\Delta t_{\ell}$ and a coarse simulation containing transport increments $\Delta T^{n}_{p,\ell-1}$ with time step size $\Delta t_{\ell-1}$, generated by Algorithm~\ref{alg:correlation}. Then, $$\E\left[\Delta_{T,n}\right] = \E\left[\sum_{m=0}^{M-1} \Delta T^{n,m}_{p,\ell} - \Delta T^{n}_{p,\ell-1}\right] = 0.$$
\end{lemma}
\begin{proof}
Based on \eqref{eq:beta_properties} and the fact that the correlation preserves the coarse model statistics (as demonstrated in Section~\ref{sec:correlation_transport}), we can  show that the expected value of both paths is zero, 
\begin{align}
&\E\left[\Delta \U^{n,m}_{p,\ell}\right] = \Delta t_\ell \E\left[V^{n,m}_{p,\Delta t_\ell}\right] = 0, \\
&\E\left[\Delta \U^{n}_{p,\ell-1}\right] = \Delta t_{\ell-1} \E\left[V^{n}_{p,\Delta t_{\ell-1}}\right] = 0.
\end{align}
The proof then follows from
\begin{equation}
\E\left[\Delta_{\U,n}\right] = \E\left[\sum_{m=0}^{M-1} \Delta \U^{n,m}_{p,\ell} - \Delta \U^{n}_{p,\ell-1}\right] = \sum_{m=0}^{M-1}  \E\left[\Delta \U^{n,m}_{p,\ell}\right] - \E\left[\Delta \U^{n}_{p,\ell-1}\right].
\end{equation}
\end{proof}

\begin{lemma}
\label{lem:transport_variance}
Given a fixed ratio $M$ between two time step sizes $\Delta t_{\ell-1} = M \Delta t_\ell$, a fixed simulation horizon $t^*=NM\Delta t_\ell$, an independent fine simulation containing transport increments $\Delta T^{n,m}_{p,\ell}$ with time step sizes $\Delta t_{\ell}$ and a coarse simulation containing transport increments $\Delta T^{n}_{p,\ell-1}$, generated by Algorithm~\ref{alg:correlation} with time step size $\Delta t_{\ell-1}$. Then, the variance $$\V \left[ \sum_{n=0}^{N-1} \Delta_{T,n} \right] = \V\left[\sum_{m=0}^{M-1} \Delta T^{n,m}_{p,\ell} - \Delta T^{n}_{p,\ell-1}\right]$$ converges $\mathcal{O}(\Delta t_\ell)$ to 0, asymptotically for $\Delta t_\ell \to 0$.
\end{lemma}
\begin{proof}
As there is a non-zero probability for both the simulation at level $\ell$ and at level $\ell-1$ that no collision occurs in a given time step $n$, the differences between these two simulations are themselves correlated across time steps. The variance of the difference after $N$ steps is thus given by
\begin{equation}
\label{eq:global_var}
\V \left[ \sum_{n=0}^{N-1} \Delta_{\U,n} \right] = \sum_{n=0}^{N-1} \V\left[\Delta_{\U,n}\right]+ 2\sum_{n=0}^{N-2} \sum_{n^\prime=n+1}^{N-1}\Cov \left( \Delta_{\U,n} , \Delta_{\U,n^\prime} \right), \quad \text{with}
\end{equation}
\begin{equation}
\label{eq:transport_variance}
\begin{split}
\V\left[\Delta_{\U,n}\right]= \sum_{m=0}^{M-1} \V\left[\Delta \U^{n,m}_{p,\ell}\right]+ \V\left[\Delta \U^{n}_{p,\ell-1}\right]- 2\sum_{m=0}^{M-1} \Cov\left(\Delta \U^{n,m}_{p,\ell}, \Delta \U^{n}_{p,\ell-1} \right)& \\
+ 2 \sum_{m=0}^{M-2} \sum_{m^\prime=m+1}^{M-1}\!\! \Cov\left(\Delta \U^{n,m}_{p,\ell}, \Delta \U^{n,m^\prime}_{p,\ell} \right)&.
\end{split}
\end{equation}

Given Lemma~\ref{lem:transport_expectation}, we compute the variances of the individual increments as the expectation of their squared values:
\begin{equation}
\label{eq:transport_variance_fine_coarse}
\V\left[\Delta \U^{n,m}_{p,\ell}\right] = \Delta t_\ell^2 \CV_{\Delta t_\ell}^2 \quad \text{and} \quad \V\left[\Delta \U^{n}_{p,\ell-1}\right] = \Delta t_{\ell-1}^2 \CV_{\Delta t_{\ell-1}}^2.
\end{equation}
The remaining covariance sums are workout out in Appendix~\ref{app:covariances}.

We now plug the computed variances and covariances (\eqref{eq:transport_variance_fine_coarse}, \eqref{eq:cov_sum_transport_coarse_fine}, \eqref{eq:cov_sum_transport_fine} and \eqref{eq:final_covar}) into \eqref{eq:global_var} and compute the first term of the Maclaurin series in $\Delta t_\ell$:
\begin{equation}
\label{eq:maclaurin_transport}
\lim_{\Delta t_\ell \to 0} \V \left[ \sum_{n=0}^{N-1} \Delta_{\U,n} \right] = 2\tilde{v}^2(M-1)\left(e^{-t^*/\epsilon^2} - 1 + \frac{t^*}{\epsilon^2} \right) \Delta t_\ell + \mathcal{O}(\Delta t_\ell^2).
\end{equation}
We refer to the supplementary materials for a detailed computation of \eqref{eq:maclaurin_transport}.
\end{proof}

\begin{remark}
In the supplementary materials we further show that
\begin{equation}
\label{eq:limit_transport}
\lim_{\epsilon \to 0} \V \left[ \sum_{n=0}^{N-1} \Delta_{\U,n} \right] = 0.
\end{equation}
This limit matches the expected convergence of \eqref{eq:GTmod} to \eqref{eq:heat}, as \eqref{eq:heat} no longer contains transport behavior.
\end{remark}  

\subsection{Velocity expectations and variances}
\label{sec:velocity_exp_var}

We now present the same lemmas concerning the expectation and variance of the differences of velocities.

\begin{lemma}
\label{lem:velocity_expectation}
Given a simulation interval $\Delta t_{\ell-1}$, an independent fine simulation containing velocities $V^{n,m}_{p,\ell}$ with time step size $\Delta t_{\ell}$ and a coarse simulation containing velocities $V^{n}_{p,\ell-1}$ with time step size $\Delta t_{\ell-1}$, generated by Algorithm~\ref{alg:correlation}. Then, given that both simulations have experienced at least one collision, $\E\left[V_{p,\Delta t_\ell}^{n,m}- V_{p,\Delta t_{\ell-1}}^{n}\right] = 0$.
\end{lemma}
\begin{proof}
The expected values of the individual velocities of both fine and coarse simulations are zero by \eqref{eq:beta_properties}
\begin{equation}
\E\left[V_{p,\Delta t_\ell}^{n,m}\right] = \E\left[V_{p,\Delta t_{\ell-1}}^{n}\right] = 0.
\end{equation}
This means that the expected value of their difference is also zero.
\end{proof}

\begin{lemma}
\label{lem:velocity_variance}
Given a fixed ratio $M$ between two time step sizes $\Delta t_{\ell-1} = M \Delta t_\ell$, a fixed simulation horizon $t^*=NM\Delta t_\ell$, an independent fine simulation with time step sizes $\Delta t_{\ell}$ and a coarse simulation, generated by Algorithm~\ref{alg:correlation} with time step size $\Delta t_{\ell-1}$. Then, given that both simulations have experienced at least one collision, the variance $\V\left[V_{p,\Delta t_\ell}^{n,m} - V_{p,\Delta t_{\ell-1}}^{n}\right]$ converges $\mathcal{O}(\Delta t_\ell)$ to 0, asymptotically for $\Delta t_\ell \to 0$.
\end{lemma}
\begin{proof}
We compute the variance of this difference as
\begin{align}
\!\!\!\!\V\!\left[V_{p,\Delta t_\ell}^{n,m}\!\!-\! V_{p,\Delta t_{\ell-1}}^{n}\right] \!&= \E\!\left[\left(V_{p,\Delta t_\ell}^{n,m}- V_{p,\Delta t_{\ell-1}}^{n}\right)^2\right]\\
&=\E\!\left[\!\left(V_{p,\Delta t_\ell}^{n,m}\right)^2\!\right] \!\!+\! \E\!\left[\!\left(V_{p,\Delta t_{\ell-1}}^{n}\right)^2\!\right] \!\!-\! 2\E\!\left[\!V_{p,\Delta t_\ell}^{n,m} \! V_{p,\Delta t_{\ell-1}}^{n}\!\right]\label{eq:velocity_var1}\\
\label{eq:velocity_var2}&=\CV_\ell^2 + \CV_{\ell-1}^2 - 2\CV_\ell\CV_{\ell-1}\frac{1}{M} \sum_{m=0}^{M-1}\E\!\left[\VD^{n,m}_{p,\ell}\VD^{n}_{p,\ell-1}\right]\\
\label{eq:velocity_var3}&=\CV_\ell^2 + \CV_{\ell-1}^2 - 2\CV_{\ell-1}^2\\
\label{eq:velocity_var}&=\CV_\ell^2 - \CV_{\ell-1}^2,
\end{align}
where we use \eqref{eq:beta_properties} in the step from \eqref{eq:velocity_var1} to \eqref{eq:velocity_var2} and where the step from \eqref{eq:velocity_var2} to \eqref{eq:velocity_var3} follows the same logic as the calculation of \eqref{eq:cov_sum_transport_coarse_fine} in Appendix~\ref{app:covariances}.

We now compute the Maclaurin series in $\Delta t_\ell$ of \eqref{eq:velocity_var} as
\begin{equation}
\label{eq:maclaurin_velocity}
\lim_{\Delta t_\ell \to 0} \V\left[V_{p,\Delta t_\ell}^{n,m}- V_{p,\Delta t_{\ell-1}}^{n}\right] = \frac{2\tilde{v}^2(M-1)}{\epsilon^4} \Delta t_\ell + \mathcal{O}(\Delta t_\ell^2).
\end{equation}
We refer to the supplementary materials for a detailed computation of \eqref{eq:maclaurin_velocity}.
\end{proof}
\begin{remark}
In the supplementary materials we further show that
\begin{equation}
\label{eq:limit_transport}
\lim_{\epsilon \to 0} \V\left[V_{p,\Delta t_\ell}^{n,m}- V_{p,\Delta t_{\ell-1}}^{n}\right] = 0.
\end{equation}
This limit matches the expected convergence of \eqref{eq:GTmod} to \eqref{eq:heat}, as \eqref{eq:heat} no longer contains transport behavior.
\end{remark}  

\subsection{Proof of convergence}
\label{sec:proof} 

Now that we have established some key properties of the correlation between the coupled trajectories at level $\ell$ and $\ell-1$, we have everything in place to derive bounds on the difference estimators~\eqref{eq:MClestimator}. To this end, we assume that the quantity of interest $F(x,v)$ is Lipschitz continuous in both position and velocity, i.e., there exist constants $K_x$ and $K_v$ so that,
\begin{equation}
\label{eq:lipschitz}
|F(x_1,v_1) - F(x_2,v_2)| \leq K_x |x_1 - x_2| + K_v |v_1 - v_2|
\end{equation}
holds, for all values in the domains $\D_x$ and $\D_v$. We can now prove the convergence of our scheme, i.e., consistency as $\ell\to\infty$. This proof is based on Lemmas~\ref{lem:brownian_expectation}--\ref{lem:velocity_variance} and is structured as follows: First we present three lemmas, which verify Assumptions 2--4 in Theorem \ref{thm:giles}. We then present a convergence theorem for our scheme in Theorem \ref{thm:convergence}.

First, we verify the rate of decreasing bias (Theorem \ref{thm:giles}, Assumption 2).
\begin{lemma}
\label{lem:expected}
Given $F(x,v)$, Lipschitz in position and velocity, and a sequence of approximations $\ell = 0,\dots,\infty$, coupled as described in algorithm \ref{alg:correlation} with time steps $\Delta t_{\ell-1} = M\Delta t_\ell$, then $\exists L_1, 0 < L_1 < \infty : \exists c_1 : \forall \ell \geq L_1 : \left|\E\left[\hat{F}_\ell - F\right]\right| \leq c_1 2^{-\alpha\ell}$, with $\alpha=\log_2(M)$.
\end{lemma}
\begin{proof}
By \eqref{eq:lipschitz} we have $$\left|\E\left[\hat{F}_\ell - F\right]\right| \leq \E\left[\left|\hat{F}_\ell - F\right|\right]\leq K_x\E\left[\left|X^N_{p,\Delta t}-X^*\right|\right] + K_v\E\left[\left|V^N_{p,\Delta t}-V^*\right|\right],$$ with $\left(X^*,V^*\right)$ the point in the velocity-phase space which produces the expected value $\hat{F}$ of $F(x,v)$, which exists by the mean value theorem. Given that both time-splitting and the IMEX-equation are linear approximations in $\Delta t$ in the limit $\Delta t \to 0$, we can observe that both $\E\left[\left|X^N_{p,\Delta t}-X^*\right|\right]$ and $\E\left[\left|V^N_{p,\Delta t}-V^*\right|\right]$ go to zero with the weak order of the explicit first order simulation method, i.e.,
\begin{equation}
\left|\E\left[\hat{F}_\ell - F\right]\right| \leq c^\prime_1 M^{-\ell} = c^\prime_1 2^{-\log_2(M)\ell},
\end{equation}
meaning there exists an upper bound $c_1 2^{-\log_2(M)\ell}$, once $\ell$ is sufficiently large.
\end{proof}

Second, we verify the rate of decreasing variance (Theorem \ref{thm:giles}, Assumption 3).
\begin{lemma}
\label{lem:variance}
Given $F(x,v)$, Lipschitz in position and velocity, and a sequence of approximations $\ell = 0,\dots,\infty$, coupled as described in algorithm \ref{alg:correlation} with time steps $\Delta t_{\ell-1} = M\Delta t_\ell$, then $\exists L_2, 0 < L_2 < \infty : \exists c_2 : \forall \ell \geq L_2 : \V\left[\hat{F}_\ell - F_{\ell-1}\right]\leq c_2 2^{-\beta\ell}$, with $\beta=\log_2(M)$.
\end{lemma}
\begin{proof} By using the Lipschitz property and Lemmas \ref{lem:brownian_expectation}, \ref{lem:transport_expectation} and \ref{lem:velocity_expectation} we compute the following bound on the variance of the difference estimators:
\begin{align}
\V\left[\hat{F}_\ell - F_{\ell-1}\right]&\leq \E\left[\left(\hat{F}_\ell - F_{\ell-1}\right)^2\right]\\
&\leq \E\left[\left(K_x \left| X_{p,\Delta t_\ell}^{n,m}- X_{p,\Delta t_{\ell-1}}^{n}\right| + K_v \left|V_{p,\Delta t_\ell}^{n,m}- V_{p,\Delta t_{\ell-1}}^{n}\right|\right)^2\right]\\
&= K_x^2\E\left[\left(X_{p,\Delta t_\ell}^{n,m}- X_{p,\Delta t_{\ell-1}}^{n}\right)^2\right] + K_v^2\E\left[\left(V_{p,\Delta t_\ell}^{n,m}- V_{p,\Delta t_{\ell-1}}^{n}\right)^2\right] \\
&\quad + 2K_xK_v\E\left[\left|\left(X_{p,\Delta t_\ell}^{n,m}- X_{p,\Delta t_{\ell-1}}^{n}\right)\left(V_{p,\Delta t_\ell}^{n,m}- V_{p,\Delta t_{\ell-1}}^{n}\right)\right|\right]\\
&\leq K_x^2\V\left[X_{p,\Delta t_\ell}^{n,m}- X_{p,\Delta t_{\ell-1}}^{n}\right] + K_v^2\V\left[V_{p,\Delta t_\ell}^{n,m}- V_{p,\Delta t_{\ell-1}}^{n}\right] \\
&\quad + 2K_xK_v\sqrt{\V\left[X_{p,\Delta t_\ell}^{n,m}- X_{p,\Delta t_{\ell-1}}^{n}\right]\V\left[V_{p,\Delta t_\ell}^{n,m}- V_{p,\Delta t_{\ell-1}}^{n}\right]}.
\end{align}
As the linear term is the first nonzero term in \eqref{eq:maclaurin_brownian}, \eqref{eq:maclaurin_transport} and \eqref{eq:maclaurin_velocity} we write 
\begin{equation}
\V\left[\hat{F}_\ell - F\right] \leq c^\prime_2 M^{-\ell} = c^\prime_2 2^{-\log_2(M)\ell},
\end{equation}
meaning there exists an upper bound $c_2 2^{-\log_2(M)\ell}$, once $\ell$ is sufficiently large.
\end{proof}

Third, we verify the rate of increasing cost (Theorem \ref{thm:giles}, Assumption 4).
\begin{lemma}
\label{lem:cost}
For a sequence of difference estimators $\ell = 0,\dots,\infty$, with the time step sizes following $\Delta t_{\ell-1} = M\Delta t_\ell$, correlated as by algorithm \ref{alg:correlation}, the cost per sample $C_\ell$ decreases as $C_\ell \leq c_3 2^{\gamma\ell}$, with $\gamma = \log_2(M)$ and $c_3$ constant, for all $\ell > 0$.
\end{lemma}
\begin{proof}
The number of simulation steps $N_\ell$ needed for a simulation at level $\ell$ is $N_0 M^\ell$, meaning that a difference estimator at level $\ell$ costs $(M+1)M^{\ell-1}$ times that of a single simulation at level 0. This means that
\begin{equation}
C_\ell \leq c_3 M^{\ell} = c_3 2^{\log_2(M)\ell}
\end{equation}
for all $l$.
\end{proof}

Finally, we combine Theorem \ref{thm:giles} and Lemmas \ref{lem:expected}--\ref{lem:cost} to prove the convergence rate of our scheme.
\begin{theorem}
\label{thm:convergence}
Given $F(x,v)$, Lipschitz in position and velocity, and a sequence of approximations $\ell = 0,\dots,\infty$, coupled as described in algorithm \ref{alg:correlation} with time steps $\Delta t_{\ell-1} = M\Delta t_\ell$. The multilevel Monte Carlo method, applied to this sequence of approximations algorithm, converges with a mean square error $\E (\hat{Y} - \E\left[F\right])^2 < \mse^2$ and a computational complexity $C$ bounded by $\E\left[C\right]\leq c_4 \mse^{-2}(\log \mse)^2$, for a given constant $c_4$ and a sufficiently small $\mse > 0$.
\end{theorem}
\begin{proof}
As $E$ constrains the bias which in turn constrains the finest level time step size $\Delta t_L$, a sufficiently small $E$ will enforce that  the number of levels $L\geq\max(L_1,L_2)$. The proof then follows by insertion of Lemmas \ref{lem:expected}, \ref{lem:variance} and \ref{lem:cost} into Theorem \ref{thm:giles}.
\end{proof}

\section{Selecting a level strategy}\label{sec:experiments}

In this section, we combine analytical and numerical results to determine a strategy for selecting levels when using our asymptotic-preserving multilevel scheme. The goal is to minimize the total simulation cost. In Section~\ref{sec:bias_variance} we first study the the bias and variance structure of the multilevel scheme in function of the simulation time step size $\Delta t_\ell$. Based on these insights, we propose two level strategies in Section \ref{sec:level_strategy}, which are then compared in terms of computational cost.

\subsection{Bias and variance structure}
\label{sec:bias_variance}

Theorem~\ref{thm:convergence} proves that the scheme converges asymptotically as $\Delta t_\ell$ decreases, for general values of $\epsilon$, $M$ and $t^*$, but does not give us the full picture for a larger range of values in $\Delta t_\ell$. To get a more complete picture, we consider the bias and variance of the difference estimators for a simple quantity of interest for different values of $\Delta t_\ell$. To this end we fix $t^*=5$ and $M=2$. At level $\ell=0$, we set $\Delta t_0 = 2.5$. At finer levels ($\ell \ge 1$) we set $\Delta t_\ell=\Delta t_{\ell-1}/M=\Delta t_0/M^\ell$. We fix the number of samples per difference estimator at 100~000. For a selection of values of $\epsilon$, we calculate the expected value and variance of the individual samples $\hat{F}_\ell$ and difference estimators $\hat{F}_\ell-\hat{F}_{\ell-1}$ of the squared particle position, as a function of $\Delta t_\ell$, for $1\le\ell$. We choose $\epsilon=10$ (Figure~\ref{fig:exp1eps10}), $\epsilon=1$ (Figure~\ref{fig:exp1eps1}), $\epsilon=0.1$ (Figure~\ref{fig:exp1eps01}) and $\epsilon=0.01$ (Figure~\ref{fig:exp1eps001}). We also plot the analytical bound on the variance given by \eqref{eq:variance_Brownian_sum} and \eqref{eq:global_var}, where we use the respective $K_x$-values 1.5, 5, 8 and 8, based on visually comparing results, and take $K_v = 0$, as the QoI is independent of the velocity.

\setlength\figurewidth{0.5\figurewidth}
\setlength\figureheight{\figurewidth}
\begin{figure} 
\begin{tikzpicture}

\definecolor{color0}{rgb}{0.12156862745098,0.466666666666667,0.705882352941177}
\definecolor{color1}{rgb}{1,0.498039215686275,0.0549019607843137}

\begin{axis}[
title={},
xlabel={Fine time step size},
ylabel={},
xmin=7.62939453125e-05, xmax=2.5,
ymin=10e-06, ymax=1,
xmode=log,
ymode=log,
max space between ticks=20,
xlabel near ticks,
axis x line*=top,
width=\figurewidth,
height=\figureheight,
tick align=outside,
x grid style={white!69.01960784313725!white},
y grid style={white!69.01960784313725!white},
x dir=reverse,
legend entries={{$\hat{F}_\ell$},{$\hat{F}_\ell - \hat{F}_{\ell-1}$}},
legend cell align={left},
legend style={at={(0.03,0.03)}, anchor=south west, draw=white!80.0!black, nodes={scale=0.8, transform shape}}
]
\addlegendimage{color0, mark=square*, mark size=1}
\addlegendimage{color1, mark=*, mark size=1}
\addplot [semithick, color0, mark=square*, mark size=1, mark options={solid}]
table {%
1.25 0.363268097802
0.625 0.305778636046
0.3125 0.276536394545
0.15625 0.260334031778
0.078125 0.252800667519
0.0390625 0.249472673742
0.01953125 0.247726393937
0.009765625 0.246757812006
0.0048828125 0.246494867005
0.00244140625 0.246195784446
0.001220703125 0.246046872963
0.0006103515625 0.245779917797
0.00030517578125 0.245981342386
0.000152587890625 0.246073464533
};
\addplot [semithick, color1, mark=*, mark size=1, mark options={solid}]
table {%
1.25 0.114791944835
0.625 0.0587506129371
0.3125 0.0299599189201
0.15625 0.0146795430835
0.078125 0.00728449994716
0.0390625 0.00368790241765
0.01953125 0.00187453971133
0.009765625 0.000858235227209
0.0048828125 0.000492204481853
0.00244140625 0.000237160276921
0.001220703125 0.000123324533482
0.0006103515625 3.97741282914e-05
0.00030517578125 2.89664580785e-05
0.000152587890625 1.53182609085e-05
};

\logLogSlopeTriangle{0.62}{-0.18}{0.1}{1}{black}; 

\end{axis}

\begin{axis}[
title={},
xlabel={Level},
ylabel={Mean},
xmin=0, xmax=15,
ymin=10e-06, ymax=1,
ymode=log,max space between ticks=20,
xlabel near ticks,
width=\figurewidth,
height=\figureheight,
tick align=outside,
x grid style={white!69.01960784313725!black},
y grid style={white!69.01960784313725!black},
]
\addplot [semithick, color1, mark=*, mark size=1, mark options={solid}]
table {%
1 0.114791944835
2 0.0587506129371
3 0.0299599189201
4 0.0146795430835
5 0.00728449994716
6 0.00368790241765
7 0.00187453971133
8 0.000858235227209
9 0.000492204481853
10 0.000237160276921
11 0.000123324533482
12 3.97741282914e-05
13 2.89664580785e-05
14 1.53182609085e-05
};

\end{axis}

\end{tikzpicture}
\begin{tikzpicture}

\definecolor{color0}{rgb}{0.12156862745098,0.466666666666667,0.705882352941177}
\definecolor{color1}{rgb}{1,0.498039215686275,0.0549019607843137}
\definecolor{color2}{rgb}{0.01,0.8,0.4}

\begin{axis}[
title={},
xlabel=Fine time step size,
ylabel={},
xmin=7.62939453125e-05, xmax=2.5,
ymin=1.5e-06, ymax=0.5,
xmode=log,
ymode=log,max space between ticks=20,
xlabel near ticks,
axis x line*=top,
width=\figurewidth,
height=\figureheight,
tick align=outside,
x grid style={white!69.01960784313725!white},
y grid style={white!69.01960784313725!white},
x dir=reverse,
legend entries={{$\hat{F}_\ell$},{$\hat{F}_\ell - \hat{F}_{\ell-1}$},Analytic},
legend cell align={left},
legend style={at={(0.03,0.03)}, anchor=south west, draw=white!80.0!black, nodes={scale=0.8, transform shape}}
]
\addlegendimage{color0, mark=square*, mark size=1}
\addlegendimage{color1, mark=*, mark size=1}
\addlegendimage{color2, mark=triangle*, mark size=1.5}
\addplot [semithick, color0, mark=square*, mark size=1, mark options={solid}]
table {%
1.25 0.148795639866
0.625 0.0687615177621
0.3125 0.0334260928904
0.15625 0.0165664419948
0.078125 0.00859430831669
0.0390625 0.00469210750172
0.01953125 0.00273843967136
0.009765625 0.00174214666174
0.0048828125 0.00126982397042
0.00244140625 0.00102793741574
0.001220703125 0.000922676778289
0.0006103515625 0.000893435765921
0.00030517578125 0.000823147506966
0.000152587890625 0.000783137241272
};
\addplot [semithick, color1, mark=*, mark size=1, mark options={solid}]
table {%
1.25 0.0473493329276
0.625 0.0173490282878
0.3125 0.00706448960306
0.15625 0.003064757657
0.078125 0.00141458967503
0.0390625 0.000687260321554
0.01953125 0.000335842964853
0.009765625 0.000166647955188
0.0048828125 8.34276476792e-05
0.00244140625 4.09447764522e-05
0.001220703125 2.06137230656e-05
0.0006103515625 1.02678456546e-05
0.00030517578125 5.15841712077e-06
0.000152587890625 2.57959641736e-06
};
\addplot [semithick, color2, mark=triangle*, mark size=1.5, mark options={solid}]
  table[row sep=crcr]{%
1.25	0.0538900385148672\\
0.625	0.0272634058948819\\
0.3125	0.0137119408325307\\
0.15625	0.00687610837975524\\
0.078125	0.00344309838031575\\
0.0390625	0.00172281141258661\\
0.01953125	0.000861721369377649\\
0.009765625	0.000430940514200731\\
0.0048828125	0.000215489726475273\\
0.00244140625	0.000107751137021779\\
0.001220703125	5.38708718395898e-05\\
0.0006103515625	2.73930018094598e-05\\
0.00030517578125	1.30206535663351e-05\\
0.000152587890625	7.89447165447408e-06\\
};

\logLogSlopeTriangle{0.62}{-0.2}{0.1}{1}{black}; 

\end{axis}

\begin{axis}[
title={},
xlabel={Level},
ylabel={Variance},
xmin=0, xmax=15,
ymin=1.5e-06, ymax=0.5,
ymode=log,max space between ticks=20,
xlabel near ticks,
width=\figurewidth,
height=\figureheight,
tick align=outside,
x grid style={white!69.01960784313725!black},
y grid style={white!69.01960784313725!black},
]
\addplot [semithick, color1, mark=*, mark size=1, mark options={solid}]
table {%
1 0.0473493329276
2 0.0173490282878
3 0.00706448960306
4 0.003064757657
5 0.00141458967503
6 0.000687260321554
7 0.000335842964853
8 0.000166647955188
9 8.34276476792e-05
10 4.09447764522e-05
11 2.06137230656e-05
12 1.02678456546e-05
13 5.15841712077e-06
14 2.57959641736e-06
};

\end{axis}

\end{tikzpicture}
\caption{\label{fig:exp1eps10}Mean and variance of the squared particle position for $\epsilon=10$.}
\centering
\end{figure}
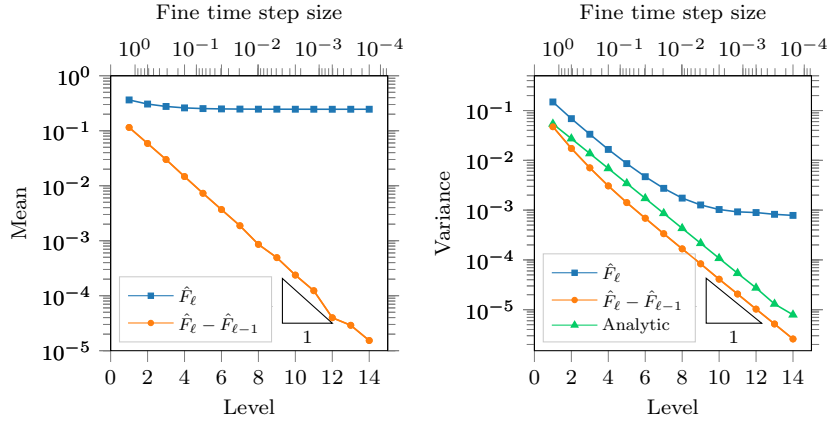
\begin{figure}
\begin{tikzpicture}

\definecolor{color0}{rgb}{0.12156862745098,0.466666666666667,0.705882352941177}
\definecolor{color1}{rgb}{1,0.498039215686275,0.0549019607843137}

\begin{axis}[
title={},
xlabel=Fine time step size,
ylabel={},
xmin=1.9073486328125e-05, xmax=2.5,
ymin=4e-06, ymax=20,
xmode=log,
ymode=log,max space between ticks=20,
max space between ticks=20,
xlabel near ticks,
axis x line*=top,
width=\figurewidth,
height=\figureheight,
tick align=outside,
x grid style={white!69.01960784313725!white},
y grid style={white!69.01960784313725!white},
x dir=reverse,
legend entries={{$\hat{F}_\ell$},{$\hat{F}_\ell - \hat{F}_{\ell-1}$}},
legend cell align={left},
legend style={at={(0.03,0.03)}, anchor=south west, draw=white!80.0!black, nodes={scale=0.8, transform shape}}
]
\addlegendimage{mark=square*, color0, mark size=1}
\addlegendimage{mark=*, color1, mark size=1}
\addplot [semithick, color0, mark=square*, mark size=1, mark options={solid}]
table {%
1.25 7.94001533677
0.625 7.6164871517
0.3125 7.6054460073
0.15625 7.73583949741
0.078125 7.80963650549
0.0390625 7.87716329441
0.01953125 7.94459045039
0.009765625 7.9576559076
0.0048828125 7.98496993476
0.00244140625 8.03304655376
0.001220703125 8.00269556914
0.0006103515625 7.98105317971
0.00030517578125 8.00154231294
0.000152587890625 7.98664204267
7.62939453125e-05 7.9468871424
3.81469726562e-05 8.03647876362
};
\addplot [semithick, color1, mark=*, mark size=1, mark options={solid}]
table {%
1.25 0.575736069864
0.625 0.282648504262
0.3125 0.0269666676531
0.15625 0.115714415131
0.078125 0.131274285832
0.0390625 0.0971562951765
0.01953125 0.0520223085125
0.009765625 0.0247565573833
0.0048828125 0.0107402011465
0.00244140625 0.00435404092613
0.001220703125 0.00418052413651
0.0006103515625 0.000720738464582
0.00030517578125 0.00021838420448
0.000152587890625 0.000497207034614
7.62939453125e-05 0.000245464152265
3.81469726562e-05 7.50383876284e-06
};

\logLogSlopeTriangle{0.62}{-0.25}{0.1}{1}{black}; 

\end{axis}

\begin{axis}[
title={},
xlabel={Level},
ylabel={Mean},
xmin=0, xmax=17,
ymin=4e-06, ymax=20,
ymode=log,max space between ticks=20,
max space between ticks=20,
xlabel near ticks,
width=\figurewidth,
height=\figureheight,
tick align=outside,
x grid style={white!69.01960784313725!black},
y grid style={white!69.01960784313725!black},
]
\addplot [semithick, color1, mark=*, mark size=1, mark options={solid}]
table {%
1 0.575736069864
2 0.282648504262
3 0.0269666676531
4 0.115714415131
5 0.131274285832
6 0.0971562951765
7 0.0520223085125
8 0.0247565573833
9 0.0107402011465
10 0.00435404092613
11 0.00418052413651
12 0.000720738464582
13 0.00021838420448
14 0.000497207034614
15 0.000245464152265
16 7.50383876284e-06
};
\end{axis}

\end{tikzpicture}
\begin{tikzpicture}

\definecolor{color0}{rgb}{0.12156862745098,0.466666666666667,0.705882352941177}
\definecolor{color1}{rgb}{1,0.498039215686275,0.0549019607843137}
\definecolor{color2}{rgb}{0.01,0.8,0.4}

\begin{axis}[
title={},
xlabel=Fine time step size,
ylabel={},
xmin=1.9073486328125e-05, xmax=2.5,
ymin=0.00501889080511248, ymax=252.141284796777,
xmode=log,
ymode=log,
max space between ticks=20,
xlabel near ticks,
axis x line*=top,
width=\figurewidth,
height=\figureheight,
tick align=outside,
x grid style={white!69.01960784313725!white},
y grid style={white!69.01960784313725!white},
x dir=reverse,
legend cell align={left},
x dir=reverse,
legend entries={{$\hat{F}_\ell$},{$\hat{F}_\ell - \hat{F}_{\ell-1}$},Analytic},
legend style={at={(0.03,0.03)}, anchor=south west, draw=white!80.0!black, nodes={scale=0.8, transform shape}}
]
\addlegendimage{mark=square*, color0, mark size=1}
\addlegendimage{mark=*, color1, mark size=1}
\addlegendimage{color2, mark=triangle*, mark size=1.5}
\addplot [semithick, color0, mark=square*, mark size=1, mark options={solid}]
table {%
1.25 118.036111864
0.625 100.489903518
0.3125 86.9503554365
0.15625 77.7852146522
0.078125 72.4342294278
0.0390625 68.9299993428
0.01953125 68.4700458316
0.009765625 67.1105737256
0.0048828125 66.8604040883
0.00244140625 66.9648667806
0.001220703125 66.8639968781
0.0006103515625 66.3854315928
0.00030517578125 66.4769142171
0.000152587890625 66.2717249777
7.62939453125e-05 66.0741211735
3.81469726562e-05 66.636515661
};
\addplot [semithick, color1, mark=*, mark size=1, mark options={solid}]
table {%
1.25 22.2781647852
0.625 22.6700174764
0.3125 18.7830829149
0.15625 12.9740475168
0.078125 8.13383212438
0.0390625 4.65351998551
0.01953125 2.49589667182
0.009765625 1.28159831227
0.0048828125 0.65987461745
0.00244140625 0.33611507775
0.001220703125 0.15027075302
0.0006103515625 0.101897965864
0.00030517578125 0.0372707063373
0.000152587890625 0.0323606238426
7.62939453125e-05 0.00845314835635
3.81469726562e-05 0.0108277783303
};
\addplot [semithick, color2, mark=triangle*, mark size=1.5, mark options={solid}]
  table[row sep=crcr]{%
1.25	24.746596289089\\
0.625	30.0564395581912\\
0.3125	28.8436460831115\\
0.15625	22.1077192629029\\
0.078125	14.2352570286122\\
0.0390625	8.1816671719111\\
0.01953125	4.40237074624432\\
0.009765625	2.28577443489076\\
0.0048828125	1.16494727743271\\
0.00244140625	0.58810783936222\\
0.001220703125	0.295477696361048\\
0.0006103515625	0.148096727113657\\
0.00030517578125	0.0741381188160735\\
0.000152587890625	0.0370916208358175\\
7.62939453125e-05	0.018551396003486\\
3.814697265625e-05	0.00927613982527274\\
};

\logLogSlopeTriangle{0.62}{-0.19}{0.1}{1}{black}; 

\end{axis}

\begin{axis}[
title={},
xlabel={Level},
ylabel={Variance},
xmin=0, xmax=17,
ymin=0.00501889080511248, ymax=252.141284796777,
ymode=log,
max space between ticks=20,
xlabel near ticks,
width=\figurewidth,
height=\figureheight,
tick align=outside,
x grid style={white!69.01960784313725!black},
y grid style={white!69.01960784313725!black},
]
\addplot [semithick, color1, mark=*, mark size=1, mark options={solid}]
table {%
1 22.2781647852
2 22.6700174764
3 18.7830829149
4 12.9740475168
5 8.13383212438
6 4.65351998551
7 2.49589667182
8 1.28159831227
9 0.65987461745
10 0.33611507775
11 0.15027075302
12 0.101897965864
13 0.0372707063373
14 0.0323606238426
15 0.00845314835635
16 0.0108277783303
};

\end{axis}

\end{tikzpicture}
\caption{\label{fig:exp1eps1} Mean and variance of the squared particle position for $\epsilon=1$.}
\centering
\end{figure}
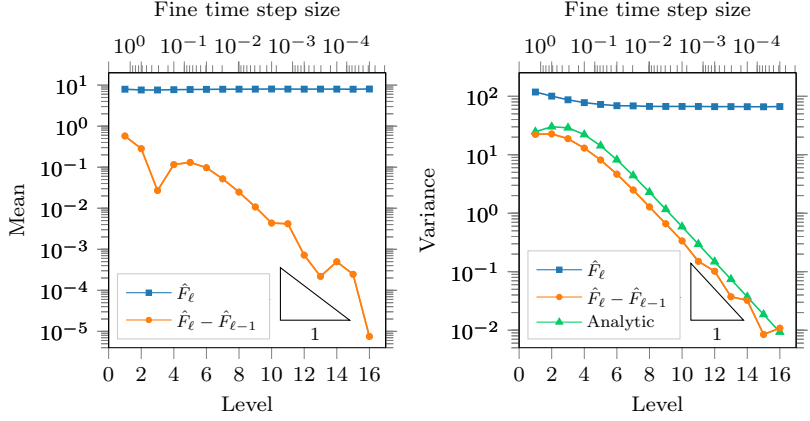
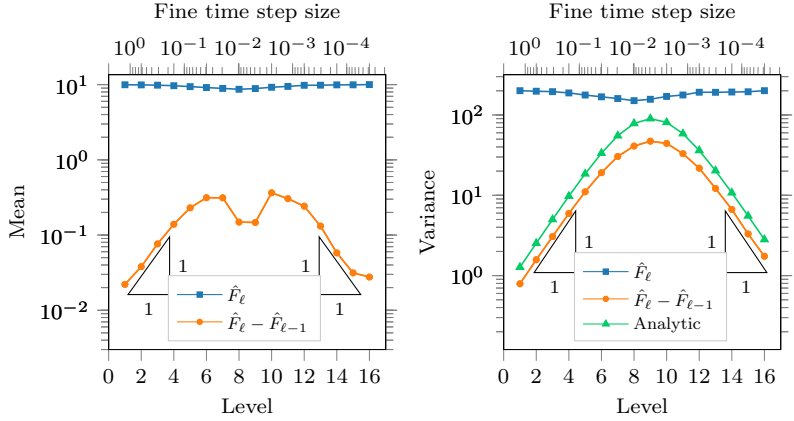
\begin{figure}
\begin{tikzpicture}

\definecolor{color0}{rgb}{0.12156862745098,0.466666666666667,0.705882352941177}
\definecolor{color1}{rgb}{1,0.498039215686275,0.0549019607843137}

\begin{axis}[
title={},
xlabel=Fine time step size,
ylabel={},
xmin=1.9073486328125e-05, xmax=2.5,
ymin=0.003, ymax=13.5970415557483,
xmode=log,
ymode=log,max space between ticks=20,
xlabel near ticks,
axis x line*=top,
width=\figurewidth,
height=\figureheight,
tick align=outside,
x grid style={white!69.01960784313725!white},
y grid style={white!69.01960784313725!white},
x dir=reverse,
legend entries={{$\hat{F}_\ell$},{$\hat{F}_\ell - \hat{F}_{\ell-1}$}},
legend cell align={left},
legend style={at={(0.49,0.03)}, anchor=south, draw=white!80.0!black, nodes={scale=0.8, transform shape}}
]
\addlegendimage{mark=square*, color0, mark size=1}
\addlegendimage{mark=*, color1, mark size=1}
\addplot [semithick, color0, mark=square*, mark size=1, mark options={solid}]
table {%
1.25 9.96840748524
0.625 9.92701209234
0.3125 9.84183035662
0.15625 9.68277993855
0.078125 9.41395321451
0.0390625 9.1542241289
0.01953125 8.91059774908
0.009765625 8.68166940737
0.0048828125 8.87359245506
0.00244140625 9.23406380224
0.001220703125 9.47984909443
0.0006103515625 9.80526703079
0.00030517578125 9.83856655483
0.000152587890625 9.93754488952
7.62939453125e-05 9.94776834595
3.81469726562e-05 10.0143362108
};
\addplot [semithick, color1, mark=*, mark size=1, mark options={solid}]
table {%
1.25 0.0220890494684
0.625 0.0379730280919
0.3125 0.0762245498094
0.15625 0.139281702302
0.078125 0.229981759774
0.0390625 0.313906904703
0.01953125 0.312839506265
0.009765625 0.14871536068
0.0048828125 0.147006851952
0.00244140625 0.364501735716
0.001220703125 0.304687669943
0.0006103515625 0.242194718957
0.00030517578125 0.131718838343
0.000152587890625 0.057894945633
7.62939453125e-05 0.0313862876206
3.81469726562e-05 0.0276631699499
};

\logLogSlopeTriangleIncrease{0.22}{0.15}{0.2}{1}{black}; 
\logLogSlopeTriangle{0.76}{-0.15}{0.2}{1}{black}; 

\end{axis}

\begin{axis}[
title={},
xlabel={Level},
ylabel={Mean},
xmin=0, xmax=17,
ymin=0.003, ymax=13.5970415557483,
ymode=log,max space between ticks=20,
xlabel near ticks,
width=\figurewidth,
height=\figureheight,
tick align=outside,
x grid style={white!69.01960784313725!black},
y grid style={white!69.01960784313725!black},
]
\addplot [semithick, color1, mark=*, mark size=1, mark options={solid}]
table {%
1 0.0220890494684
2 0.0379730280919
3 0.0762245498094
4 0.139281702302
5 0.229981759774
6 0.313906904703
7 0.312839506265
8 0.14871536068
9 0.147006851952
10 0.364501735716
11 0.304687669943
12 0.242194718957
13 0.131718838343
14 0.057894945633
15 0.0313862876206
16 0.0276631699499
};
\end{axis}

\end{tikzpicture}
\begin{tikzpicture}

\definecolor{color0}{rgb}{0.12156862745098,0.466666666666667,0.705882352941177}
\definecolor{color1}{rgb}{1,0.498039215686275,0.0549019607843137}
\definecolor{color2}{rgb}{0.01,0.8,0.4}

\begin{axis}[
title={},
xlabel=Fine time step size,
ylabel={},
xmin=1.9073486328125e-05, xmax=2.5,
ymin=0.12, ymax=317.866980175038,
xmode=log,
ymode=log,
max space between ticks=20,
xlabel near ticks,
axis x line*=top,
width=\figurewidth,
height=\figureheight,
tick align=outside,
x grid style={white!69.01960784313725!white},
y grid style={white!69.01960784313725!white},
x dir=reverse,
legend entries={{$\hat{F}_\ell$},{$\hat{F}_\ell - \hat{F}_{\ell-1}$},Analytic},
legend cell align={left},
legend style={at={(0.53,0.03)}, anchor=south, draw=white!80.0!black, nodes={scale=0.8, transform shape}}
]
\addlegendimage{mark=square*, color0, mark size=1}
\addlegendimage{mark=*, color1, mark size=1}
\addlegendimage{color2, mark=triangle*, mark size=1.5}
\addplot [semithick, color0, mark=square*, mark size=1, mark options={solid}]
table {%
1.25 200.599953961
0.625 197.602810507
0.3125 195.100976482
0.15625 188.13466271
0.078125 177.108466721
0.0390625 168.291019986
0.01953125 160.106371909
0.009765625 150.811056995
0.0048828125 156.82559135
0.00244140625 170.106705416
0.001220703125 177.42238752
0.0006103515625 192.008903765
0.00030517578125 191.599309484
0.000152587890625 193.104990471
7.62939453125e-05 194.436980328
3.81469726562e-05 200.665272706
};
\addplot [semithick, color1, mark=*, mark size=1, mark options={solid}]
table {%
1.25 0.79182562221
0.625 1.57527331776
0.3125 3.05934537554
0.15625 5.93665230953
0.078125 11.0592531206
0.0390625 19.1070602726
0.01953125 30.4314797644
0.009765625 41.055244543
0.0048828125 47.1504126074
0.00244140625 44.3055273615
0.001220703125 32.9694210477
0.0006103515625 21.7525565864
0.00030517578125 12.1576923066
0.000152587890625 6.66458096373
7.62939453125e-05 3.29932453829
3.81469726562e-05 1.74211974356
};
\addplot [semithick, color2, mark=triangle*, mark size=1.5, mark options={solid}]
  table[row sep=crcr]{%
1.25	1.27229574950238\\
0.625	2.52910597890772\\
0.3125	4.99600027305723\\
0.15625	9.74315728959438\\
0.078125	18.5126615897468\\
0.0390625	33.4421115996959\\
0.01953125	55.1650449965872\\
0.009765625	78.391653550056\\
0.0048828125	90.0084234745474\\
0.00244140625	80.9331949899323\\
0.001220703125	58.5047103219861\\
0.0006103515625	36.1364302508411\\
0.00030517578125	20.2509302206521\\
0.000152587890625	10.7449425039298\\
7.62939453125e-05	5.53785228302983\\
3.814697265625e-05	2.81167702735844\\
};

\logLogSlopeTriangleIncrease{0.26}{0.15}{0.28}{1}{black}; 
\logLogSlopeTriangle{0.80}{-0.15}{0.28}{1}{black}; 

\end{axis}

\begin{axis}[
title={},
xlabel={Level},
ylabel={Variance},
xmin=0, xmax=17,
ymin=0.12, ymax=317.866980175038,
ymode=log,max space between ticks=20,
xlabel near ticks,
width=\figurewidth,
height=\figureheight,
tick align=outside,
x grid style={white!69.01960784313725!black},
y grid style={white!69.01960784313725!black},
]
\addplot [semithick, color1, mark=*, mark size=1, mark options={solid}]
table {%
1 0.79182562221
2 1.57527331776
3 3.05934537554
4 5.93665230953
5 11.0592531206
6 19.1070602726
7 30.4314797644
8 41.055244543
9 47.1504126074
10 44.3055273615
11 32.9694210477
12 21.7525565864
13 12.1576923066
14 6.66458096373
15 3.29932453829
16 1.74211974356
};

\end{axis}

\end{tikzpicture}
\caption{\label{fig:exp1eps01} Mean and variance of the squared particle position for $\epsilon=0.1$.}
\centering
\end{figure}
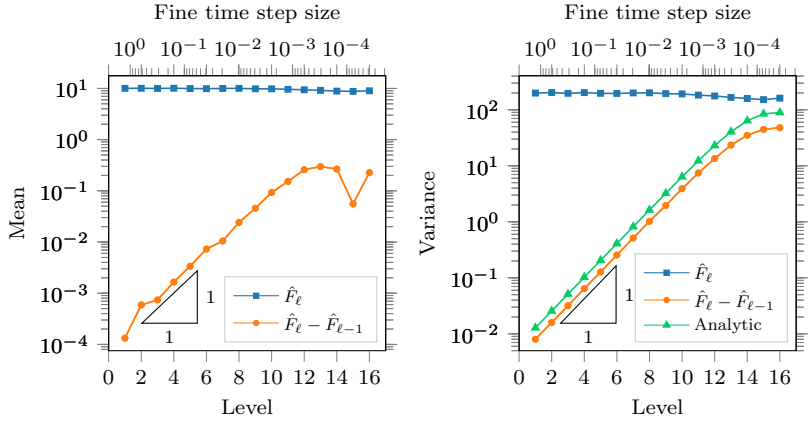
\begin{figure}
\begin{tikzpicture}

\definecolor{color0}{rgb}{0.12156862745098,0.466666666666667,0.705882352941177}
\definecolor{color1}{rgb}{1,0.498039215686275,0.0549019607843137}

\begin{axis}[
title={},
xlabel=Fine time step size,
ylabel={},
xmin=1.9073486328125e-05, xmax=2.5,
ymin=7.54152663311099e-05, ymax=17.6700281591673,
xmode=log,
ymode=log,max space between ticks=20,
xlabel near ticks,
axis x line*=top,
width=\figurewidth,
height=\figureheight,
tick align=outside,
x grid style={white!69.01960784313725!white},
y grid style={white!69.01960784313725!white},
x dir=reverse,
legend entries={{$\hat{F}_\ell$},{$\hat{F}_\ell - \hat{F}_{\ell-1}$}},
legend style={at={(0.97,0.03)}, anchor=south east, draw=white!80.0!black, nodes={scale=0.8, transform shape}},
legend cell align={left}
]
\addlegendimage{mark=square*, color0, mark size=1}
\addlegendimage{mark=*, color1, mark size=1}
\addplot [semithick, color0, mark=square*, mark size=1, mark options={solid}]
table {%
1.25 10.0332914554
0.625 10.060206185
0.3125 9.98993990723
0.15625 10.0729441195
0.078125 9.9422042789
0.0390625 9.89146703197
0.01953125 9.98076221171
0.009765625 10.0333511082
0.0048828125 9.82068682864
0.00244140625 9.82179382673
0.001220703125 9.60603246538
0.0006103515625 9.39013738906
0.00030517578125 9.16276008958
0.000152587890625 8.86152059235
7.62939453125e-05 8.74415627272
3.81469726562e-05 8.99037781221
};
\addplot [semithick, color1, mark=*, mark size=1, mark options={solid}]
table {%
1.25 0.000132293981173
0.625 0.000590123376795
0.3125 0.000739905339119
0.15625 0.00165452038106
0.078125 0.00336643013131
0.0390625 0.00729403394273
0.01953125 0.0104704337608
0.009765625 0.0242198789932
0.0048828125 0.0455253154734
0.00244140625 0.0927604805724
0.001220703125 0.152236971983
0.0006103515625 0.258347036428
0.00030517578125 0.297942469482
0.000152587890625 0.265981327232
7.62939453125e-05 0.0555163237107
3.81469726562e-05 0.226589289455
};

\logLogSlopeTriangleIncrease{0.32}{0.2}{0.1}{1}{black}; 

\end{axis}

\begin{axis}[
title={},
xlabel={Level},
ylabel={Mean},
xmin=0, xmax=17,
ymin=7.54152663311099e-05, ymax=17.6700281591673,
ymode=log,
max space between ticks=20,
xlabel near ticks,
width=\figurewidth,
height=\figureheight,
tick align=outside,
x grid style={white!69.01960784313725!black},
y grid style={white!69.01960784313725!black},
]
\addplot [semithick, color1, mark=*, mark size=1, mark options={solid}]
table {%
1 0.000132293981173
2 0.000590123376795
3 0.000739905339119
4 0.00165452038106
5 0.00336643013131
6 0.00729403394273
7 0.0104704337608
8 0.0242198789932
9 0.0455253154734
10 0.0927604805724
11 0.152236971983
12 0.258347036428
13 0.297942469482
14 0.265981327232
15 0.0555163237107
16 0.226589289455
};

\end{axis}

\end{tikzpicture}
\begin{tikzpicture}

\definecolor{color0}{rgb}{0.12156862745098,0.466666666666667,0.705882352941177}
\definecolor{color1}{rgb}{1,0.498039215686275,0.0549019607843137}
\definecolor{color2}{rgb}{0.01,0.8,0.4}

\begin{axis}[
title={},
xlabel=Fine time step size,
ylabel={},
xmin=1.9073486328125e-05, xmax=2.5,
ymin=0.005, ymax=405.49470514225,
xmode=log,
ymode=log,max space between ticks=20,
xlabel near ticks,
axis x line*=top,
width=\figurewidth,
height=\figureheight,
tick align=outside,
x grid style={white!69.01960784313725!white},
y grid style={white!69.01960784313725!white},
x dir=reverse,
legend style={at={(0.97,0.03)}, anchor=south east, draw=white!80.0!black, nodes={scale=0.8, transform shape}},
legend cell align={left},
legend entries={{$\hat{F}_\ell$},{$\hat{F}_\ell - \hat{F}_{\ell-1}$},Analytic},
]
\addlegendimage{mark=square*, color0, mark size=1}
\addlegendimage{mark=*, color1, mark size=1}
\addlegendimage{color2, mark=triangle*, mark size=1.5}
\addplot [semithick, color0, mark=square*, mark size=1, mark options={solid}]
table {%
1.25 198.862274153
0.625 203.264847785
0.3125 196.820091658
0.15625 202.618168867
0.078125 196.896303242
0.0390625 195.790959567
0.01953125 200.101688317
0.009765625 201.094123104
0.0048828125 194.529087682
0.00244140625 191.928195123
0.001220703125 183.244126641
0.0006103515625 176.509720185
0.00030517578125 166.784253643
0.000152587890625 159.355932779
7.62939453125e-05 151.888375566
3.81469726562e-05 161.727824715
};
\addplot [semithick, color1, mark=*, mark size=1, mark options={solid}]
table {%
1.25 0.00801706694158
0.625 0.0159920064974
0.3125 0.0318247800558
0.15625 0.0642507795671
0.078125 0.126873502589
0.0390625 0.255852942328
0.01953125 0.513408155702
0.009765625 1.02186143721
0.0048828125 1.96156312125
0.00244140625 3.89792605992
0.001220703125 7.45081206185
0.0006103515625 13.4873310251
0.00030517578125 23.4039117139
0.000152587890625 35.1162175386
7.62939453125e-05 44.7572179481
3.81469726562e-05 47.9011176518
};
\addplot [semithick, color2, mark=triangle*, mark size=1.5, mark options={solid}]
  table[row sep=crcr]{%
1.25	0.0127992319748073\\
0.625	0.0255969277958151\\
0.3125	0.0511877103677839\\
0.15625	0.102350834991461\\
0.078125	0.204603288753395\\
0.0390625	0.408812755096231\\
0.01953125	0.816047974690823\\
0.009765625	1.62576990424785\\
0.0048828125	3.22613897400257\\
0.00244140625	6.35032132383497\\
0.001220703125	12.2949683064841\\
0.0006103515625	23.0286639759075\\
0.00030517578125	40.5096923092229\\
0.000152587890625	63.8689668881264\\
7.62939453125e-05	84.7148474070524\\
3.814697265625e-05	89.2373047051834\\
};

\logLogSlopeTriangleIncrease{0.35}{0.2}{0.1}{1}{black}; 

\end{axis}

\begin{axis}[
title={},
xlabel={Level},
ylabel={Variance},
xmin=0, xmax=17,
ymin=0.005, ymax=405.49470514225,
ymode=log,
max space between ticks=20,
xlabel near ticks,
width=\figurewidth,
height=\figureheight,
tick align=outside,
x grid style={white!69.01960784313725!black},
y grid style={white!69.01960784313725!black},
]
\addplot [semithick, color1, mark=*, mark size=1, mark options={solid}]
table {%
1 0.00801706694158
2 0.0159920064974
3 0.0318247800558
4 0.0642507795671
5 0.126873502589
6 0.255852942328
7 0.513408155702
8 1.02186143721
9 1.96156312125
10 3.89792605992
11 7.45081206185
12 13.4873310251
13 23.4039117139
14 35.1162175386
15 44.7572179481
16 47.9011176518
};

\end{axis}

\end{tikzpicture}
\caption{\label{fig:exp1eps001} Mean and variance of the squared particle position for $\epsilon=0.01$.}
\centering
\end{figure}
\setlength\figurewidth{2\figurewidth}

\textbf{The regime $\Delta t \ll \epsilon^2$.} In Figures \ref{fig:exp1eps10} through \ref{fig:exp1eps01}, we see that, as the level $\ell$ increases, the slopes of both the mean and variance curves for the differences approach an asymptotic limit $\mathcal{O}\left( \Delta t \right)$ for $\Delta t \ll \epsilon^2$. This observation matches the weak convergence order of the Euler-Maruyama scheme, used to simulate the model \eqref{eq:transport}--\eqref{eq:collision}, as well as the expected behavior from the time step dependent bias in the asymptotic-preserving model. This confirms the expected behavior from \eqref{eq:maclaurin_brownian} and \eqref{eq:maclaurin_transport} as well as Lemmas~\ref{lem:expected} and \ref{lem:variance}. In this regime, the existing theory for multilevel Monte Carlo methods can be applied, e.g., on the required number of samples per level and conditions for adding levels~\cite{Giles2015}.

\textbf{The regime $\Delta t \gg \epsilon^2$.} For time steps $\Delta t \gg \epsilon^2$, however, we see in Figures~\ref{fig:exp1eps01} and \ref{fig:exp1eps001} that both the mean and the variance curves increase geometrically in terms of increasing level. To explain this perhaps counterintuitive result, we will look at the limit of the modified Goldstein-Taylor model when $\Delta t$ tends to infinity. In this limit, the model \eqref{eq:GTmod} converges to the heat equation:
\begin{equation}
\label{mlmcapeq_GTheat}
\begin{dcases}
\partial_t f_+(x,t) = \partial_{xx} f_+(x,t)\\
\partial_t f_-(x,t) = \partial_{xx} f_-(x,t)
\end{dcases} \Rightarrow \partial_t \rho(x,t) = \partial_{xx} \rho(x,t).
\end{equation}
This means that taking increasingly larger time steps in \eqref{eq:GTmod} is equivalent to taking the limit $\epsilon \to 0$.
 This observation is precisely the asymptotic-preserving property of the particle scheme of Section~\ref{sec:ap_scheme}.

That the scheme approaches two different limiting models in these two limits can be seen most clearly in Figures~\ref{fig:exp1eps10} and~\ref{fig:exp1eps01}. In Figure~\ref{fig:exp1eps01}, the curves for the mean and variance of the differences $\hat{F}_\ell - \hat{F}_{\ell-1}$ (orange lines with circles) decrease for both small and large $\Delta t$, as the model converges to the two limits. In the right hand panel of Figure~\ref{fig:exp1eps10} we see that the variance of the individual simulations at level $\ell$ (blue line with squares) changes drastically as a function of $\Delta t_\ell$ in the region where it is of the same order of magnitude as $\epsilon^2$. This is caused by the approximated models for large and small $\Delta t$ having differences in behavior, which are significant enough to be observed when plotted. The scheme thus converges to different equations for the two limits in $\Delta t$. For small $\Delta t$, there is convergence to \eqref{eq:GT}. For large $\Delta t$, there is convergence to \eqref{eq:heat}. In practice, the size of $\Delta t$ is limited by the simulation time horizon, so it is not possible to get arbitrarily close to \eqref{eq:heat} by increasing the time step size.

\textbf{Connecting the two regimes.} Combining the observations from the two limits ($\epsilon$ tending to zero and $\Delta t_\ell$ tending to zero) in the time step size gives an intuitive interpretation to the multilevel Monte Carlo method in this setting: the method can be interpreted as correcting the result of a pure diffusion simulation by decreasing $\Delta t$ to get a good approximation of the transport-diffusion equation that describes the behavior for a given value of $\epsilon$. The peak of the variance of the differences lies near $\Delta t \approx \epsilon^2$. This makes sense, as this is the region where the model parameters $D_{\Delta t}$ and $\CV_{\Delta t}$ vary the most as a function of $\Delta t$. We also see a dip in the mean of the difference curves in the region of $\Delta t \approx \epsilon^2$. A full analysis of the behavior that occurs in the transition between the asymptotic regimes is left for future work.

\subsection{Performance and level placement strategy}
\label{sec:level_strategy}
In Section \ref{sec:bias_variance}, we experimentally verified the asymptotic convergence rates of the bias and variance in function of increasing level number that were used in the proof of Theorem \ref{thm:convergence}. In doing so, we observed an increasing mean and variance for the difference estimators in the region $\Delta t \gg \epsilon^2$, for $M=2$ and a number of values for $\epsilon$. In this regime, it makes little sense to include a full sequence of levels, as Theorem \ref{thm:giles} only claims a speedup over classical Monte Carlo in the case of decreasing variance. Levels in this region can be interpreted as producing bias estimators that are orders of magnitude smaller than the bias in the model which they are estimating, which means wasted computation.

That a full sequence of levels makes no sense for time step sizes larger than $\epsilon^2$, is therefore intuitively clear. However the question still remains as to what the best approach is to selecting levels. We consider two possible simulation strategies:
\begin{itemize}
\item \textbf{Strategy 1:} Geometric sequence, starting from $\epsilon^2$:
\begin{enumerate}
\item We generate an initial estimate of the quantity of interest at level zero, where we simulate to $t^*$ using $\Delta t_0 = \epsilon^2$.
\item We continue to generate a geometric sequence of levels until an acceptably low bias has been achieved, i.e., $\Delta t_l = \epsilon^2M^{-l}$ for $l>0$.
\end{enumerate}
\item \textbf{Strategy 2:} Additional inclusion of a single coarse level:
\begin{enumerate}
\item We generate an initial estimate of the quantity of interest at level zero, where we simulate to $t^*$ using $\Delta t_0 = t^*$.
\item At level 1 we perform correlated simulations to $t^*$ using $\Delta t_0 = t^*$ and $\Delta t_1 = \epsilon^2$.
\item We continue to generate a geometric sequence of levels until an acceptably low bias has been achieved, i.e., $\Delta t_l = \epsilon^2M^{1-l}$ for $l>1$.
\end{enumerate} 
\end{itemize}
In Appendix~\ref{sec:coarse_theoretical}, we apply multilevel theory to compare both strategies for some very simple quantities of interest. In the considered cases, Strategy 2 proves more efficient. The theoretical approach does not expand to general quantities of interest, however. In the following two subsections (Section~\ref{sec:geometric} for Strategy~1 and Section~\ref{mlmcapsec_coarsexp} for Strategy~2), we compare these strategies for the quantity of interest considered in Section~\ref{sec:bias_variance}. We choose to set $M=2$ and $\epsilon=0.1$, and reduce the time horizon to $t^*=0.5$. This gives us an expensive, but computationally feasible problem. The number of samples per level is derived using the formula~\cite{Giles2015}
\begin{equation}
\left\lceil 2 E^{-2} \sqrt{\frac{V_\ell}{C_\ell}}\left( \sum_{\ell=0}^L \sqrt{V_\ell C_\ell} \right) \right\rceil,
\end{equation}
with $\mse$ the desired bound on the root mean square error, $C_\ell$ the computational cost of the estimator at level $\ell$, and $V_\ell$ the estimated variance of the estimator at level $\ell$, i.e., $V_\ell = \mathbb{V}\left[\hat{F}_\ell - F_{\ell-1} \right]$, where we set $F_{-1} \equiv 0$. An initial estimate for $V_\ell$ is computed using 40, 500 and 1 000 initial samples for respective $\mse$-values 0.1, 0.01 and 0.001. The criteria for adding levels and determining convergence are as described in~\cite{Giles2015}. The cost per sample is determined relative to the cost of a trajectory simulated with $\Delta t = \epsilon^2$.

 The code for performing the numerical experiments can be found at \url{github.com/ELoevbak/APMLMC}, together with the data files containing the simulation results.
\subsubsection{Simulating a geometric sequence}
\label{sec:geometric}
 The results of the simulations with strategy 1 for $\mse$ values 0.1, 0.01 and 0.001 are given in Tables \ref{tab:simulationresults1} through \ref{tab:simulationresults3}. In these tables, we list the time step size $\Delta t_\ell$, number of samples $P_\ell$, variance of the fine simulations $\mathbb{V}\left[ \hat{F}_\ell \right]$, expected value $\mathbb{E}\left[ \hat{F}_\ell - F_{\ell-1} \right]$ and variance $V_\ell$ of the differences of simulations, estimated variance of the estimator $\mathbb{V}\left[\hat{Y}_\ell\right]$, cost per sample $C_\ell$  and level cost $P_\ell C_\ell$. The level $\ell$ estimator variance is estimated as
\begin{equation}
\mathbb{V}\left[ \hat{Y}_\ell \right] = \frac{V_\ell}{P_\ell}.
\end{equation}

{\setlength{\tabcolsep}{2pt}
\def\arraystretch{1.1}
\begin{table}
\caption{Computing the QoI with a geometric level sequence for $E=0.1$.\label{tab:simulationresults1}}
\centering
\makebox[\textwidth][c]{
\begin{tabular}{c | c | c | c | r | c | c | c || c}
Level & $\Delta t_\ell$ & $P_\ell$ & $\mathbb{V}\left[ \hat{F}_\ell \right]$ & \multicolumn{1}{c|}{$\mathbb{E}\left[ \hat{F}_\ell - F_{\ell-1} \right]$} & $V_\ell$ & $\mathbb{V}\left[\hat{Y}_\ell\right]$ & $C_\ell$ & $P_\ell C_\ell$ \\
\hline
0 & $1.00 \times 10^{-2}$ & 1 393 & 1.32 & $8.18 \times 10^{-1}$ & $1.32 \times 10^{0\phantom{-}}$ & $9.45 \times 10^{-4}$ & 1 & 1 393\\
1 & $5.00 \times 10^{-3}$ & 395 & 1.52 & $7.91 \times 10^{-3}$ & $3.58 \times 10^{-1}$ & $9.07 \times 10^{-4}$ & 3 & 1 185\\
2 & $2.50 \times 10^{-3}$ & 296 & 1.59 & $2.18 \times 10^{-2}$  & $4.82 \times 10^{-1}$ & $1.59 \times 10^{-3}$ & 6 & 1 776\\
3 & $1.25 \times 10^{-3}$ & 229 & 2.22 & $-1.48 \times 10^{-2}$  & $3.22 \times 10^{-1}$ & $1.41 \times 10^{-3}$ & 12 & 2 748\\
4 & $6.25 \times 10^{-4}$ & 40  & 1.70 & $1.57 \times 10^{-3}$ & $4.56 \times 10^{-2}$ & $1.14 \times 10^{-3}$ & 24 & 960\\
\hline
$\sum$ & \multicolumn{4}{c}{} & & $6.00 \times 10^{-3}$ & & 8 062
\end{tabular}
}
\end{table}

\begin{table}
\caption{Computing the QoI with a geometric level sequence for $E=0.01$.\label{tab:simulationresults2}}
\centering
\makebox[\textwidth][c]{
\begin{tabular}{c | c | c | c | r | c | c | c || c}
Level & $\Delta t_\ell$ & $P_\ell$ & $\mathbb{V}\left[ \hat{F}_\ell \right]$ & $\mathbb{E}\left[ \hat{F}_\ell - F_{\ell-1} \right]$ & $V_\ell$ & $\mathbb{V}\left[\hat{Y}_\ell\right]$ & $C_\ell$ & $P_\ell C_\ell$ \\
\hline
0 & $1.00 \times 10^{-2}$ & 453 182 & 1.47 & $8.66 \times 10^{-1}$ & $1.47 \times 10^{0\phantom{-}}$ & $3.24 \times 10^{-6}$ & 1   & 453 182\\
1 & $5.00 \times 10^{-3}$ & 142 675 & 1.48 & $1.11 \times 10^{-2}$ & $4.37 \times 10^{-1}$ & $3.06 \times 10^{-6}$ & 3   & 428 025\\
2 & $2.50 \times 10^{-3}$ &  96 682 & 1.55 & $2.82 \times 10^{-2}$ & $4.02 \times 10^{-1}$ & $4.16 \times 10^{-6}$ & 6   & 580 092\\
3 & $1.25 \times 10^{-3}$ & 59 418  & 1.66 & $3.28 \times 10^{-2}$ & $3.01 \times 10^{-1}$ & $5.07 \times 10^{-6}$ & 12  & 713 016\\
4 & $6.25 \times 10^{-4}$ & 33 973  & 1.73 & $2.18 \times 10^{-2}$ & $1.97 \times 10^{-1}$ & $5.80 \times 10^{-6}$ & 24  & 815 352\\
5 & $3.13 \times 10^{-4}$ & 18 484  & 1.82 & $7.10 \times 10^{-3}$  & $1.17 \times 10^{-1}$ & $6.30 \times 10^{-6}$ & 48  & 887 232\\
6 & $1.56 \times 10^{-4}$ & 9 249   & 1.88 & $5.76 \times 10^{-3}$  & $5.91 \times 10^{-2}$ & $6.39 \times 10^{-6}$ & 96  & 887 904\\
7 & $7.81 \times 10^{-5}$ & 4 503   & 1.89 & $7.70 \times 10^{-3}$ & $2.75 \times 10^{-2}$ & $6.11 \times 10^{-6}$ & 192 & 864 576\\
8 & $3.91 \times 10^{-5}$ & 2 523   & 1.69 & $-5.16 \times 10^{-4}$ & $1.13 \times 10^{-2}$ & $4.47 \times 10^{-6}$ & 384 & 968 832\\
9 & $1.95 \times 10^{-5}$ & 757   	& 1.65 & $2.93 \times 10^{-3}$ & $2.62 \times 10^{-3}$ & $3.46 \times 10^{-6}$ & 768 & 581 376\\
10 & $9.75 \times 10^{-6}$ & 500    & 2.04 & $2.08 \times 10^{-3}$ & $4.15 \times 10^{-3}$ & $8.29 \times 10^{-5}$ & 1 536 & 768 000\\
\hline
$\sum$ & \multicolumn{4}{c}{} & & $5.64 \times 10^{-5}$ & & 7 947 587
\end{tabular}
}
\end{table}

\begin{table}
\caption{Computing the QoI with a geometric level sequence for $E=0.001$.\label{tab:simulationresults3}}
\centering
\makebox[\textwidth][c]{
\begin{tabular}{c | c | c | c | r | c | c | c || c}
Level & $\Delta t_\ell$ & $P_\ell$ & $\mathbb{V}\left[ \hat{F}_\ell \right]$ & \multicolumn{1}{c|}{$\mathbb{E}\left[ \hat{F}_\ell - F_{\ell-1} \right]$} & $V_\ell$ & $\mathbb{V}\left[\hat{Y}_\ell\right]$ & $C_\ell$ & $P_\ell C_\ell$ \\
\hline
0 & $1.00 \times 10^{-2}$ & 64 213 534 & 1.47 & $8.65 \times 10^{-1}$ & $1.47 \times 10^{0\phantom{-}}$ & $2.29 \times 10^{-8}$ & 1  & 64 213 534\\
1 & $5.00 \times 10^{-3}$ & 20 183 309 & 1.49 & $1.07 \times 10^{-2}$ & $4.37 \times 10^{-1}$ & $2.16 \times 10^{-8}$ & 3   & 60 549 927\\
2 & $2.50 \times 10^{-3}$ & 13 692 369 & 1.57 & $2.86 \times 10^{-2}$ & $4.02 \times 10^{-1}$ & $2.94 \times 10^{-8}$ & 6   & 82 154 214\\
3 & $1.25 \times 10^{-3}$ & 8 407 373  & 1.65 & $2.89 \times 10^{-2}$ & $3.03 \times 10^{-1}$ & $3.61 \times 10^{-8}$ & 12  & 100 888 476\\
4 & $6.25 \times 10^{-4}$ & 4 771 795  & 1.73 & $2.07 \times 10^{-2}$ & $1.95 \times 10^{-1}$ & $4.09 \times 10^{-8}$ & 24  & 114 523 080\\
5 & $3.13 \times 10^{-4}$ & 2 548 616  & 1.76 & $1.22 \times 10^{-2}$  & $1.12 \times 10^{-1}$ & $4.38 \times 10^{-8}$ & 48  & 122 333 568\\
6 & $1.56 \times 10^{-4}$ & 1 323 896  & 1.78 & $6.82 \times 10^{-3}$  & $6.00 \times 10^{-2}$ & $4.53 \times 10^{-8}$ & 96  & 127 094 016\\
7 & $7.81 \times 10^{-5}$ & 677 724    & 1.79 & $3.57 \times 10^{-3}$ & $3.13 \times 10^{-2}$ & $4.62 \times 10^{-8}$ & 192 & 130 123 008\\
8 & $3.91 \times 10^{-5}$ & 336 519    & 1.82 & $1.94 \times 10^{-3}$ & $1.57 \times 10^{-2}$ & $4.66 \times 10^{-8}$ & 384 & 129 223 296\\
9 & $1.95 \times 10^{-5}$ & 172 183    & 1.81 & $6.79 \times 10^{-4}$ & $8.12 \times 10^{-3}$ & $4.72 \times 10^{-8}$ & 768 & 132 236 544\\
10 & $9.75 \times 10^{-6}$ & 92 747    & 1.80 & $6.00 \times 10^{-5}$ & $3.46 \times 10^{-3}$ & $3.73 \times 10^{-8}$ &1 536& 142 459 392\\
11 & $4.88 \times 10^{-6}$ & 69 615    & 1.79 & $5.27 \times 10^{-4}$ & $2.46 \times 10^{-3}$ & $3.53 \times 10^{-8}$ &3 072 &213 857 280\\
12 & $2.44 \times 10^{-6}$ & 1 000     & 2.01 & $-3.87 \times 10^{-4}$ & $6.12 \times 10^{-4}$&$6.12 \times 10^{-7}$ & 6 144 & 6 144 000\\
\hline
$\sum$ & \multicolumn{4}{c}{} & & $1.06 \times 10^{-6}$ & & 1 425 800 335
\end{tabular}
}
\end{table}
}

We see that the number of samples $P_\ell$ needed to keep $\sum_{\ell=0}^L \mathbb{V}\left[ \hat{Y}_\ell \right] < E^2$ decreases drastically in function of $\ell$. We also see that $\mathbb{E}\left[ F_L - F_{L-1} \right] < E$. The cost per level $P_\ell C_\ell$ is spread quite evenly over the levels, which is to be expected as the geometric factor with which the cost increases with $\ell$ is asymptotically the same as that with which $V_\ell$ decreases. In short, we thus achieve the bias of the finest level, while a large amount of variance reduction is performed in the coarser levels. We can thus conclude that the experimental results match the expected behavior of the multilevel Monte Carlo method.

The total cost of each multilevel simulation, relative to the cost of a single sample at the coarsest level is computed as the sum of the costs at each level, i.e., the sum of the right most column of Tables \ref{tab:simulationresults1} through \ref{tab:simulationresults3}. We can estimate the cost for an equivalent classical Monte Carlo simulation by considering that one needs to perform
\begin{equation}
P_C = \left\lceil \frac{\mathbb{V}\left[ \hat{F}_L \right]}{\sum_{\ell=0}^L\mathbb{V}\left[ \hat{Y}_\ell \right]} \right\rceil
\end{equation}
samples with the fine time step at level $L$, to achieve the same bias and variance as the multilevel estimator. The cost of each sample in the classic Monte Carlo estimator is $\frac{2}{3}C_L$, as we do not need to perform a correlated coarse simulation. Note that, for $E=0.1$, the variance $\mathbb{V}\left[ F_L \right]$ is estimated using very few samples. One should thus be careful about drawing further conclusions from these tables than those made here. 

We now compare the cost of the classical and multilevel Monte Carlo simulations in Table \ref{tab:summary}.
\begin{table}
\caption{Cost comparison between classical and multilevel Monte Carlo\label{tab:summary}}
\centering
\begin{tabular}{c | c | c | c}
RMSE & Classical cost & Multilevel cost & Multilevel speedup\\
\hline
0.1 & 4 544 & 8 062 & 0.56\\
0.01 & 37 011 456 & 7 947 587 & 4.66\\
0.001 & 7 722 983 424 & 1 425 800 335 & 5.42
\end{tabular}
\end{table}
As can be concluded from the table, the multilevel Monte Carlo scheme gives a significant computational advantage when we want to compute low bias results in the setting of the modified Goldstein-Taylor model. This speedup increases as the requested accuracy of the simulation is increased and is expected to asymptotically scale with $\left(\mse\log^2\mse\right)^{-1}$ as the requested root mean square error is further decreased.

\subsubsection{Simulating with a very coarse level}
\label{mlmcapsec_coarsexp}
We now add an extra coarse level with $\Delta t_0=0.5$, and repeat the experiment as before. The results are shown in Tables~\ref{tab:simulationresultscoarse1} through~\ref{tab:simulationresultscoarse3}. In these tables, we see that very little work is done on the coarsest level in comparison with the levels in the geometric sequence. The extra level thus does not have a significant cost, in comparison with the rest of the simulation. We observe that the expected behavior of the multilevel Monte Carlo method, as discussed in the previous section, is also present when including the coarser level. 

We present a cost comparison with and without the coarse level in Table~\ref{tab:summary2}. We see that including a very coarse level consistently gives a speedup, matching expectations from multilevel theory. One observation is that the speedup becomes less significant as the requested root mean square error $\mse$ decreases. This makes sense as the higher the requested accuracy, the more levels are needed, and the smaller the influence of the coarse level strategy. Another thing to note is that, although $V_1$ is smaller than $V_0$, it is still relatively large, and much larger than $V_\ell$ in the following fine levels. We believe that it may be possible to further reduce the variance of level 1 in this strategy by using a different correlation strategy, in which we take into account that the coefficients in \eqref{eq:GTmod} vary strongly if $\Delta t$ takes values with different orders of magnitude.

{\setlength{\tabcolsep}{2pt}
\def\arraystretch{1.05}
\begin{table}
\caption{Computing the QoI with an extra coarse level for $E=0.1$.\label{tab:simulationresultscoarse1}}
\centering 
\makebox[\textwidth][c]{
\begin{tabular}{c | c | c | c | r | c | c | c || c}
Level & $\Delta t_\ell$ & $P_\ell$ & $\mathbb{V}\left[ \hat{F}_\ell \right]$ & \multicolumn{1}{c|}{$\mathbb{E}\left[ \hat{F}_\ell - F_{\ell-1} \right]$} & $V_\ell$ & $\mathbb{V}\left[\hat{Y}_\ell\right]$ & $C_\ell$ & $P_\ell C_\ell$ \\
\hline
0 & $5.00 \times 10^{-1}$ & 6 476 & 2.17 & $1.01 \times 10^{0\phantom{-}}$ & $2.17 \times 10^{0\phantom{-}}$ & $3.36 \times 10^{-4}$ & 0.02   & 130\\
1 & $1.00 \times 10^{-2}$ & 733 & 1.41 & $-6.56 \times 10^{-2}$ & $1.41 \times 10^{0\phantom{-}}$ & $1.92 \times 10^{-3}$ & 1.02   & 748\\
2 & $5.00 \times 10^{-3}$ & 232   & 1.65 & $2.05 \times 10^{-2}$ & $3.91 \times 10^{-1}$ & $1.68 \times 10^{-3}$ & 3   & 696\\
3 & $2.50 \times 10^{-3}$ & 69   & 1.09 & $1.64 \times 10^{-2}$ & $1.05 \times 10^{-1}$ & $1.52 \times 10^{-3}$ & 6   & 414\\
4 & $1.25 \times 10^{-3}$ & 40    & 0.76 & $-1.42 \times 10^{-2}$ & $3.06 \times 10^{-1}$ & $7.65 \times 10^{-3}$ & 12  & 480\\
\hline
$\sum$ & \multicolumn{4}{c}{} & & $1.31 \times 10^{-2}$ & & 2 467
\end{tabular}
}
\end{table}

\begin{table}
\caption{Computing the QoI with an extra coarse level for $E=0.01$.\label{tab:simulationresultscoarse2}}
\centering 
\makebox[\textwidth][c]{
\begin{tabular}{c | c | c | c | r | c | c | c || c}
Level & $\Delta t_\ell$ & $P_\ell$ & $\mathbb{V}\left[ \hat{F}_\ell \right]$ & \multicolumn{1}{c|}{$\mathbb{E}\left[ \hat{F}_\ell - F_{\ell-1} \right]$} & $V_\ell$ & $\mathbb{V}\left[\hat{Y}_\ell\right]$ & $C_\ell$ & $P_\ell C_\ell$ \\
\hline
0 & $5.00 \times 10^{-1}$ & 3 771 030 & 1.97 & $9.91 \times 10^{-1}$ & $1.97 \times 10^{0\phantom{-}}$ & $5.22 \times 10^{-7}$ & 0.02   & 75 421\\
1 & $1.00 \times 10^{-2}$ & 448 812   & 1.48 & $-1.22 \times 10^{-1}$ & $1.42 \times 10^{0\phantom{-}}$ & $3.16 \times 10^{-6}$ & 1.02   & 457 788\\
2 & $5.00 \times 10^{-3}$ & 145 586   & 1.49 & $1.03 \times 10^{-2}$ & $4.40 \times 10^{-1}$ & $3.02 \times 10^{-6}$ & 3   & 436 758\\
3 & $2.50 \times 10^{-3}$ & 97 850    & 1.55 & $2.93 \times 10^{-2}$ & $3.99 \times 10^{-1}$ & $4.07 \times 10^{-6}$ & 6   & 587 100\\
4 & $1.25 \times 10^{-3}$ & 59 970    & 1.69 & $2.86 \times 10^{-2}$ & $2.98 \times 10^{-1}$ & $4.98 \times 10^{-6}$ & 12  & 719 640\\
5 & $6.25 \times 10^{-4}$ & 33 540    & 1.74 & $2.13 \times 10^{-2}$ & $1.88 \times 10^{-1}$ & $5.62 \times 10^{-6}$ & 24  & 804 960\\
6 & $3.13 \times 10^{-4}$ & 18 508    & 1.69 & $1.09 \times 10^{-2}$ & $1.13 \times 10^{-1}$ & $6.08 \times 10^{-6}$ & 48  & 888 384\\
7 & $1.56 \times 10^{-4}$ & 9 399     & 1.74 & $8.59 \times 10^{-3}$ & $5.64 \times 10^{-2}$ & $6.01 \times 10^{-6}$ & 96  & 902 304\\
8 & $7.81 \times 10^{-5}$ & 4 674     & 1.95 & $6.90 \times 10^{-3}$ & $3.45 \times 10^{-2}$ & $7.38 \times 10^{-6}$ & 192 & 897 408\\
9 & $3.91 \times 10^{-5}$ & 3 000     & 1.70 & $5.10 \times 10^{-3}$ & $2.17 \times 10^{-2}$ & $5.69 \times 10^{-6}$ & 384 & 1 462 272\\
10& $1.96 \times 10^{-5}$ &   500     & 1.48 &$-4.09 \times 10^{-3}$ & $4.87 \times 10^{-3}$ & $9.73 \times 10^{-6}$ & 768 & 384 000\\
\hline
$\sum$ & \multicolumn{4}{c}{} & & $5.63 \times 10^{-5}$ & & 7 616 035
\end{tabular}
}
\end{table}

\begin{table}
\caption{Computing the QoI with an extra coarse level for $E=0.001$.\label{tab:simulationresultscoarse3}}
\centering 
\makebox[\textwidth][c]{
\begin{tabular}{c | c | c | c | r | c | c | c || c}
Level & $\Delta t_\ell$ & $P_\ell$ & $\mathbb{V}\left[ \hat{F}_\ell \right]$ & \multicolumn{1}{c|}{$\mathbb{E}\left[ \hat{F}_\ell - F_{\ell-1} \right]$} & $V_\ell$ & $\mathbb{V}\left[\hat{Y}_\ell\right]$ & $C_\ell$ & $P_\ell C_\ell$ \\
\hline
0 & $5.00 \times 10^{-1}$ & 503 703 652 & 1.96 & $9.90 \times 10^{-1}$ & $1.96 \times 10^{0\phantom{-}}$ & $3.89 \times 10^{-9}$ & 0.02   & 10 074 073\\
1 & $1.00 \times 10^{-2}$ & 59 957 303  & 1.47 & $-1.25 \times 10^{-1}$& $1.42 \times 10^{0\phantom{-}}$ & $2.36 \times 10^{-8}$ & 1.02   & 61 156 449\\
2 & $5.00 \times 10^{-3}$ & 19 401 720  & 1.49 & $1.06 \times 10^{-2}$ & $4.36 \times 10^{-1}$ & $2.25 \times 10^{-8}$ & 3   & 58 205 160\\
3 & $2.50 \times 10^{-3}$ & 13 171 457  & 1.57 & $2.88 \times 10^{-2}$ & $4.02 \times 10^{-1}$ & $3.05 \times 10^{-8}$ & 6   & 79 028 742\\
4 & $1.25 \times 10^{-3}$ & 8 077 832   & 1.66 & $2.90 \times 10^{-2}$ & $3.02 \times 10^{-1}$ & $3.74 \times 10^{-8}$ & 12  & 96 933 984\\
5 & $6.25 \times 10^{-4}$ & 4 581 254   & 1.73 & $2.08 \times 10^{-2}$ & $1.95 \times 10^{-1}$ & $4.25 \times 10^{-8}$ & 24  & 109 950 096\\
6 & $3.13 \times 10^{-4}$ & 2 461 186   & 1.77 & $1.22 \times 10^{-2}$ & $1.12 \times 10^{-1}$ & $4.56 \times 10^{-8}$ & 48  & 118 136 928\\
7 & $1.56 \times 10^{-4}$ & 1 268 014   & 1.79 & $6.82 \times 10^{-3}$ & $5.97 \times 10^{-2}$ & $4.71 \times 10^{-8}$ & 96  & 121 729 344\\
8 & $7.81 \times 10^{-5}$ & 648 311     & 1.79 & $3.52 \times 10^{-3}$ & $3.11 \times 10^{-2}$ & $4.79 \times 10^{-8}$ & 192 & 124 475 712\\
9 & $3.91 \times 10^{-5}$ & 331 642     & 1.81 & $1.79 \times 10^{-3}$ & $1.65 \times 10^{-2}$ & $4.97 \times 10^{-8}$  & 384 & 127 350 528\\
10 &$1.95 \times 10^{-5}$ & 157 940    & 1.81  & $1.07 \times 10^{-3}$ & $7.36 \times 10^{-3}$ & $4.66 \times 10^{-8}$ & 768 & 121 297 920\\
11 &$9.75 \times 10^{-6}$ & 94 235     & 1.79 & $2.35 \times 10^{-4}$  & $3.67 \times 10^{-3}$ & $3.89 \times 10^{-8}$ &1 536 & 144 744 960\\
12 &$4.88 \times 10^{-6}$ & 47 286     & 1.80 & $1.65 \times 10^{-4}$  & $2.63 \times 10^{-3}$ & $5.57 \times 10^{-8}$ &3 072 & 145 262 592\\
13 &$2.44 \times 10^{-6}$ & 1 000      & 1.87 & $-2.69 \times 10^{-4}$ & $4.09 \times 10^{-4}$&$4.09 \times 10^{-7}$ & 6 144 & 6 144 000\\
\hline
$\sum$ & \multicolumn{4}{c}{} & & $9.01 \times 10^{-7}$ & & 1 324 490 488
\end{tabular}
}
\end{table}
}

\begin{table}
\caption{Cost comparison with and without an extra coarse level\label{tab:summary2}}
\centering
\begin{tabular}{c | c | c | c}
RMSE & Without coarse level & With coarse level & Coarse level speedup\\
\hline
0.1 & 8 062 & 2 467 & 3.27\\
0.01 & 7 947 587 & 7 616 035 & 1.04\\
0.001 & 1 425 800 335 & 1 324 490 488 & 1.08
\end{tabular}
\end{table}

\section{Conclusion}
\label{sec:conclusion}

We presented a multilevel Monte Carlo scheme for simulating a generic class of kinetic equations using asymptotic-preserving particle schemes. Although the scheme was derived in one dimension, it can be generalized to higher dimensional simulations with little extra effort. After presenting the scheme, we analyzed its convergence behavior for general functions of the particle position and velocity and provided some insights into level selection strategies.

Some analytical properties of the variance of the difference in position and velocity of two particle simulations were derived. Using these properties, we studied the behavior of the multilevel scheme as a function of the simulation time step $\Delta t_l$ for general Lipschitz quantities of interest. We proved that the multilevel Monte Carlo scheme converges with a computational cost that is asymptotically bounded in the root mean square error bound $\mse$ by $\mathcal{O}\left(\mse^{-2}\log^2(\mse)\right)$. The scheme's speedup over single level Monte Carlo was confirmed through numerical experiments.

We also compared two approaches to select the levels in the multilevel scheme. After both theoretical analysis and numerical experiments, we concluded that the best option given the proposed correlation strategy is to start with a coarse level with $\Delta t_0=t^*$, followed by a geometric sequence of levels, decreasing from $\Delta t_1=\epsilon^2$.

This work is a first step in combining the multilevel Monte Carlo method to asymptotic-preserving particle schemes for kinetic equations. In future work, we intend to look at alternative ways to increase the correlation of level 1 in the simulation strategy with a very coarse level, as this is of key importance to reduce the high computational cost of simulations for small values of $\epsilon$. We also intend to expand our simulation code to cope with more complex models, including higher dimensional cases, absorption and position dependent model parameters, making the scheme directly applicable in relevant applications, such as fusion reactor design. An expansion to multi-index Monte Carlo, where the value of both $\Delta t$ and $\epsilon$ are simultaneously varied, can also be considered.

\section*{Acknowledgments}
We thank the anonymous reviewer for their substantial suggestions which have significantly improved this article. Emil L{\o}vbak is funded by the Research Foundation - Flanders (FWO) under SB-Fellowship number 1SB1919N. The computational resources and services used in this work were provided by the VSC (Flemish Supercomputer Center), funded by the Research Foundation - Flanders (FWO) and the Flemish Government -- department EWI.

\section*{Conflict of interest}

The authors declare that they have no conflict of interest.

\appendix

\section{Computing the covariance sums needed to prove Lemma~\ref{lem:transport_variance}}\label{app:covariances}

\subsection{$\displaystyle\sum_{m=0}^{M-1} \Cov\left(\Delta \U^{n,m}_{p,\ell}, \Delta \U^{n}_{p,\ell-1} \right)$}
\label{sec:cov_transp_coarse_fine}

The covariance between the coarse increment $n$ and a fine sub-increment $(n,m)$, given Lemma~\ref{lem:transport_expectation}, can be written out as
\begin{equation}
\label{eq:cov_transport_coarse_fine}
\Cov\left(\Delta \U^{n,m}_{p,\ell}, \Delta \U^{n}_{p,\ell-1} \right) = \frac{\Delta t_{\ell-1}^2 \CV_{\Delta t_\ell}  \CV_{\Delta t_{\ell-1}}}{M} \E\left[\VD^{n,m}_{p,\ell}\VD^{n}_{p,\ell-1}\right].
\end{equation}

To calculate the expected value in~\eqref{eq:cov_transport_coarse_fine}, we need to consider the probabilities of coupled collisions taking place in the correlated simulations. If a collision takes place in a given set of $M$ fine increments, the probability that the correlated coarse simulation time step will also simulate a collision is given by
\begin{align}
P&\left(\PT^n_{p,\ell-1} \geq \pncr \middle| \PT^{n,\text{max}}_{p,\ell} \geq \pncf \right) \\
&\qquad= P\left( {\left(\PT^{n,\text{max}}_{p,\ell}\right)}^M \geq \pncr \middle| {\left(\PT^{n,\text{max}}_{p,\ell}\right)}^M \geq \pncf^M \right) \\
&\qquad= \frac{1-\pncr}{1-\pncf^M}.\label{eq:correlation_probability_GT}
\end{align}
From Lemma~\ref{lem:collision_likelihood}, we know that it is not possible for a collision to take place in the coarse simulation, without a collision taking place in the fine simulation. This leaves three possibilities, when considering collision behavior in coarse time step $n-1$:
\begin{itemize}
\item \textbf{Both at level $\ell-1$ and at level $\ell$, no collision occurred in time step $n-1$.} In this case, time step $n-1$ will not affect the correlation of the velocities between the simulations. If the velocities were correlated at the beginning of time step $n-1$, they will still be so at the end of the time step, and vice versa. In this case, we thus need to look at step $n-2$, and so on, until we reach a past time step that satisfies one of the following two cases.
\item \textbf{A collision occurred at level $\ell$ in time step $n-1$, but not at level $\ell-1$.} In this case, a new $\VD^{n-1,m,*}_{p,\ell}$ was drawn from $\mathcal{M}(v)$, for some $m$, independently of the value of $\VD^{n-1,*}_{p,\ell-1}$. Because all sampled velocities are independent, there is no correlation between $\VD^{n,m-1}_{p,\ell}$ and $\VD^{n-1}_{p,\ell-1}$, making the expected value of their product zero by \eqref{eq:beta_properties}.
\item \textbf{A collision occurred both at level $\ell$ and at level $\ell-1$ in time step $n-1$.} We know that $\VD^{n,0}_{p,\ell} = \VD^{n-1,M,*}_{p,\ell} = \VD^{n}_{p,\ell-1} = \VD^{n-1,*}_{p,\ell-1}$ by \eqref{eq:betacorr}. For every other fine simulation sub-step $m>1$ in time step $n$, there is a probability $1-\pncf^{m}$ that a collision has taken place in at least one of the steps $(n,0)$ through $(n,m)$. These collisions will not affect the coarse simulation until time step $n+1$, where $\VD^{n,*}_{p,\ell}$ is used to generate a new velocity. So we have that $\VD^{n,m}_{p,\ell} = \VD^{n}_{p,\ell-1}$ with a probability $\pncf^{m}$, otherwise the two velocities are uncorrelated.
\end{itemize}
From the above list of possibilities, we conclude that the random variables $\VD^{n,m}_{p,\ell}$ and $\VD^{n}_{p,\ell-1}$ are equal if both of the following are true:
\begin{enumerate}
\item The last simulation step at level $\ell$ that underwent a collision is a sub-step of a simulation step $n^\prime$ at level $\ell-1$ which also underwent a collision. 
\item The coarse simulation step $n^\prime$ is not the current coarse step $n$. 
\end{enumerate}
Otherwise, $\VD^{n,m}_{p,\ell}$ and $\VD^{n}_{p,\ell-1}$ are uncorrelated.

Making use of \eqref{eq:beta_properties} and \eqref{eq:correlation_probability_GT}, we have that
\begin{equation}
\label{eq:corr_transport_coarse_fine}
\E\left[\VD^{n,m}_{p,\ell}\VD^{n}_{p,\ell-1}\right] = \frac{1-\pncr}{1-\pncf^M} \pncf^{m}.
\end{equation}
Plugging \eqref{eq:corr_transport_coarse_fine} into \eqref{eq:cov_transport_coarse_fine} gives
\begin{align}
\sum_{m=0}^{M-1} \Cov\left(\Delta \U^{n,m}_{p,\ell}, \Delta \U^{n}_{p,\ell-1} \right) &= \sum_{m=0}^{M-1} \frac{\Delta t_{\ell-1}^2 \CV_{\Delta t_\ell}  \CV_{\Delta t_{\ell-1}}}{M} \frac{1-\pncr}{1-\pncf^M} \pncf^{m} \\
&= \frac{\Delta t_{\ell-1}^2 \CV_{\Delta t_\ell}  \CV_{\Delta t_{\ell-1}}}{M} \frac{\pcr}{\pcf} \\
&=\Delta t_{\ell-1}^2\CV_{\Delta t_{\ell-1}}^2 = \V\left[\Delta \U^{n}_{p,\ell-1}\right],\label{eq:cov_sum_transport_coarse_fine}
\end{align}
where the final equality is a fortuitous coincidence that will shorten some expressions in Section~\ref{sec:analysis}.

\subsection{$\displaystyle\sum_{m=0}^{M-2} \sum_{m^\prime = m+1}^{M-1}\Cov\left(\Delta \U^{n,m}_{p,\ell}, \Delta \U^{n,m^\prime}_{p,\ell} \right)$}

As in Section~\ref{sec:cov_transp_coarse_fine}, we start by writing out the covariance between subsequent fine sub-increments $(n,m)$ and $(n,m^\prime)$ as
\begin{equation}
\label{eq:cov_coll_fine_start}
\Cov\left(\Delta \U^{n,m}_{p,\ell}, \Delta \U^{n,m^\prime}_{p,\ell} \right) = \Delta t_\ell^2 \CV_{\Delta t_\ell}^2 \E \left[ \VD^{n,m}_{p,\Delta t_\ell} \VD^{n,m^\prime}_{p,\Delta t_\ell} \right].
\end{equation}

Calculating the sum of the covariance between subsequent fine increments is straightforward. To simplify notation, we introduce $\Delta m$ as shorthand for $m^\prime - m$. As the collision probability is constant across time steps, the covariance is given by
\begin{equation}
\label{eq:cov_coll_fine_exp} \sum_{m=0}^{M-2} \sum_{m^\prime = m+1}^{M-1} \Cov\left(\Delta \U^{n,m}_{p,\ell}, \Delta \U^{n,m^\prime}_{p,\ell} \right) = \Delta t_\ell^2 \CV_{\Delta t_\ell}^2 \sum_{m=0}^{M-2} \sum_{m^\prime = m+1}^{M-1} \E\left[\VD^{n,m}_{p,\ell}\VD^{n,m^\prime}_{p,\ell}\right].
\end{equation}
Making use of \eqref{eq:beta_properties}, we get that the r.h.s. of \eqref{eq:cov_coll_fine_exp} is equal to
\begin{equation}
\Delta t_\ell^2 \CV_{\Delta t_\ell}^2 \sum_{m=0}^{M-2} \sum_{m^\prime = m+1}^{M-1} \pncf^{m^\prime - m} = \Delta t_\ell^2 \CV_{\Delta t_\ell}^2 \sum_{\dm=1}^{M-1} (M-\dm) \pncf^{\dm}\label{eq:cov_coll_fine_p}
\end{equation}
By making use of the identity
\begin{equation}
\label{eq:identity}
\sum_{m=1}^{M-1}(M-m)a^m = \frac{a(a^{M}+M(1-a)-1)}{(1-a)^2}
\end{equation}
and noting that $\frac{\pncf}{\pcf}=\frac{\epsilon^2}{\Delta t_\ell}$, we work out \eqref{eq:cov_coll_fine_p} as
\begin{align}
\label{eq:cov_coll_fine_eps} \Delta t_\ell^2 \CV_{\Delta t_\ell}^2 \sum_{\dm=1}^{M-1} (M-\dm) \pncf^{\dm} &= \epsilon^2 \Delta t_\ell \CV_{\Delta t_\ell}^2 \frac{\pncf^{M}+M\pcf-1}{\pcf} \\
&\label{eq:cov_sum_transport_fine}= \epsilon^2 \CV_{\Delta t_\ell}^2 \left( M\Delta t_\ell - (\epsilon^2+\Delta t_\ell)\left(1 - \pncf^M \right) \right).
\end{align}

\subsection{$\displaystyle\sum_{n=0}^{N-2} \sum_{n^\prime=n+1}^{N-1}\Cov \left( \Delta_{\U,n} , \Delta_{\U,n^\prime} \right)$}
Making use of Lemma~\ref{lem:transport_expectation}, we write out the covariance of differences of transport increments for a given $n$ and $n^\prime$ as
\begin{align}
\Cov\left(\Delta_{\U,n},\Delta_{\U,n^\prime}\right) &= \E \left[ \left(\sum_{m=0}^{M-1} \Delta \U^{n,m}_{p,\ell} - \Delta \U^n_{p,\ell-1} \right) \left(\sum_{m^\prime=0}^{M-1} \Delta \U^{n^\prime,m^\prime}_{p,\ell} - \Delta \U^{n^\prime}_{p,\ell-1} \right)  \right]\\
\begin{split}
&= \underbrace{\E \left[ \sum_{m=0}^{M-1} \sum_{m^\prime=0}^{M-1} \Delta \U^{n,m}_{p,\ell} \Delta \U^{n^\prime,m^\prime}_{p,\ell} \right]}_{(I)} + \underbrace{\E \left[ \Delta \U^n_{p,\ell-1} \Delta \U^{n^\prime}_{p,\ell-1} \right]\vphantom{\E \left[ \sum_{m=0}^{M-1} \sum_{m^\prime=0}^{M-1} \Delta \U^{n,m}_{p,\ell} \Delta \U^{n^\prime,m^\prime}_{p,\ell} \right]}}_\text{(II)}\\	
&\qquad\quad - \underbrace{\E \left[ \sum_{m=0}^{M-1} \Delta \U^{n,m}_{p,\ell} \Delta \U^{n^\prime}_{p,\ell-1} \right]}_\text{(III)} - \underbrace{\E \left[ \sum_{m=0}^{M-1} \Delta \U^n_{p,\ell-1} \Delta \U^{m^\prime,n^\prime}_{p,\ell} \right]}_\text{(IV)}.
\label{eq:global_covar}\end{split}
\end{align}
We now calculate each term on the r.h.s. of \eqref{eq:global_covar} separately. We start by writing the expressions in terms of $\Delta n$, as a shorthand for $n^\prime - n$. In the remainder of the section we assume $n^\prime > n$, without loss of generality.

The increment covariance at level $\ell$ $(\textrm{I})$ is calculated from a similar starting point to \eqref{eq:cov_coll_fine_start}:
\begin{align}
\E \left[ \sum_{m=0}^{M-1} \sum_{m^\prime=0}^{M-1} \Delta \U^{n,m}_{p,\ell} \Delta \U^{n^\prime,m^\prime}_{p,\ell} \right] &= \sum_{m=0}^{M-1} \sum_{m^\prime=0}^{M-1} \Delta t_\ell^2 \CV_{\Delta t_\ell}^2 \E\left[\VD^{n,m}_{p,\ell}\VD^{n^\prime,m^\prime}_{p,\ell}\right]\\
&= \sum_{m=0}^{M-1} \sum_{m^\prime=0}^{M-1} \Delta t_\ell^2 \CV_{\Delta t_\ell}^2 \pncf^{M\dns+m^\prime - m}.\label{eq:cov_subsequent_fine_fine_intermediate}
\end{align}
We now substitute the double summation in \eqref{eq:cov_subsequent_fine_fine_intermediate} with a two summations over $\dm = |m^\prime - m|$:
\begin{align}
\Delta t_\ell^2 \CV_{\Delta t_\ell}^2 & \left( \vphantom{\sum_{\dm=0}^{M-1} (M-\dm) \pncf^{M\dns-\dm}} \right. \underbrace{\sum_{\dm=0}^{M-1} (M-\dm) \pncf^{M\dns-\dm}}_{m^\prime \leq m} + \underbrace{\sum_{\dm=1}^{M-1} (M-\dm) \pncf^{M\dns+\dm}}_{m^\prime > m} \left. \vphantom{\sum_{\dm=0}^{M-1} (M-\dm) \pncf^{M\dns-\dm}} \right)\\
&= \Delta t_\ell^2 \CV_{\Delta t_\ell}^2  \left(M + \sum_{\dm=1}^{M-1} (M-\dm) \left( \pncf^{\dm} + \pncf^{-\dm} \right) \right)\pncf^{M\dns}.\label{eq:cov_subsequent_fine_fine_intermediate2}
\end{align}
We now again make use of the identity \eqref{eq:identity}, giving that \eqref{eq:cov_subsequent_fine_fine_intermediate2} equals
\begin{align}
& \Delta t_\ell^2 \CV_{\Delta t_\ell}^2 \!\! \left(\!\! M + \frac{\pncf}{(1-\pncf)^2}\!\left( \! M\!\left( 2 - \pncf \! -\pncf^{-1} \! \right) \! + \pncf^M \! + \pncf^{-M} \! -2 \right) \!\! \right)\pncf^{M\dns} \\
&\quad= \epsilon^2 \tilde{v}^2 \left(\pncf^{1-M} + \pncf^{M+1} - 2 \pncf \right) \pncf^{M\dns}.\label{eq:cov_subsequent_fine_fine}
\end{align}

Calculating the covariance of increments at level $\ell-1$ (II) is also straightforward:
\begin{align}
\E \left[ \Delta \U^n_{p,\ell-1} \Delta \U^{n^\prime}_{p,\ell-1} \right] = \Delta t_{\ell-1}^2\CV_{\Delta t_{\ell-1}}^2 \E\left[\VD^n_{p,\ell-1}\VD^{n^\prime}_{p,\ell-1}\right]= \Delta t_{\ell-1}^2\CV_{\Delta t_{\ell-1}}^2 \pncr^{\dns}.\retainlabel{eq:cov_subsequent_coarse_coarse}
\end{align}

The expression for the covariance between increments at level $\ell$ and level $\ell-1$ at differing time steps $n$ and $n'$ depends on the relative position of the increments, i.e., whether the increment at level $\ell$ comes before that at level $\ell-1$, or not. If the fine increment at level $\ell$ comes first (III), we need to calculate
\begin{equation}
\label{eq:covar_coarse_fine}
\E \left[ \sum_{m=0}^{M-1} \Delta \U^{n,m}_{p,\ell} \Delta \U^{n^\prime}_{p,\ell-1}  \right] = \Delta t_\ell \Delta t_{\ell-1} \CV_{\Delta t_\ell} \CV_{\Delta t_{\ell-1}} \sum_{m=0}^{M-1} \E\left[\VD^{n,m}_{p,\ell}\VD^{n^\prime}_{p,\ell-1}\right].
\end{equation}
To calculate the expected value of the r.h.s. of \eqref{eq:covar_coarse_fine}, we list the possible simulation behaviors at level $\ell-1$ in time step $n$, relative to the fine sub-step $m$:
\begin{itemize}
\item \textbf{No collision occurs in the simulation at level $\ell$.} This situation occurs with probability $\pncf^M.$ In this case, we have already established that no collision occurs at level $\ell-1$. The probability of the simulations being correlated at the end of time step $n$ is therefore equal to the probability of them being correlated at the start of time step $n$, given by \eqref{eq:correlation_probability_GT}.
\item \textbf{No collision occurs in sub-steps 0 through $m-1$, but at least one collision occurs in sub-steps $m$ through $M-1$.} This situation occurs with probability $$\pncf^{m} \left( 1 - \pncf^{M-m} \right) = \pncf^{m} - \pncf^{M}.$$ If the  simulation at level $\ell-1$ also has a collision, then $\VD^{n+1}_{p,\ell-1}$ and $\VD^{n,m}_{p,\ell}$ are independent by \eqref{eq:betacorr}. If no collision occurred in the simulation at level $\ell-1$, which is the case with probability $$\cfrac{\pncr-\pncf^M}{1-\pncf^M},$$ then $\VD^n_{p,\ell-1} = \VD^{n,m}_{p,\ell}$, if the trajectories were correlated at the beginning of time step $n$. This is the case with probability \eqref{eq:correlation_probability_GT}.
\item \textbf{At least one collision occurs in sub-steps 0 through $m-1$, but none occur in sub-steps $m$ through $M$.} This occurs with probability $$\left( 1 - \pncf^{m} \right) \pncf^{M-m} = \pncf^{M-m} - \pncf^M.$$ By \eqref{eq:betacorr} we know that $\VD^{n,*}_{p,\ell-1} = \VD^{n,m}_{p,\ell}$, if a collision also occurs in time step $n$ of the simulation at level $\ell-1$. This happens with probability \eqref{eq:correlation_probability_GT}. 
\item \textbf{Collisions happen both before sub-step $m$ and during or afterwards.} In this case, there is no correlation between $\VD^{n,*}_{p,\ell-1}$ and $\VD^{n,m}_{p,\ell}$.
\end{itemize}
By adding the non-zero contributions from these four cases, we calculate the probability that $\VD^{n,m}_{p,\ell}$ is correlated with $\VD^{n,*}_{p,\ell-1}$. To get the probability of the correlation holding until $\VD^{n^\prime}_{p,\ell-1}$, we multiply the sum by $\pncr^{\dns-1}$. Given the properties \eqref{eq:beta_properties}, we can state that the right hand side sum in \eqref{eq:covar_coarse_fine} is given by
\begin{align}
\sum_{m=0}^{M-1} \E&\left[\VD^{n,m}_{p,\ell}\VD^{n^\prime}_{p,\ell-1}\right]= \sum_{m=0}^{M-1} \left(\pncf^M + \left( \pncf^{m} \!\! - \! \pncf^{M} \right) \cfrac{\pncr-\pncf^M}{1-\pncf^M} \right.\\
&\phantom{\sum_{m=0}^{M-1} \E\left[\VD^{n,m}_{p,\ell}\VD^{n^\prime}_{p,\ell-1}\right]= \sum_{m=0}^{M-1}}\quad\left. \vphantom{\cfrac{\pncr-\pncf^M}{1-\pncf^M}} + \pncf^{M-m} - \pncf^M \right) \cfrac{1-\pncr}{1-\pncf^M} \pncr^{\dns-1}.\label{eq:cov_subsequent_fine_rough_GT_intermediate}
\end{align}
Making us of the fact that $\frac{\pcr}{\pncr}=\frac{\Delta t_{\ell-1}}{\epsilon^2}$ we can then write the r.h.s. of \eqref{eq:cov_subsequent_fine_rough_GT_intermediate} as
\begin{align}
\sum_{m=0}^{M-1} & \frac{\Delta t_{\ell-1}}{\epsilon^2} \! \left( \cfrac{\pncr-\pncf^M}{1-\pncf^M} \left(\pncf^{m} \!\! - \! \pncf^{M} \right) + \pncf^{M-m} \right) \!\! \cfrac{\pncr^{\dns}}{1-\pncf^M}\\
&= \frac{\Delta t_{\ell-1}}{\epsilon^2} \! \left( \cfrac{\pncr-\pncf^M}{1-\pncf^M} \left( \cfrac{1-\pncf^M}{1-\pncf} - M \pncf^{M} \right) - \cfrac{\pncf^{M}-1}{\pncf^{-1}-1} \right) \!\! \cfrac{\pncr^{\dns}}{1-\pncf^M} \mkern-60mu\\
&= \frac{\Delta t_{\ell-1}}{\epsilon^2} \left( \cfrac{ \pncr-\pncf^M }{ 1-\pncf^M } \left(\frac{1}{\pcf} - \cfrac{M\pncf^{M}}{ 1-\pncf^M} \right) + \frac{\epsilon^2}{\Delta t_\ell} \right) \pncr^{\dns}.\retainlabel{eq:cov_subsequent_fine_rough_GT}
\end{align}

To calculate the covariance of coarse time steps preceding fine time steps (IV), fewer calculations are needed. The two time steps are correlated if the trajectories were correlated at the start of coarse simulation time step $n$ and this correlation was not lost due to a collision in the fine simulation between time steps $(n,0)$ through $(n^\prime,m^\prime-1)$
\begin{align}
&\E \left[ \sum_{m=0}^{M-1} \Delta \U^n_{p,\ell-1} \Delta \U^{n^\prime,m^\prime}_{p,\ell} \right] = \Delta t_\ell \Delta t_{\ell-1} \CV_{\Delta t_\ell} \CV_{\Delta t_{\ell-1}} \sum_{m^\prime=0}^{M-1} \E\left[\VD_{p,\ell-1}^n\VD_{p,\ell}^{n^\prime,m^\prime}\right]\\
&= \Delta t_\ell \Delta t_{\ell-1} \CV_{\Delta t_\ell} \CV_{\Delta t_{\ell-1}} \sum_{m^\prime=0}^{M-1} \cfrac{1-\pncr}{1-\pncf^M} \pncf^{M\dns + m^\prime} \\
&= \Delta t_\ell \Delta t_{\ell-1} \CV_{\Delta t_\ell} \CV_{\Delta t_{\ell-1}} \cfrac{1-\pncr}{1-\pncf^M}  \cfrac{1-\pncf^M}{1-\pncf} \pncf^{M\dns} \\
&= \Delta t_{\ell-1}^2 \CV_{\Delta t_{\ell-1}}^2 \pncf^{M\dns}.\label{eq:cov_subsequent_rough_fine_GT}
\end{align}

Making use of expressions \eqref{eq:cov_subsequent_fine_fine} through \eqref{eq:cov_subsequent_rough_fine_GT}, we get
\begin{align}
\mkern-5mu\sum_{n=0}^{N-1} \sum_{n^\prime=n}^N\mkern-10mu & \mkern10mu \Cov(\Delta_{\U,n},\Delta_{\U,n^\prime}) \\
&= \sum_{n=0}^{N-1} \sum_{n^\prime=n}^N \covone \pncf^{M(\dn)} + \covtwo  \pncr^{\dn}\\
&= \sum_{\dns=1}^{N-1} (N-\dns) \left( \covone \pncf^{M\dns} + \covtwo \pncr^{\dns} \right)\\
&= \covone \pncf^M \!\! \cfrac{\pncf^{MN} \!\! - \! N\pncf^M \!\! + \! N \! - \! 1}{\left( 1 - \pncf^M \right)^2} + \covtwo  \pncr \!\! \frac{\pncr^{N} \!\!\! - \! N\pncr \!\! + \! N \! - \! 1}{\left( 1-\pncr \right)^2},\mkern-12mu\label{eq:final_covar}
\end{align}
$
\text{with} \qquad \covone =  \epsilon^2 \tilde{v}^2 \left( \pncf^{1-M} + \pncf^{M+1} - 2 \pncf \right) - \Delta t_{\ell-1}^2 \CV_{\Delta t_{\ell-1}}^2 \qquad \text{and}
$
\begin{equation}
\covtwo = \Delta t_{\ell-1}^2\CV_{\Delta t_{\ell-1}}^2 \!\! \left( \! 1 \! - \! \frac{\Delta t_\ell \CV_{\Delta t_\ell}}{\epsilon^2 \CV_{\Delta t_{\ell-1}}} \!\! \left( \cfrac{ \pncr-\pncf^M }{ 1-\pncf^M } \! \left(  \! \frac{1}{\pcf} - \cfrac{M\pncf^{M}}{ 1-\pncf^M} \right) \! + \frac{\epsilon^2}{\Delta t_\ell} \right) \!\! \right) \!\! .
\end{equation}

\section{Theoretical background to Section~\ref{sec:level_strategy}}
\label{sec:coarse_theoretical}
Based on Section 2.6 in~\cite{Giles2015} it can be shown that it is beneficial to leave out level 0 if
\begin{equation}
\label{eq:coarse_threshold}
\sqrt{C_0 \V \left[ F_0 \right]} + \sqrt{\left( C_0 + C_1 \right]) \V\left[ F_1-F_0 \right]} - \sqrt{C_1 \V \left[ F_1 \right]} > 0,
\end{equation}
with $C_\ell$ the cost of a sample at level $\ell$. Our analytical results from Section~\ref{sec:bias_variance} can only strictly be applied in \eqref{eq:coarse_threshold} if the quantity of interest is the expected particle position or velocity, as applying Lipschitz constants gives independent upper bounds for the positive and negative terms, hence we will also back up our claims with numerical experiments. Here, we only consider the particle position variance, as $\V\left[X_\ell-X_{\ell-1}\right]\gg \V\left[V_\ell-V_{\ell-1}\right]$ and $\V\left[X_\ell\right]\gg \V\left[V_\ell\right]$, for large $\Delta t_\ell$.

We first fix $\Delta t_1 = \epsilon^2$ and solve \eqref{eq:coarse_threshold} for positive real values of $M>1$ for which the inequality becomes an equality, taking different values of $\epsilon$ and $t^*$.  In Figure~\ref{fig:coarse_level_threshold} we show these unique, numerically computed $M$-values. It is clear from the range of the color bar that this $M$-value consistently exists and lies within the interval $[6,13]$, meaning that any sign change in the r.h.s. of \eqref{eq:coarse_threshold} lies in this range. Next, we evaluate the left hand side of \eqref{eq:coarse_threshold} for $M=6$ and $M=13$ and plot the results in Figure~\ref{fig:coarse_level_sign}. From this figure, it is clear that \eqref{eq:coarse_threshold} holds for small values of $M$, but no longer holds once the threshold value from Figure~\ref{fig:coarse_level_threshold} is passed, meaning that \eqref{eq:coarse_threshold} consistently holds once $M$ is sufficiently large. We can thus conclude that it makes sense, from a theoretical point of view, to add an extra coarse level, if it is sufficiently coarse in comparison with the level with $\Delta t = \epsilon^2$.

\setlength\oldfw{\figurewidth}
\setlength\oldfh{\figureheight}
\setlength\figurewidth{0.4\figurewidth}
\setlength\figureheight{0.8\figureheight}
\tikzexternalenable
\begin{figure}
\centering
\input{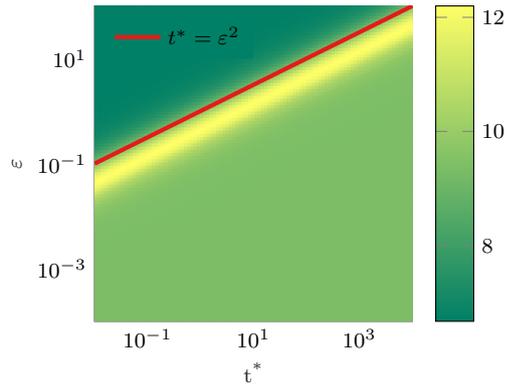}
\caption{$M$-value causing a sign change in the left hand side of \eqref{eq:coarse_threshold}.\label{fig:coarse_level_threshold}}
\end{figure}

\begin{figure}
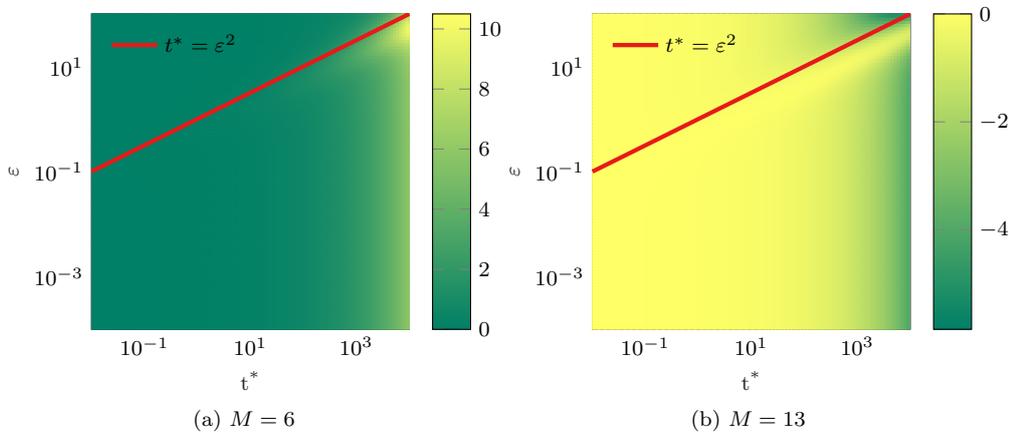

\makebox[\textwidth][c]{
\centering
\makebox[1.5\textwidth][c]{
\centering
\subfloat[$M=6$\label{fig:coarse_level_sign6}]{
\centering
\makebox[0.55\textwidth][c]{
\input{M6_cost_flat.tex}
}}
\subfloat[$M=13$\label{fig:coarse_level_sign13}]{
\centering
\makebox[0.55\textwidth][c]{
\input{M13_cost_flat.tex}
}}
}}
\caption{Left hand side of \eqref{eq:coarse_threshold}.\label{fig:coarse_level_sign}}
\end{figure}
\tikzexternaldisable
\setlength\figurewidth{\oldfw}
\setlength\figureheight{\oldfh}

\bibliographystyle{spmpsci}
\bibliography{APMLMC}



\section*{Supplementary material}

# Calculation of Maclaurin expansions for $\Delta t_\ell \to 0$ and limits for $\varepsilon \to 0$

## Emil Løvbak, Giovanni Samaey and Stefan Vandewalle

This file accompanies the paper "A multilevel Monte Carlo method for asymptotic-preserving particle schemes in the diffusive limit" by Emil Løvbak, Giovanni Samaey and Stefan Vandewalle and provides comments to supplement to the Matlab symbolic computations used to calculate the limits and Maclaurin series in Section 5 of the paper. The Matlab symbolic computations are also available in the form of a `.mlx` file, with the same name.

### Set up preliminaries

```
1  clear;
2  syms t_star M dt epsilon v_tilde;
3  assume(M>1);
4  assume(epsilon>0);
5  assume(t_star>0);
6  assume(dt>0);
7  assume(v_tilde>0);
8  N = t_star/dt/M;
```

### Calculate taylor series for transport increments (Equation (41))

```
9   pf = epsilon^2/(epsilon^2+dt);
10  pc = epsilon^2/(epsilon^2+M*dt);
11  Vf = epsilon/(epsilon^2+dt)*v_tilde;
12  Vc = epsilon/(epsilon^2+M*dt)*v_tilde;
13
14  V_sum_T = N*( M*dt^2*Vf^2 - M^2*dt^2*Vc^2 + 2*epsilon^2*Vf^2*(M*dt-(epsilon^2+dt)*(1-pf^M)) );
15  sig1 = epsilon^2*v_tilde^2*(pf^(1-M)+pf^(M+1)-2*pf)-M^2*dt^2*Vc^2;
16  sig21 = M^2*dt^2*Vc^2;
17  sig22 = M^2*dt^3*Vc*Vf/epsilon^2 * ((pc-pf^M)/(1-pf^M)*(1/(1-pf)-M*pf^M/(1-pf^M))+epsilon^2/dt);
18  sig2 = sig21-sig22;
19  coef1 = pf^M*(pf^(M*N)-N*pf^M+N-1)/(1-pf^M)^2;
20  coef2 = pc*(pc^N-N*pc+N-1)/(1-pc)^2;
21  cov_sum_T = sig1*coef1 + sig2*coef2;
22  total_var_T = V_sum_T + 2*cov_sum_T;
23  taylor_T_dt = simplify(taylor(total_var_T,dt,'ExpansionPoint',0,'Order',2))
```

$$\text{taylor\_T\_dt} = \frac{2\,dt\,\tilde{v}^2\,\mathrm{e}^{-\frac{t_{\text{star}}}{\varepsilon^2}}\,(M-1)\,\left(\varepsilon^2 - \varepsilon^2\,\mathrm{e}^{-\frac{t_{\text{star}}}{\varepsilon^2}} + t_{\text{star}}\,\mathrm{e}^{-\frac{t_{\text{star}}}{\varepsilon^2}}\right)}{\varepsilon^2}$$

### Calculate taylor series for Brownian increments (Equation (38))

```
24  Df = dt/(epsilon^2+dt)*v_tilde^2;
25  Dc = M*dt/(epsilon^2+M*dt)*v_tilde^2;
26
27  total_var_W = simplify(N*(2*M*dt*(Dc+Df)-4*M*dt*sqrt(Dc*Df)));
28  taylor_W_dt = simplify(taylor(total_var_W,dt,'ExpansionPoint',0,'Order',2))
```

$$\text{taylor\_W\_dt} = \frac{2\,dt\,t_{\text{star}}\,\tilde{v}^2\,\left(\sqrt{M}-1\right)^2}{\varepsilon^2}$$

Calculate taylor series of total position update

```
29  total_var = total_var_T + total_var_W;
30  taylor_total_dt = simplify(taylor(total_var,dt,'ExpansionPoint',0,'Order',2))
```

$$\text{taylor\_total\_dt} = \frac{2\,\mathrm{dt}\,\varepsilon^2\,\tilde{v}^2 + 4\,M\,\mathrm{dt}\,t_{\text{star}}\,\tilde{v}^2 - 2\,\mathrm{dt}\,\varepsilon^2\,\tilde{v}^2\,\mathrm{e}^{-\frac{t_{\text{star}}}{\varepsilon^2}} - 2\,M\,\mathrm{dt}\,\varepsilon^2\,\tilde{v}^2 - 4\,\sqrt{M}\,\mathrm{dt}\,t_{\text{star}}\,\tilde{v}^2 + 2\,M\,\mathrm{dt}\,\varepsilon^2\,\tilde{v}^2\,\mathrm{e}^{-\frac{t_{\text{star}}}{\varepsilon^2}}}{\varepsilon^2}$$

Calculate taylor series of difference of velocities

```
31  V_V = Vf^2−Vc^2;
32  taylor_V_V = taylor(V_V, dt, 'ExpansionPoint',0, 'Order',2)
```

$$\text{taylor\_V\_V} = -\mathrm{dt}\left(\frac{2\,\tilde{v}^2}{\varepsilon^4} - \frac{2\,M\,\tilde{v}^2}{\varepsilon^4}\right)$$

Calculate limit for epsilon equal to zero (Remarks 3 and 4)

```
33  limit_T_epsilon = limit(total_var_T,epsilon,0)
34  limit_V_sum_T_epsilon = limit(V_sum_T,epsilon,0)
35  limit_sig1_epsilon = limit(sig1,epsilon,0)
36  limit_sig21_epsilon = limit(sig21,epsilon,0)
37  limit_sig22_epsilon = limit(sig22,epsilon,0)
38  limit_coef1_epsilon = limit(coef1,epsilon,0)
39  limit_coef2_epsilon = limit(coef2,epsilon,0)
40  limit_T_epsilon = 0
41  limit_W_epsilon = simplify(limit(total_var_W,epsilon,0))
42  limit_total_epsilon = limit_T_epsilon + limit_W_epsilon
43  limit_V_V_epsilon = limit(V_V,epsilon,0)
```

$$\text{limit\_T\_epsilon} = \lim_{\varepsilon \to 0} \frac{2\,\phi_1{}^M \left(\varepsilon^2\,\tilde{v}^2 \left(\phi_1{}^{M+1} - \dfrac{2\,\varepsilon^2}{\varepsilon^2 + \mathrm{dt}} + \phi_1{}^{1-M}\right) - \phi_3\right)\left(\phi_1{}^{t_{\text{star}}/\mathrm{dt}} + \dfrac{t_{\text{star}}}{M\,\mathrm{dt}} - \dfrac{t_{\text{star}}\,\phi_1{}^M}{M\,\mathrm{dt}} - 1\right)}{\left(\phi_1{}^M - 1\right)^2}$$

$$+ \frac{t_{\text{star}}\left(\dfrac{2\,\varepsilon^4\,\tilde{v}^2\left(\left(\phi_1{}^M - 1\right)\left(\varepsilon^2 + \mathrm{dt}\right) + M\,\mathrm{dt}\right)}{\phi_2} + \dfrac{M\,\mathrm{dt}^2\,\varepsilon^2\,\tilde{v}^2}{\phi_2} - \phi_3\right)}{M\,\mathrm{dt}}$$

$$+ \frac{2\,\varepsilon^2 \left(\phi_3 - \dfrac{M^2\,\mathrm{dt}^3\,\tilde{v}^2\left(\dfrac{\varepsilon^2}{\mathrm{dt}} + \dfrac{\left(\dfrac{1}{\phi_1 - 1} - \dfrac{M\,\phi_1{}^M}{\phi_1{}^M - 1}\right)\left(\dfrac{\varepsilon^2}{\phi_4} - \phi_1{}^M\right)}{\phi_1{}^M - 1}\right)}{\phi_4\left(\varepsilon^2 + \mathrm{dt}\right)}\right)\left(\left(\dfrac{\varepsilon^2}{\phi_4}\right)^{\dfrac{t_{\text{star}}}{M\,\mathrm{dt}}} + \dfrac{t_{\text{star}}}{M\,\mathrm{dt}} - \dfrac{\varepsilon^2\,t_{\text{star}}}{M\,\mathrm{dt}\,\phi_4} - 1\right)}{\left(\dfrac{\varepsilon^2}{\phi_4} - 1\right)^2 \phi_4}$$

where

$$\phi_1 = \frac{\varepsilon^2}{\varepsilon^2 + \mathrm{dt}}$$

$$\phi_2 = \left(\varepsilon^2 + \mathrm{dt}\right)^2$$

$$\phi_3 = \frac{M^2\,\mathrm{dt}^2\,\varepsilon^2\,\tilde{v}^2}{\phi_4{}^2}$$

$$\phi_4 = \varepsilon^2 + M\,\mathrm{dt}$$

We observe that the Matlab symbolic toolbox fails in directly computing the limit $\varepsilon \to 0$. We now compute limits of sub-expressions of the full expression, helping the symbolic toolbox where it is needed.

First, we take the limit of the sum of increments of the form (40):

$$\text{limit\_V\_sum\_T\_epsilon} = 0.$$

Next, we take the limit of the terms $\sigma_1$ and $\sigma_2$ in (66):

$$\text{limit\_sig1\_epsilon} = \lim_{\varepsilon \to 0} \ \varepsilon^2\,\tilde{v}^2\left(\left(\frac{\varepsilon^2}{\varepsilon^2 + \mathrm{dt}}\right)^{M+1} - \frac{2\,\varepsilon^2}{\varepsilon^2 + \mathrm{dt}} + \left(\frac{\varepsilon^2}{\varepsilon^2 + \mathrm{dt}}\right)^{1-M}\right) - \frac{M^2\,\mathrm{dt}^2\,\varepsilon^2\,\tilde{v}^2}{\left(\varepsilon^2 + M\,\mathrm{dt}\right)^2},$$

$$\text{limit\_sig21\_epsilon} = 0,$$

$$\text{limit\_sig22\_epsilon} = M^2 \, \mathrm{dt}^3 \, \tilde{v}^2 \left( \lim_{\varepsilon \to 0} \frac{\dfrac{\varepsilon^2}{\mathrm{dt}} + \dfrac{\left( \dfrac{1}{\dfrac{\varepsilon^2}{\varepsilon^2 + \mathrm{dt}} - 1} - \dfrac{M \, \phi_1}{\phi_1 - 1} \right) \left( \dfrac{\varepsilon^2}{\varepsilon^2 + M \, \mathrm{dt}} - \phi_1 \right)}{\phi_1 - 1}}{(\varepsilon^2 + M \, \mathrm{dt}) \, (\varepsilon^2 + \mathrm{dt})} \right)$$

where

$$\phi_1 = \left( \frac{\varepsilon^2}{\varepsilon^2 + \mathrm{dt}} \right)^M.$$

We see that the symbolic toolbox fails in the computation of the limit of $\phi_1$ and the limit of part of $\phi_2$, however, given that both the time step and $M$ are fixed positive values. It is straightforward to see that these limits are in fact 0. All that is left is to check that the terms with which $\phi_1$ and $\phi_2$ are multiplied remain finite in the limit $\epsilon \to 0$:

$$\text{limit\_coef1\_epsilon} = 0,$$

$$\text{limit\_coef2\_epsilon} = 0.$$

As this is indeed the case, we can conclude that the limit in Remark 4 is indeed zero:

$$\text{limit\_T\_epsilon} = 0.$$

The limits in Remark 3 gives no issues:

$$\text{limit\_W\_epsilon} = 0.$$

This gives us:

$$\text{limit\_total\_epsilon} = 0.$$

Finally, the limit in Remark 5 also gives no issues:

$$\text{limit\_V\_V\_epsilon} = 0.$$



\end{document}